\documentclass[11pt,oneside]{amsart}

\usepackage{url}
\usepackage{amssymb,amsmath,amsthm,amsfonts}
\usepackage{fullpage}
\usepackage{mathtools}
\usepackage{graphicx}
\usepackage{xcolor}
\usepackage[export]{adjustbox}
\usepackage[margin=1in]{geometry}
\usepackage{tabu}
\usepackage{fancyhdr}
\usepackage{listings}
\usepackage{enumitem}
\usepackage{color}
\usepackage{multicol}
\usepackage{multirow}
\usepackage{ragged2e}
\usepackage{stmaryrd}
\usepackage{bm}
\usepackage{hyperref}
\usepackage{tikz}
\usepackage{verbatim}
\usepackage{pict2e}
\usepackage{nicefrac}
\usepackage{stackengine,scalerel}
\usepackage{tikz-cd}
\usepackage{mathrsfs}
\usepackage{subcaption}

\usepackage{graphicx}
\DeclareGraphicsExtensions{.pdf,.png,.jpg}
\newtheorem{theorem}{Theorem}[section]
\newtheorem*{thm1}{Theorem~\ref{dim4uniqueness}}
\newtheorem*{thm0}{Theorem~\ref{pseudo_existence_2}}
\newtheorem{cor}[theorem]{Corollary}
\newtheorem*{thm2}{Theorem~\ref{diagram_to_pseudotrisection}}
\newtheorem{prop}[theorem]{Proposition}
\theoremstyle{definition}
\newtheorem{defn}[theorem]{Definition}

\newtheorem{lem}[theorem]{Lemma}

\newtheorem{eg}[theorem]{Example}
\newtheorem{remark}[theorem]{Remark}

\DeclareMathOperator{\rank}{rank}

\DeclareMathOperator{\Hom}{Hom}

\newcommand{\R}{\mathbb{R}}

\newcommand{\CP}{\mathbb{CP}}

\newcommand{\Z}{\mathbb Z}

\definecolor{WongRe}{RGB}{213, 94, 0}
\definecolor{WongBl}{RGB}{0, 114, 178}
\definecolor{WongGr}{RGB}{0, 158, 115}
\definecolor{WongPu}{RGB}{204, 121, 167}
\definecolor{WongCy}{RGB}{86, 180, 233}
\definecolor{WongOr}{RGB}{230, 159, 0}

\begin{document}
\title{Pseudo-trisections of four-manifolds with boundary}
\author{Shintaro Fushida-Hardy}
\address{Department of Mathematics, Stanford University}
\email{sfh@stanford.edu}
\urladdr{\href{https://stanford.edu/~sfh/}{https://stanford.edu/~sfh/}}

\begin{abstract}We introduce the concept of pseudo-trisections of smooth oriented compact 4-manifolds with boundary. The main feature of pseudo-trisections is that they have lower complexity than relative trisections for given 4-manifolds. We prove existence and uniqueness of pseudo-trisections, and further establish a one-to-one correspondence between pseudo-trisections and their diagrammatic representations. We next introduce the concept of pseudo-bridge trisections of neatly embedded surfaces in smooth oriented compact 4-manifolds. We develop a diagrammatic theory of pseudo-bridge trisections and provide examples of computations of invariants of neatly embedded surfaces in 4-manifolds using said diagrams.
\end{abstract}
	\maketitle 

	\section{Introduction}\label{introduction_section}
	Trisections of oriented compact 4-manifolds were introduced by Gay and Kirby in \cite{GayKir}, who established existence and uniqueness in the closed setting. In the relative setting---that is, for oriented compact 4-manifolds with boundary---a relative trisection induces an open book decomposition on the boundary. Gay and Kirby show existence of relative trisections inducing any given open book on the boundary, and uniqueness (up to stabilisation) given a fixed open book on the boundary. This uniqueness result was improved in Castro's PhD thesis \cite{CAS1} by the introduction of \emph{relative stabilisations}. Castro-Islambouli-Miller-Tomova \cite{CasIslMilTom} obtained a full uniqueness result by introducing \emph{relative double twists}.

	In this paper we introduce \emph{pseudo-trisections}, an alternative framework for studying 4-manifolds with boundary. Rather than inducing open book decompositions on the boundary, they induce 3-manifold trisections in the sense of Koenig \cite{Koe}. We prove an analogous existence result to the aforementioned result of Gay-Kirby:
	\begin{theorem}\label{pseudo_existence_2}Let $X$ be a compact oriented smooth 4-manifold with non-empty connected boundary $Y$. Let $\tau$ be a trisection of $Y$. Then there is a pseudo-trisection of $X$ which restricts to $\tau$ on the boundary.
	\end{theorem}
	On the other hand, we define three types of stabilisation (internal, boundary, and Heegaard) for pseudo-trisections, two of which extend notions of stabilisation defined for 3-manifold trisections. We show uniqueness of pseudo-trisections up to these moves.
	\begin{theorem}\label{dim4uniqueness}
		Let $X$ be a compact oriented smooth 4-manifold with non-empty connected boundary. Any two pseudo-trisections of $X$ are equivalent up to internal, boundary, and Heegaard stabilisation.
	\end{theorem}
	The main feature of pseudo-trisections is that they are general enough to exist with lower \emph{complexity} than relative trisections. Complexity is an analogue of \emph{trisection genus} and is defined in Section \ref{4man_section}. Figure \ref{cplx_table} lists complexities of some simple 4-manifolds to demonstrate the discrepancy in complexity. Example \ref{cplx_big_gap} shows that $c(\natural^\ell (S^2\times D^2)) \leq \ell$ among pseudo-trisections, whereas we expect that $c(\natural^\ell (S^2\times D^2)) \geq 3\ell$ among relative trisections. (The latter would be true if relative trisection genus were additive under connected sum.)

	We also introduce \emph{pseudo-trisection diagrams} as a method for representing pseudo-trisections, analogously to trisection diagrams \cite{GayKir} and relative trisection diagrams \cite{GayCaiCas2}. In each of these settings, diagrams are shown to be in one-to-one correspondence with (relative) trisections when modding out by appropriate moves. We obtain an analogous result:

	\begin{theorem}\label{diagram_to_pseudotrisection}The \emph{realisation map} $$\mathcal R : \{\text{pseudo-trisection diagrams}\} \to \frac{\{\text{pseudo-trisections}\}}{\text{diffeomorphism}}$$ induces a bijection
	$$\frac{\Big\{\text{pseudo-trisection diagrams}\Big\}}{\begin{matrix}\text{band and torus stabilisation,}\\ \text{handleslide, isotopy}\end{matrix}} \longrightarrow \frac{\Big\{\begin{matrix}\text{compact oriented 4-manifolds}\\ \text{with one boundary component}\end{matrix}\Big\}}{\text{diffeomorphism}}.$$
	\end{theorem}

	Combined with the seeming lower complexity of pseudo-trisections, this results in a diagrammatic theory for 4-manifolds with boundary with lower complexity than relative trisection diagrams. Along the way, we introduce \emph{triple Heegaard diagrams} to represent trisections of closed 3-manifolds. A pseudo-trisection diagram of a pseudo-trisected 4-manifold has the convenient property of restricting to a triple Heegaard diagram of the induced trisection on the boundary 3-manifold. 

	In the final section of the article we extend our diagrammatic calculus from 4-manifolds with boundary to pairs $(X, \mathcal K)$ where $\mathcal K$ is a surface embedded in $X$ in \emph{pseudo-bridge position}. This generalises the theory of bridge trisections introduced by Meier-Zupan \cite{MeiZup2} in the closed setting, and by Meier \cite{Meier} in the relative setting. Pseudo-bridge trisections are defined to be compatible with pseudo-trisections, in much the same way that relative bridge trisections are compatible with relative trisections. Consequently we can define \emph{pseudo-shadow diagrams}, which consist of an underlying pseudo-trisection diagram to encode the ambient 4-manifold, and additional arcs to encode the embedded surface. 

	We finish by exploring several examples of pseudo-shadow diagrams, such as a M\"obius strip in $\CP^2 - B^4$, and a slice disk of a trefoil knot in $\CP^2 - B^4$. Through these examples we demonstrate how to compute some invariants from pseudo-shadow diagrams, namely Euler characteristic, orientability, and homology class.

	\subsection*{Conventions}\label{conventions}
	Hereafter $Y$ refers to connected closed smooth 3-manifolds, and $X$ to compact smooth oriented 4-manifolds with one boundary component. The boundary of $X$ is denoted by $Y$. The indices $i$ and $j$ are taken to be in $\Z/3\Z$. Bolded symbols such as $\boldsymbol{k}$ and $\boldsymbol{y}$ represent triples of three non-negative integers $k_1, k_2, k_3$, or $y_1, y_2, y_3$. We write $|\boldsymbol{k}|$ to mean the 1-norm, which in this case is just the sum of the $k_i$. We also write $\boldsymbol \delta$ and $\boldsymbol \alpha$ to represent certain \emph{collections} of loops. The loops themselves are unbolded, for example $\delta_1$ or $\alpha$.

We say $N \subset M$ is \emph{neatly embedded} if $\partial N \subset \partial M$, and $N$ is transverse to $\partial M$. Unless otherwise stated, any arcs in surfaces with boundary or 3-manifolds with boundary are assumed to be neatly embedded. We write $\mathcal K$ to denote a connected compact surface with boundary neatly embedded in $X$. The boundary of $\mathcal K$ is a link $L$ in $Y$. 
	\subsection*{Organisation of the paper}
	In section \ref{3man_section} we review 3-manifold trisections and introduce triple Heegaard diagrams. In section \ref{4man_section} we introduce pseudo-trisections, three notions of stabilisation, and prove Theorem \ref{dim4uniqueness}. Further, pseudo-trisection are compared with relative trisections. In section \ref{diagrams}
	we introduce pseudo-trisection diagrams, and prove Theorems \ref{pseudo_existence_2} and \ref{diagram_to_pseudotrisection}. In section \ref{bridge_section} we introduce pseudo-bridge trisections and pseudo-shadow diagrams, and study several examples.
	\subsection*{Acknowledgements}
	The author is deeply grateful to Ciprian Manolescu for his continued support across many meetings, and Maggie Miller for many helpful conversations. Further, the author thanks Cole Hugelmeyer, Gabriel Islambouli, and Peter Lambert-Cole for helpful discussions. This work was supported in part by NSF grant DMS-2003488.

	\section{Trisections of 3-manifolds}\label{3man_section}
In this section we summarise some definitions concerning trisections of 3-manifolds, introduced by Dale Koenig \cite{Koe}. In particular, we describe stabilisations and Koenig's result that (almost all) 3-manifold trisections are stably equivalent in this more general sense. Next, we introduce \emph{Heegaard stabilisations}, inspired by stabilisations of Heegaard splittings. We show that any two trisections for any 3-manifold are equivalent up to stabilisation and Heegaard stabilisation. Next, we introduce \emph{triple Heegaard diagrams}, an analogue of Heegaard diagrams for trisections of 3-manifolds. Finally we briefly discuss complexity of 3-manifold trisections.

\subsection{3-manifold trisections and stabilisation} 
\label{3manbasics}
In this subsection we review the theory of 3-manifold trisections introduced in \cite{Koe}.

\begin{defn}\cite{Koe} A $(\boldsymbol{y}, b)$-trisection of a 3-manifold $Y$ is a decomposition $Y = Y_1 \cup Y_2 \cup Y_3$ such that
	\begin{itemize}
		\item each $Y_i$ is a handlebody of genus $y_i$,
		\item each $\Sigma_{i} = Y_{i-1} \cap Y_i$ is a compact connected surface with some genus $p_i$ and boundary $B$, and
		\item $B = Y_1 \cap Y_2 \cap Y_3$ is a $b$-component link.
	\end{itemize}
	The link $B \subset Y$ is the \emph{binding} of the trisection. See Figure \ref{trisection_figure} for a schematic for how the pieces fit together.
\end{defn}
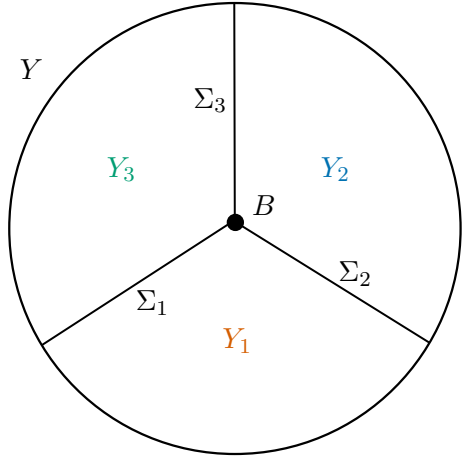
\begin{figure}	
	\noindent\hspace{20px}%% Creator: Inkscape 1.2.2 (1:1.2.2+202305151915+b0a8486541), www.inkscape.org
%% Accompanies image file 'fig_1.pdf' (pdf, eps, ps)
\begingroup%
  \makeatletter%
  \providecommand\color[2][]{%
    \errmessage{(Inkscape) Color is used for the text in Inkscape, but the package 'color.sty' is not loaded}%
    \renewcommand\color[2][]{}%
  }%
  \providecommand\transparent[1]{%
    \errmessage{(Inkscape) Transparency is used (non-zero) for the text in Inkscape, but the package 'transparent.sty' is not loaded}%
    \renewcommand\transparent[1]{}%
  }%
  \providecommand\rotatebox[2]{#2}%
  \newcommand*\fsize{\dimexpr\f@size pt\relax}%
  \newcommand*\lineheight[1]{\fontsize{\fsize}{#1\fsize}\selectfont}%
  \ifx\svgwidth\undefined%
    \setlength{\unitlength}{218.29151239bp}%
    \ifx\svgscale\undefined%
      \relax%
    \else%
      \setlength{\unitlength}{\unitlength * \real{\svgscale}}%
    \fi%
  \else%
    \setlength{\unitlength}{\svgwidth}%
  \fi%
  \global\let\svgwidth\undefined%
  \global\let\svgscale\undefined%
  \makeatother%
  \begin{picture}(1,0.78303152)%
    \lineheight{1}%
    \setlength\tabcolsep{0pt}%
    \put(0,0){\includegraphics[width=\unitlength,page=1]{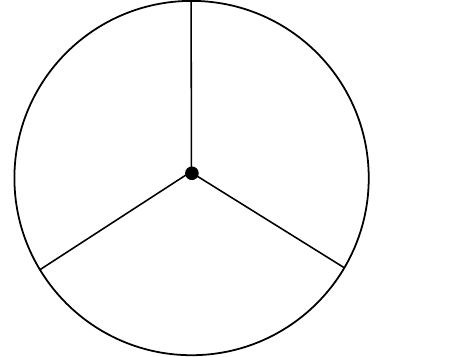}}%
    \put(0.05,0.65){\color[rgb]{0,0,0}\makebox(0,0)[lt]{\lineheight{1.25}\smash{\begin{tabular}[t]{l}$Y$\end{tabular}}}}%
    \put(0.4,0.18479974){\makebox(0,0)[lt]{\lineheight{1.25}\smash{\begin{tabular}[t]{l}\textcolor{WongRe}{$Y_1$}\end{tabular}}}}%
    \put(0.56938348,0.48074151){\makebox(0,0)[lt]{\lineheight{1.25}\smash{\begin{tabular}[t]{l}\textcolor{WongBl}{$Y_2$}\end{tabular}}}}%
    \put(0.20016589,0.48088031){\makebox(0,0)[lt]{\lineheight{1.25}\smash{\begin{tabular}[t]{l}\textcolor{WongGr}{$Y_3$}\end{tabular}}}}%
    \put(0.45117355,0.41553543){\makebox(0,0)[lt]{\lineheight{1.25}\smash{\begin{tabular}[t]{l}$B$\end{tabular}}}}%
    \put(0.6,0.3){\makebox(0,0)[lt]{\lineheight{1.25}\smash{\begin{tabular}[t]{l}$\Sigma_2$\end{tabular}}}}%
    \put(0.35,0.6){\makebox(0,0)[lt]{\lineheight{1.25}\smash{\begin{tabular}[t]{l}$\Sigma_3$\end{tabular}}}}%
    \put(0.25,0.25){\makebox(0,0)[lt]{\lineheight{1.25}\smash{\begin{tabular}[t]{l}$\Sigma_1$\end{tabular}}}}%
  \end{picture}%
\endgroup%

	\caption{The components of a trisection of a 3-manifold $Y$.}
\label{trisection_figure}
\end{figure}
\begin{prop}\label{3mantriprop}3-manifold trisections satisfy the following properties:
	\begin{enumerate}
		\item Each $y_i$ is given by $p_i + p_{i+1} + b - 1$. In particular, $y_i$ is bounded below by $b-1$.
		\item The genera $p_i$ of the surfaces $\Sigma_i$ are given by $\frac{1}{2}(y_{i-1} + y_i - y_{i+1} - b + 1)$.
	\end{enumerate}
\end{prop}
\begin{proof}The first fact follows from the fact that $Y_i$ is a handlebody with boundary $\Sigma_i \cup \Sigma_{i+1}$, where the surfaces are identified along their boundaries. The second fact is obtained by solving the simultaneous equations in the first fact.
\end{proof}

\begin{eg}The \emph{trivial trisection} is a decomposition of $S^3$ into three 3-balls, with any two balls meeting along a common disk. The triple intersection of these balls is an unknot.
\end{eg}
\begin{eg}\label{3admits}Every 3-manifold admits a trisection: one construction is to start with a Heegaard splitting $Y = H_1 \cup H_2$. Then consider a contractible loop $\gamma$ on $\Sigma = \partial H_1$. This bounds a disk $D$ in $\Sigma$, and a neatly embedded parallel disk $D'$ in $H_1$ (by isotoping $D$). The surface $D\cup D'$ bounds a ball $B$ in $H_1$. Now
	$$(H_1 - B, B, H_2)$$
	is a trisection of $Y$, with $B = \gamma$.

	Another recipe to construct trisections is to start with an open book decomposition. Details for this example (and further examples) are provided in \cite{Koe}. 
\end{eg}
\begin{eg}\label{connected_sum} There is a (non-unique) \emph{connected sum} of trisected 3-manifolds. Specifically one can delete a standard neighbourhood of a point in the binding of two trisections. Now each 3-manifold has spherical boundary, and the spheres have induced decompositions into three wedges. An (oriented) connected sum along these spheres respecting the decomposition into wedges produces a trisection of the connected sum of the 3-manifolds.
\end{eg}
\begin{defn}\label{3manstab} Given a trisection $(Y_1, Y_2, Y_3)$ of a 3-manifold $Y$, a \emph{stabilisation} is a new trisection $(Y_1', Y_2', Y_3')$ constructed as follows:
	\begin{enumerate}
		\item Choose a neatly embedded non-separating arc $\alpha$ in $\Sigma_i$ for some $i \in \{1,2,3\}$. (Such an arc exists provided $\Sigma_i$ is not a disk.)
		\item Let $N(\alpha)$ be a tubular neighbourhood of $\alpha$ in $Y$, and define
			\begin{itemize}
				\item $Y_{i+1}' = Y_{i+1} \cup \overline{N(\alpha)}$,
				\item $Y_{i}' = Y_{i} - N(\alpha)$,
				\item $Y_{i-1}' = Y_{i-1} - N(\alpha)$.
			\end{itemize}
	\end{enumerate}
\end{defn}

See Figure \ref{stab_figure} for a schematic of stabilisation, given a non-separating arc $\alpha$ in $\Sigma_1$.

\begin{figure}
	\noindent\hspace{80px}%% Creator: Inkscape 1.2.2 (1:1.2.2+202305151915+b0a8486541), www.inkscape.org
%% Accompanies image file 'fig_2.pdf' (pdf, eps, ps)
\begingroup%
  \makeatletter%
  \providecommand\color[2][]{%
    \errmessage{(Inkscape) Color is used for the text in Inkscape, but the package 'color.sty' is not loaded}%
    \renewcommand\color[2][]{}%
  }%
  \providecommand\transparent[1]{%
    \errmessage{(Inkscape) Transparency is used (non-zero) for the text in Inkscape, but the package 'transparent.sty' is not loaded}%
    \renewcommand\transparent[1]{}%
  }%
  \providecommand\rotatebox[2]{#2}%
  \newcommand*\fsize{\dimexpr\f@size pt\relax}%
  \newcommand*\lineheight[1]{\fontsize{\fsize}{#1\fsize}\selectfont}%
  \ifx\svgwidth\undefined%
    \setlength{\unitlength}{271.59929586bp}%
    \ifx\svgscale\undefined%
      \relax%
    \else%
      \setlength{\unitlength}{\unitlength * \real{\svgscale}}%
    \fi%
  \else%
    \setlength{\unitlength}{\svgwidth}%
  \fi%
  \global\let\svgwidth\undefined%
  \global\let\svgscale\undefined%
  \makeatother%
  \begin{picture}(1,0.69975261)%
    \lineheight{1}%
    \setlength\tabcolsep{0pt}%
    \put(0,0){\includegraphics[width=\unitlength,page=1]{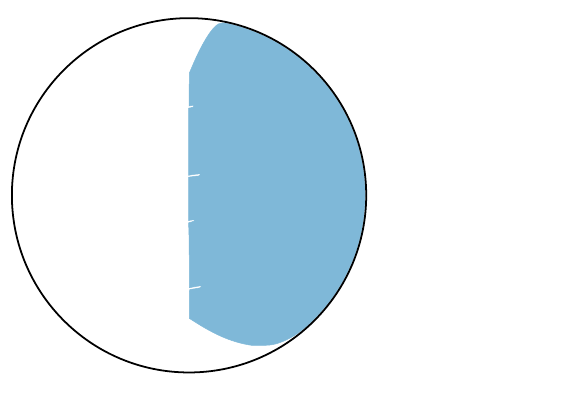}}%
    \put(0.21464433,0.00611252){\makebox(0,0)[lt]{\lineheight{1.25}\smash{\begin{tabular}[t]{l}\textcolor{WongRe}{$Y_1'$}\end{tabular}}}}%
    \put(0.04136719,0.27237469){\makebox(0,0)[lt]{\lineheight{1.25}\smash{\begin{tabular}[t]{l}$N(\alpha)$\end{tabular}}}}%
    \put(0.56157046,0.09091891){\makebox(0,0)[lt]{\lineheight{1.25}\smash{\begin{tabular}[t]{l}$\Sigma_2$\end{tabular}}}}%
    \put(0.39181825,0.67794891){\makebox(0,0)[lt]{\lineheight{1.25}\smash{\begin{tabular}[t]{l}$\Sigma_3$\end{tabular}}}}%
    \put(-0.00031641,0.17965489){\makebox(0,0)[lt]{\lineheight{1.25}\smash{\begin{tabular}[t]{l}$\Sigma_1$\end{tabular}}}}%
    \put(0,0){\includegraphics[width=\unitlength,page=2]{stab.pdf}}%
    \put(0.07101781,0.61048173){\makebox(0,0)[lt]{\lineheight{1.25}\smash{\begin{tabular}[t]{l}\textcolor{WongGr}{$Y_3'$}\end{tabular}}}}%
    \put(0.64827866,0.45256121){\makebox(0,0)[lt]{\lineheight{1.25}\smash{\begin{tabular}[t]{l}\textcolor{WongBl}{$Y_2'$}\end{tabular}}}}%
  \end{picture}%
\endgroup%

	\caption{A stabilisation of a trisection of a 3-manifold.}
	\label{stab_figure}
\end{figure}

	Notice that stabilisation increases $y_{i+1}$ by 1 while leaving $y_i$ and $y_{i-1}$ unchanged. On the other hand, $b$ increases by 1 if both endpoints of $\alpha$ are in the same component of $B$, and decreases by 1 if the endpoints are in different components. See \cite{Koe} for examples of stabilisations.

	\begin{remark} Provided $Y$ is not $S^3$, any trisection of $Y$ can be stabilised, since at least one $\Sigma_i$ must not be a disk. In Subsection \ref{heegaardstab} we introduce a second notion of stabilisation to deal with the $S^3$ case.
\end{remark}

We say that two trisections $\mathcal{T}$ and $\mathcal{T}'$ of $Y$ are \emph{isotopic} if there is a diffeomorphism from $Y$ to itself, isotopic to the identity, sending the $i$th sector of $\mathcal{T}$ to the $i$th sector of $\mathcal{T}'$. Futher, we say a trisection is a \emph{stabilisation} of another if it can be obtained by a finite sequence of stabilisations up to isotopy. Two trisections are said to be \emph{equivalent up to stabilisation} or \emph{stably equivalent} if there is a third trisection which is a stabilisation of each of the other two.

\begin{theorem}[Koenig \cite{Koe}]\label{dim3stab} Let $\mathcal{T}, \mathcal{T}'$ be two trisections of a closed orientable 3-manifold $Y$. If $Y$ is the 3-sphere, assume neither trisections are the trivial trisection. Then $\mathcal T$ and $\mathcal T'$ are equivalent up to stabilisation.
\end{theorem}

\subsection{Heegaard stabilisation}
In this subsection we introduce \emph{Heegaard stabilisation}. 
\label{heegaardstab}
	\begin{defn}Given a trisection $(Y_1, Y_2, Y_3)$ of a 3-manifold $Y$, a \emph{Heegaard stabilisation} is a new trisection $(Y_1', Y_2', Y_3')$ constructed as follows:
		\begin{enumerate}
			\item Choose a neatly embedded boundary parallel arc $\alpha$ in $Y_i$, with both endpoints on $\Sigma_i$.
			\item Let $N(\alpha)$ be a tubular neighbourhood of $\alpha$ in $Y$, and define
				\begin{itemize}
					\item $Y_{i+1}' = Y_{i+1}$,
					\item $Y_{i}' = Y_{i} - N(\alpha)$,
					\item $Y_{i-1}' = Y_{i-1} \cup \overline{N(\alpha)}$.
				\end{itemize}
		\end{enumerate}
\end{defn}

See Figure \ref{heeg_stab_figure} for a schematic of Heegaard stabilisation, given a non-separating arc $\alpha$ in $Y_1$ with endpoints on $\Sigma_1$. Requiring $\alpha$ to be boundary parallel ensures that each new sector is a 3-dimensional 1-handlebody.

Heegaard stabilisation increases $y_i$ and $y_{i-1}$ by 1 while leaving $y_{i+1}$ and $b$ unchanged.

\begin{figure}
	\noindent\hspace{75px}%% Creator: Inkscape 1.2.2 (1:1.2.2+202305151915+b0a8486541), www.inkscape.org
%% Accompanies image file 'fig_3.pdf' (pdf, eps, ps)
\begingroup%
  \makeatletter%
  \providecommand\color[2][]{%
    \errmessage{(Inkscape) Color is used for the text in Inkscape, but the package 'color.sty' is not loaded}%
    \renewcommand\color[2][]{}%
  }%
  \providecommand\transparent[1]{%
    \errmessage{(Inkscape) Transparency is used (non-zero) for the text in Inkscape, but the package 'transparent.sty' is not loaded}%
    \renewcommand\transparent[1]{}%
  }%
  \providecommand\rotatebox[2]{#2}%
  \newcommand*\fsize{\dimexpr\f@size pt\relax}%
  \newcommand*\lineheight[1]{\fontsize{\fsize}{#1\fsize}\selectfont}%
  \ifx\svgwidth\undefined%
    \setlength{\unitlength}{272.77526551bp}%
    \ifx\svgscale\undefined%
      \relax%
    \else%
      \setlength{\unitlength}{\unitlength * \real{\svgscale}}%
    \fi%
  \else%
    \setlength{\unitlength}{\svgwidth}%
  \fi%
  \global\let\svgwidth\undefined%
  \global\let\svgscale\undefined%
  \makeatother%
  \begin{picture}(1,0.69673587)%
    \lineheight{1}%
    \setlength\tabcolsep{0pt}%
    \put(0,0){\includegraphics[width=\unitlength,page=1]{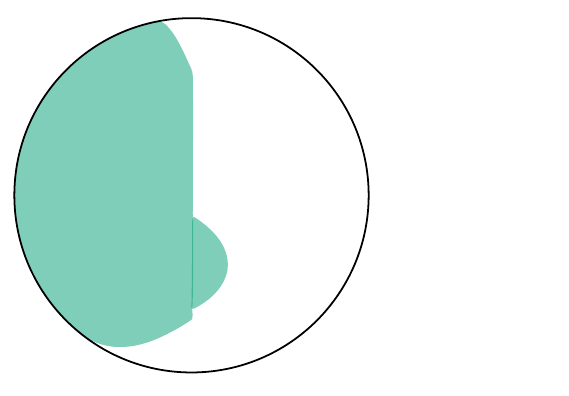}}%
    \put(0.21803018,0.00608616){\makebox(0,0)[lt]{\lineheight{1.25}\smash{\begin{tabular}[t]{l}\textcolor{WongRe}{$Y_1'$}\end{tabular}}}}%
    \put(0.40246165,0.28358598){\makebox(0,0)[lt]{\lineheight{1.25}\smash{\begin{tabular}[t]{l}$N(\alpha)$\end{tabular}}}}%
    \put(0.56346058,0.09052694){\makebox(0,0)[lt]{\lineheight{1.25}\smash{\begin{tabular}[t]{l}$\Sigma_2$\end{tabular}}}}%
    \put(0.39444028,0.67502616){\makebox(0,0)[lt]{\lineheight{1.25}\smash{\begin{tabular}[t]{l}$\Sigma_3$\end{tabular}}}}%
    \put(-0.00031505,0.18682198){\makebox(0,0)[lt]{\lineheight{1.25}\smash{\begin{tabular}[t]{l}$\Sigma_1$\end{tabular}}}}%
    \put(0,0){\includegraphics[width=\unitlength,page=2]{heeg_stab.pdf}}%
    \put(0.07183581,0.60754367){\makebox(0,0)[lt]{\lineheight{1.25}\smash{\begin{tabular}[t]{l}\textcolor{WongGr}{$Y_3'$}\end{tabular}}}}%
    \put(0.64979497,0.45061012){\makebox(0,0)[lt]{\lineheight{1.25}\smash{\begin{tabular}[t]{l}\textcolor{WongBl}{$Y_2'$}\end{tabular}}}}%
    \put(0,0){\includegraphics[width=\unitlength,page=3]{heeg_stab.pdf}}%
  \end{picture}%
\endgroup%

	\caption{A Heegaard stabilisation of a trisection of a 3-manifold.}
	\label{heeg_stab_figure}
\end{figure}

\begin{prop}\label{3stabunique} Any two trisections of a closed orientable 3-manifold are equivalent up to stabilisation and at least one Heegaard stabilisation.
\end{prop}
\begin{proof}If a trisection is trivial, Heegaard stabilising once produces a non-trivial trisection. This means that applying at most one Heegaard stabilisation, the trisections can be assumed to be non-trivial. The trisections are now equivalent up to stabilisation, by Theorem \ref{dim3stab}.
\end{proof}
\subsection{Triple Heegaard diagrams}
In this subsection we introduce \emph{triple Heegaard diagrams}, which are diagrams for 3-manifold trisections adapted from Heegaard diagrams. Given a $(\boldsymbol{y}, b)$-trisection of a 3-manifold, a triple Heegaard diagram will consist of $|\boldsymbol y|$ closed curves on surfaces, with each set of $y_i$ curves determining a handlebody. This is analagous to Heegaard diagrams requiring $g+g$ closed curves, with each set of $g$ curves determining a handlebody.

\begin{defn}A \emph{cut system} of curves on a closed surface $\Sigma$ of genus $g$ is a collection of $g$ disjoint simple closed curves that cut $\Sigma$ into a sphere with $2g$ boundary components.
\end{defn}

\begin{defn}A \emph{$(\boldsymbol{y},b)$-triple Heegaard diagram} consists of the data $(\Sigma_1,\Sigma_2,\Sigma_3, \boldsymbol\delta_{1}, \boldsymbol\delta_2,\boldsymbol\delta_3)$ such that:
	\begin{enumerate}
		\item The $\Sigma_i$ are surfaces with genus $p_i = \frac{1}{2}(y_{i-1} + y_i - y_{i+1} - b + 1)$ and $b$ boundary components.
		\item There is an identification of the boundaries of all three of the $\Sigma_i$. In particular, $\Sigma_i \cup \Sigma_{i+1}$ is a closed surface of genus $y_i$.
		\item Each $\boldsymbol\delta_{i}$ is a collection of disjoint neatly embedded arcs and simple closed curves on $\Sigma_i$ and $\Sigma_{i+1}$, which glues to a cut system of $y_i$ curves on $\Sigma_i \cup \Sigma_{i+1}$.
	\end{enumerate}
	In other words, a triple Heegaard diagram consists of three surfaces with boundary, and curves and arcs of two colours on each surface. For notational brevity, we frequently write $(\Sigma_i, \boldsymbol\delta_i)$ instead of $(\Sigma_1, \Sigma_2, \Sigma_3, \boldsymbol\delta_1,\boldsymbol\delta_2,\boldsymbol\delta_3)$. 
\end{defn}

\begin{prop}\label{3mantriuptodiffeo}Triple Heegaard diagrams determine trisected 3-manifolds up to diffeomorphism. In particular, a triple Heegaard diagram whose surfaces have genera $p_1, p_2, p_3$ and $b$ boundary components determines a $(\boldsymbol{y}, b)$-trisection with $y_i = p_i + p_{i+1} + b -1$.
\end{prop}
\begin{proof}For each $i$, the cut system on $\Sigma_i \cup \Sigma_{i+1}$ are instructions for gluing $p_i + p_{i+1} + b - 1$ disks into $\Sigma_i \cup \Sigma_{i+1}$. A 3-ball now glues into the remaining cavity to produce a handlebody $Y_i$ of genus $y_i = p_i + p_{i+1} + b - 1$. The handlebodies $Y_i$ and $Y_{i+1}$ glue along $\Sigma_{i+1}$. Gluing all three handlebodies together in this way produces a trisected closed 3-manifold. To see that the resulting trisected 3-manifold is unique up to diffeomorphism, we use that $\text{Diff}(B^n \text{ rel } \partial)$ is contractible for $n=2$ and $n=3$ \cite{Smale, Hatcher}. Suppose $Y$ and $Y'$ are built from a given triple Heegaard diagram as above. We will incrementally build a diffeomorphism $Y \to Y'$ preserving the trisection structure. First, consider the identity map on $\cup_i \Sigma_i$. The first steps in building $Y$ and $Y'$ was to glue disks along curves in the cut system. Since $\text{Diff}(B^2 \text{ rel } \partial)$ is contractible, the aforementioned map extends to $\cup_i \Sigma_i \cup \text{\{disks\}}$ independent of how these disks are glued in. The final step is to glue in three 3-balls, and again, the map extends, since $\text{Diff}(B^3 \text{ rel } \partial)$ is contractible. This guarantees that we can build a diffeormophism $Y \to Y'$, and moreover that the diffeomorphism preserves the trisection structure. That is, $Y$ and $Y'$ are diffeomorphic as trisected 3-manifolds.  
\end{proof}

\begin{eg}Figure \ref{diag_eg_figure} is an example of a triple Heegaard diagram of $T^3$. By noting identifications of some boundary components in $\Sigma_1$ and $\Sigma_2$, we see that all three of $\Sigma_1, \Sigma_2$ and $\Sigma_3$ share the same $S^1$ boundary. Inspecting the pairs $(\Sigma_i \cup \Sigma_{i+1}, \boldsymbol{\delta}_i)$, one can verify that these are cut systems on closed surfaces. 

	We observe that this is a diagram of $T^3$ by passing through the standard genus-3 Heegaard splitting of $T^3$. The standard Heegaard splitting of $T^3$ has diagram $(\Sigma_1 \cup \Sigma_2, \boldsymbol{\delta}_1, \boldsymbol{\delta}_2 \cup \boldsymbol{\delta}_3)$. To verify that the triple Heegaard diagram of Figure \ref{diag_eg_figure} builds the same 3-manifold, it suffices to observe that the cut systems $(\Sigma_3 \cup \Sigma_1, \boldsymbol{\delta}_3)$ and $(\Sigma_2 \cup \Sigma_3, \boldsymbol{\delta}_2)$ together determine the same handlebody as $(\Sigma_1 \cup \Sigma_2, \boldsymbol{\delta}_2 \cup \boldsymbol{\delta}_3)$. Indeed, the two former cut systems determine handlebodies with a common disk boundary $\Sigma_3$, and gluing the handlebodies along $\Sigma_3$ results in the handlebody determined by the latter cut system.
\end{eg}

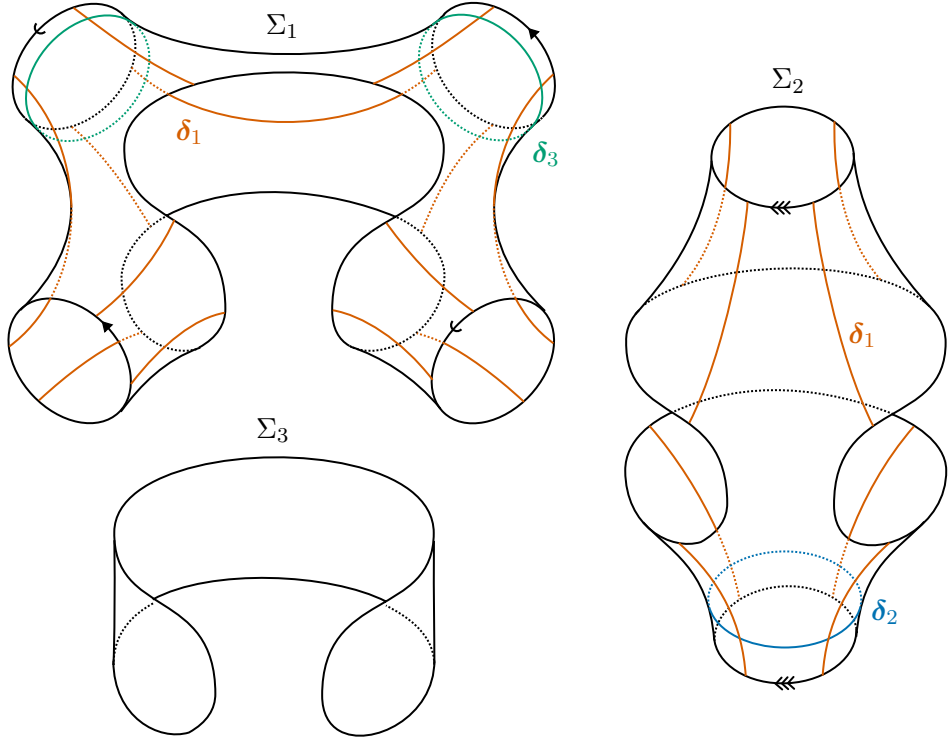
\begin{figure}
	\noindent\hspace{45px}%% Creator: Inkscape 1.2.2 (1:1.2.2+202305151915+b0a8486541), www.inkscape.org
%% Accompanies image file 'fig_4.pdf' (pdf, eps, ps)
\begingroup%
  \makeatletter%
  \providecommand\color[2][]{%
    \errmessage{(Inkscape) Color is used for the text in Inkscape, but the package 'color.sty' is not loaded}%
    \renewcommand\color[2][]{}%
  }%
  \providecommand\transparent[1]{%
    \errmessage{(Inkscape) Transparency is used (non-zero) for the text in Inkscape, but the package 'transparent.sty' is not loaded}%
    \renewcommand\transparent[1]{}%
  }%
  \providecommand\rotatebox[2]{#2}%
  \newcommand*\fsize{\dimexpr\f@size pt\relax}%
  \newcommand*\lineheight[1]{\fontsize{\fsize}{#1\fsize}\selectfont}%
  \ifx\svgwidth\undefined%
    \setlength{\unitlength}{559.92883358bp}%
    \ifx\svgscale\undefined%
      \relax%
    \else%
      \setlength{\unitlength}{\unitlength * \real{\svgscale}}%
    \fi%
  \else%
    \setlength{\unitlength}{\svgwidth}%
  \fi%
  \global\let\svgwidth\undefined%
  \global\let\svgscale\undefined%
  \makeatother%
  \begin{picture}(1,0.49494764)%
    \lineheight{1}%
    \setlength\tabcolsep{0pt}%
    \put(0,0){\includegraphics[width=\unitlength,page=1]{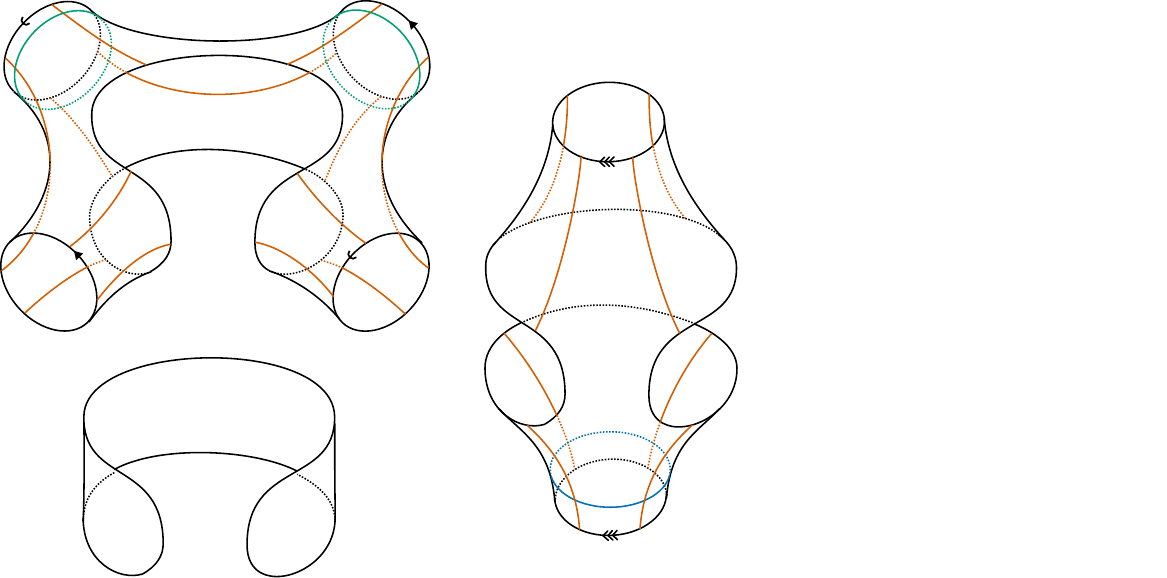}}%
    \put(0.1737792,0.47350043){\makebox(0,0)[lt]{\lineheight{1.25}\smash{\begin{tabular}[t]{l}$\Sigma_1$\end{tabular}}}}%
    \put(0.51424046,0.43674251){\makebox(0,0)[lt]{\lineheight{1.25}\smash{\begin{tabular}[t]{l}$\Sigma_2$\end{tabular}}}}%
    \put(0.16744612,0.20139158){\makebox(0,0)[lt]{\lineheight{1.25}\smash{\begin{tabular}[t]{l}$\Sigma_3$\end{tabular}}}}%
    \put(0.11324795,0.4020931){\makebox(0,0)[lt]{\lineheight{1.25}\smash{\begin{tabular}[t]{l}$\textcolor{WongRe}{\boldsymbol{\delta}_1}$\end{tabular}}}}%
    \put(0.35305099,0.38736531){\makebox(0,0)[lt]{\lineheight{1.25}\smash{\begin{tabular}[t]{l}$\textcolor{WongGr}{\boldsymbol{\delta}_3}$\end{tabular}}}}%
    \put(0.58133664,0.07873021){\makebox(0,0)[lt]{\lineheight{1.25}\smash{\begin{tabular}[t]{l}$\textcolor{WongBl}{\boldsymbol{\delta}_2}$\end{tabular}}}}%
    \put(0.56613575,0.265148){\makebox(0,0)[lt]{\lineheight{1.25}\smash{\begin{tabular}[t]{l}$\textcolor{WongRe}{\boldsymbol{\delta}_1}$\end{tabular}}}}%
  \end{picture}%
\endgroup%

	\caption{A triple Heegaard diagram of $T^3$.}
	\label{diag_eg_figure}
\end{figure}

\begin{remark}The data of a triple Heegaard diagram doesn't explictly describe how the binding $B$ is knotted---that is to say, the boundary $B = \partial \Sigma_i$ can be depicted in a diagram as an unlink. However, Proposition \ref{3mantriuptodiffeo} shows that the isotopy class of the link $B$ is nevertheless determined by triple Heegaard diagrams.
\end{remark}
\begin{prop}Every 3-manifold is described by a triple Heegaard diagram.
\end{prop}
\begin{proof}Every 3-manifold admits a trisection, as in Example \ref{3admits}. Moreover, every trisection is induced by a triple Heegaard diagram. This is because every sector of the trisection is a handlebody, and every handlebody is determined by a cut system on the boundary surface. A collection of three cut systems for each sector is precisely the data of a triple Heegaard diagram. 
\end{proof}

\begin{eg}\label{stabilisationdiagram} Figure \ref{stab_diag_figure} depicts stabilisation along a non-separating arc in $\Sigma_1$. Stabilising a diagram has three steps:
	\begin{enumerate}
		\item Identify the non-separating arc $\alpha$ in some $\Sigma_i$ along which the stabilisation occurs. Thicken the arc to a band in $\Sigma_i$. (The thickening is unique in the sense that the framing is determined by $\Sigma_i$.)
		\item The new surfaces $\Sigma_i', \Sigma_{i+1}', \Sigma_{i+2}'$ are obtained by subtracting the band in $\Sigma_i'$, and attaching the band to both $\Sigma_{i+1}'$ and $\Sigma_{i+2}'$. 
		\item The curves on $\Sigma_i$ intersecting $\alpha$ are cut and pasted in the new surfaces, as shown by the change from $\delta$ to $\delta'$ in Figure \ref{stab_diag_figure}. A single new curve is added to $\boldsymbol{\delta}_{i+1}$, consisting of an arc in each band in $\Sigma_{i+1}'$ and $\Sigma_{i+2}'$. This is $\gamma$ in Figure \ref{stab_diag_figure}, and corresponds to a meridian of the handle glued to $Y_{i+1}$ in the stabilisation.   
	\end{enumerate}
\end{eg}

\begin{figure}
	\noindent\hspace{40px}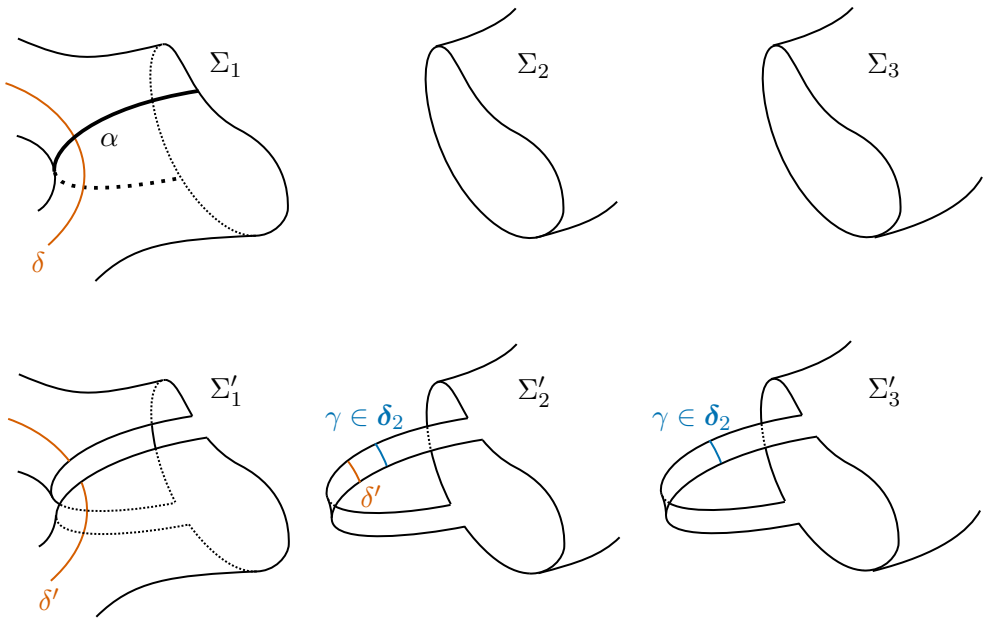
	\caption{Stabilisation of triple Heegaard diagrams.}
	\label{stab_diag_figure}
\end{figure}

\begin{eg}\label{heegaardstabilisationdiagram} Figure \ref{heeg_stab_diag_figure} depicts Heegaard stabilisation along an arc in $Y_1$ with endpoints on $\Sigma_1$. It is evidently diagrammatically similar to stabilisation of Heegaard splittings. In this example, since the arc has endpoints on $\Sigma_1$, the sector $Y_2$ is unchanged while $Y_1$ and $Y_3$ increase in genus. Therefore the two new curves $\eta$ and $\gamma$ in Figure \ref{heeg_stab_diag_figure} belong to $\boldsymbol{\delta}_1$ and $\boldsymbol{\delta}_3$ respectively. 
\end{eg}

\begin{figure}
	\noindent\hspace{45px}%% Creator: Inkscape 1.2.2 (1:1.2.2+202305151915+b0a8486541), www.inkscape.org
%% Accompanies image file 'heeg_stab_diag.pdf' (pdf, eps, ps)
\begingroup%
  \makeatletter%
  \providecommand\color[2][]{%
    \errmessage{(Inkscape) Color is used for the text in Inkscape, but the package 'color.sty' is not loaded}%
    \renewcommand\color[2][]{}%
  }%
  \providecommand\transparent[1]{%
    \errmessage{(Inkscape) Transparency is used (non-zero) for the text in Inkscape, but the package 'transparent.sty' is not loaded}%
    \renewcommand\transparent[1]{}%
  }%
  \providecommand\rotatebox[2]{#2}%
  \newcommand*\fsize{\dimexpr\f@size pt\relax}%
  \newcommand*\lineheight[1]{\fontsize{\fsize}{#1\fsize}\selectfont}%
  \ifx\svgwidth\undefined%
    \setlength{\unitlength}{608.4965846bp}%
    \ifx\svgscale\undefined%
      \relax%
    \else%
      \setlength{\unitlength}{\unitlength * \real{\svgscale}}%
    \fi%
  \else%
    \setlength{\unitlength}{\svgwidth}%
  \fi%
  \global\let\svgwidth\undefined%
  \global\let\svgscale\undefined%
  \makeatother%
  \begin{picture}(1,0.16755559)%
    \lineheight{1}%
    \setlength\tabcolsep{0pt}%
    \put(0,0){\includegraphics[width=\unitlength,page=1]{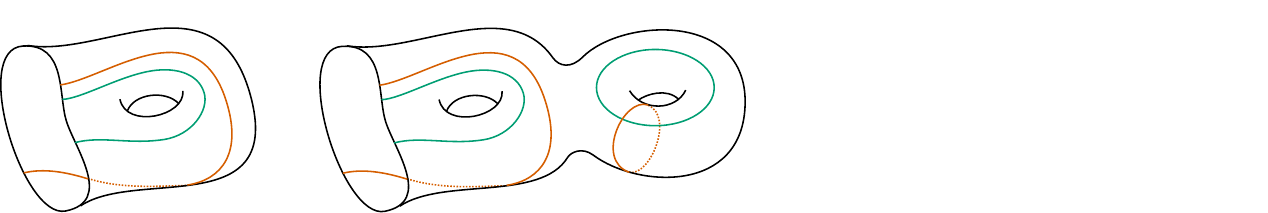}}%
    \put(0.081396,0.15095367){\makebox(0,0)[lt]{\lineheight{1.25}\smash{\begin{tabular}[t]{l}$\Sigma_1$\end{tabular}}}}%
    \put(0.35551235,0.15295764){\makebox(0,0)[lt]{\lineheight{1.25}\smash{\begin{tabular}[t]{l}$\Sigma_1'$\end{tabular}}}}%
    \put(0.50735756,0.15211737){\makebox(0,0)[lt]{\lineheight{1.25}\smash{\begin{tabular}[t]{l}$\textcolor{WongGr}{\gamma \in \boldsymbol{\delta}_3}$\end{tabular}}}}%
    \put(0.47818167,0.01010871){\makebox(0,0)[lt]{\lineheight{1.25}\smash{\begin{tabular}[t]{l}$\textcolor{WongRe}{\eta \in \boldsymbol{\delta}_1}$\end{tabular}}}}%
  \end{picture}%
\endgroup%

	\caption{Heegaard stabilisation of triple Heegaard diagrams.}
	\label{heeg_stab_diag_figure}
\end{figure}
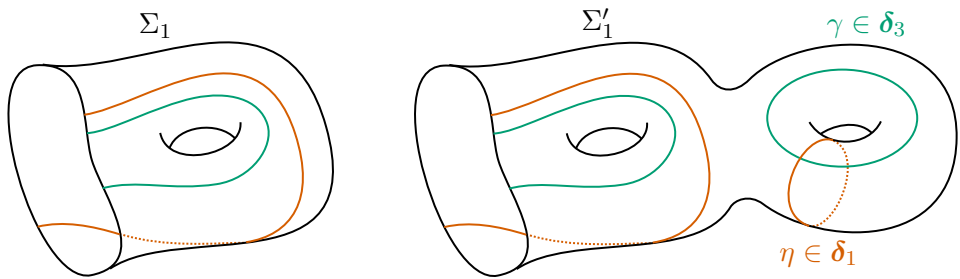

\begin{eg}\label{handleslide_eg} Handleslides and isotopies of triple Heegaard diagrams are the same as those for Heegaard diagrams. More explicitly, given a Heegaard diagram $(\Sigma, \boldsymbol{\alpha}, \boldsymbol{\beta})$, handleslide and isotopy moves are defined one handlebody at a time, i.e. on $(\Sigma, \boldsymbol{\alpha})$ and $(\Sigma, \boldsymbol{\beta})$. In the case of a triple Heegaard diagram $(\Sigma_i, \boldsymbol{\delta}_i)$, handleslides and isotopies are defined on the pairs $(\Sigma_1 \cup \Sigma_2, \boldsymbol{\delta}_2), (\Sigma_2 \cup \Sigma_3, \boldsymbol{\delta}_3)$, and $(\Sigma_3 \cup \Sigma_1, \boldsymbol{\delta}_1)$.

	Note that isotopies and handleslides of families of curves on $\Sigma_i \cup \Sigma_{i+1}$ can change how they intersect their common boundary, as shown in Figure \ref{handleslide_figure}.
\end{eg}

\begin{figure}
	\noindent\hspace{50px}%% Creator: Inkscape 1.2.2 (1:1.2.2+202305151915+b0a8486541), www.inkscape.org
%% Accompanies image file 'slide.pdf' (pdf, eps, ps)
\begingroup%
  \makeatletter%
  \providecommand\color[2][]{%
    \errmessage{(Inkscape) Color is used for the text in Inkscape, but the package 'color.sty' is not loaded}%
    \renewcommand\color[2][]{}%
  }%
  \providecommand\transparent[1]{%
    \errmessage{(Inkscape) Transparency is used (non-zero) for the text in Inkscape, but the package 'transparent.sty' is not loaded}%
    \renewcommand\transparent[1]{}%
  }%
  \providecommand\rotatebox[2]{#2}%
  \newcommand*\fsize{\dimexpr\f@size pt\relax}%
  \newcommand*\lineheight[1]{\fontsize{\fsize}{#1\fsize}\selectfont}%
  \ifx\svgwidth\undefined%
    \setlength{\unitlength}{510.7958048bp}%
    \ifx\svgscale\undefined%
      \relax%
    \else%
      \setlength{\unitlength}{\unitlength * \real{\svgscale}}%
    \fi%
  \else%
    \setlength{\unitlength}{\svgwidth}%
  \fi%
  \global\let\svgwidth\undefined%
  \global\let\svgscale\undefined%
  \makeatother%
  \begin{picture}(1,0.32478779)%
    \lineheight{1}%
    \setlength\tabcolsep{0pt}%
    \put(0,0){\includegraphics[width=\unitlength,page=1]{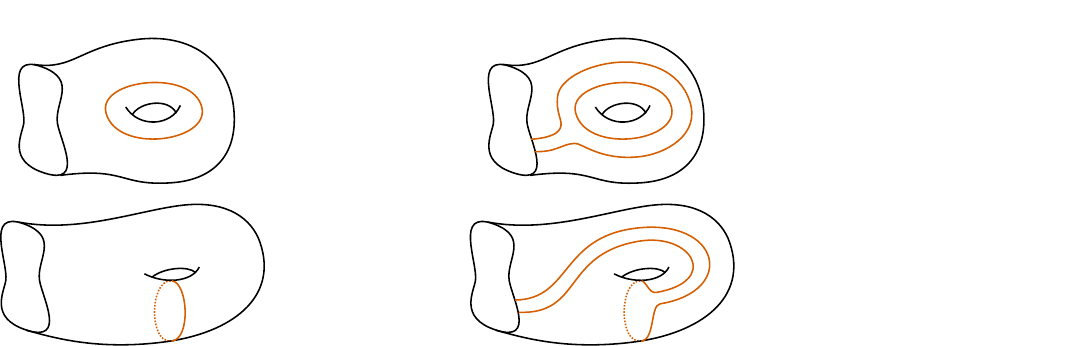}}%
    \put(0.29245413,0.15955466){\makebox(0,0)[lt]{\lineheight{1.25}\smash{\begin{tabular}[t]{l}Handleslide\end{tabular}}}}%
    \put(0,0){\includegraphics[width=\unitlength,page=2]{slide.pdf}}%
    \put(0.0626853,0.30739766){\makebox(0,0)[lt]{\lineheight{1.25}\smash{\begin{tabular}[t]{l}$(\Sigma_1, \Sigma_2, \boldsymbol{\delta}_1)$\end{tabular}}}}%
    \put(0.50486308,0.30733511){\makebox(0,0)[lt]{\lineheight{1.25}\smash{\begin{tabular}[t]{l}$(\Sigma_1, \Sigma_2, \boldsymbol{\delta}_1')$\end{tabular}}}}%
  \end{picture}%
\endgroup%

	\caption{A handleslide in the context of triple Heegaard diagrams.}
	\label{handleslide_figure}
\end{figure}

\begin{lem}\label{diagramuniqueness_lem}Fix a trisection $(Y_1, Y_2, Y_3)$ of a 3-manifold $Y$. Any two triple Heegaard diagrams $(\Sigma_i, \boldsymbol{\delta}_i), (\Sigma_i', \boldsymbol{\delta}_i')$ of $(Y_1, Y_2, Y_3)$ are diffeomorphic after a sequence of handleslides and isotopies. (That is, there are diffeomorphisms $\varphi_i : \Sigma_i \to \Sigma_i'$ rel boundary sending $\boldsymbol{\delta}_i \cap \Sigma_i$ and $\boldsymbol{\delta}_{i-1} \cap \Sigma_i$ to $\boldsymbol{\delta}_i' \cap \Sigma_i$ and $\boldsymbol{\delta}_{i-1}' \cap \Sigma_i$ respectively.) 
\end{lem}
\begin{proof}Johannson \cite{Joh} defines a \emph{meridian-system} to be a pair $(M,\mathcal D)$ where $M$ is a 3-manifold with boundary and $\mathcal D$ is a collection of neatly embedded disks such that $M - D$ is homeomorphic to a 3-ball. Corollary 1.6 of \cite{Joh} shows that meridian-systems of handlebodies (which are considered up to isotopy) are unique up to handleslides. Given a sector $Y_i$ of the trisection, $(\Sigma_{i}, \Sigma_{i+1}, \boldsymbol{\delta}_i)$ induces a unique meridian-system. It follows that, up to some diffeomorphism of $(Y_1, Y_2, Y_3)$, the triple Heegaard diagram is unique up to isotopy and handleslides. Finally, diffeomorphisms of the trisection descend to diffeomorphisms of triple Heegaard diagrams. 
\end{proof}

\begin{prop}\label{diagramuniqueness}Any two triple Heegaard diagrams of a given 3-manifold are diffeomorphic after a sequence of stabilisations, Heegaard stabilisations, handleslides, and isotopies.
\end{prop}
\begin{proof}By Lemma \ref{diagramuniqueness_lem}, any two triple Heegaard diagrams of a fixed trisection are diffeomorphic after a sequence of handleslides and isotopies. By Proposition \ref{3stabunique}, any two trisections of a given 3-manifold are equivalent up to stabilisation and Heegaard stabilisation. (Examples \ref{stabilisationdiagram} and \ref{heegaardstabilisationdiagram} show how stabilisation and Heegaard stabilisation descend to triple Heegaard diagrams.) The result follows by applying each of these results.
\end{proof}

\begin{remark}Proposition \ref{diagramuniqueness} is a trisection version of the \emph{Reidemeister-Singer theorem} \cite{Reidemeister, Singer} for Heegaard splittings, which states that Heegaard diagrams of a given 3-manifold are diffeomorphic after a sequence of stabilisations, handleslides, and isotopies. (Strictly speaking the \emph{Reidemeister-Singer theorem} refers to the result that Heegaard splittings are stably equivalent, but both original proofs were at the diagrammatic level.)
\end{remark}
\subsection{Complexity of 3-manifold trisections}
This subsection introduces a notion of complexity for 3-manifold trisections, which will be used later to study the complexity of \emph{pseudo-trisections}.

\begin{defn} The \emph{complexity} of a $(\boldsymbol{y}, b)$-trisection $\mathcal{T}$ of a 3-manifold $Y$ is the integer
	$c(\mathcal T) = |\boldsymbol y|$.
The minimum complexity among all trisections of $Y$ is denoted $c(Y)$.
\end{defn}
\begin{eg} The trivial trisection of $S^3$ has complexity $0$. It follows that $c(S^3) = 0$.
\end{eg}
\begin{prop}\label{3d_cplx_prop} Complexity of 3-manifold trisections satisfies the following properties:
	\begin{enumerate}
		\item Recalling that $p_i$ denotes the genus of the surface $\Sigma_i$ in the trisection $\mathcal T$,
			$$c(\mathcal T) = 2|\boldsymbol p| + 3b - 3.$$
		\item Stabilisation increases complexity by 1, Heegaard stabilisation increases complexity by 2.
		\item Writing $g_H(Y)$ to denote the minimum genus of a Heegaard splitting of $Y$,
			$$c(Y) \leq 2g_H(Y).$$
	\end{enumerate}
\end{prop}
\begin{proof} Item (1) is immediate from Proposition \ref{3mantriprop}. Item (2) follows from earlier observations about how stabilisation and Heegaard stabilisation affect each $y_i$. Finally item (3) follows from the fact that a genus $g$ Heegaard splitting induces a $(g, g, 0; 1)$-trisection.
\end{proof}

\begin{eg}[Koenig \cite{Koe}, Example 8] The trisection of $(S^1 \times S^2) \# (S^1 \times S^2)$ depicted in Figure \ref{2S1xS2_figure} has complexity $3$. One can show that any trisection with complexity at most $2$ is necessarily induced by a Heegaard splitting of genus at most 1, and such 3-manifolds are lens spaces or $(S^1 \times S^2)$. It follows that $c((S^1 \times S^2) \# (S^1 \times S^2)) = 3$. Moreover, $g_H((S^1\times S^2) \# (S^1 \times S^2)) = 2$, so this example shows that the inequality $c(Y) \leq 2g_H(Y)$ can be strict. 
\end{eg}
\begin{figure}
	\noindent\hspace{0px}%% Creator: Inkscape 1.2.2 (1:1.2.2+202305151915+b0a8486541), www.inkscape.org
%% Accompanies image file '2s1xs2.pdf' (pdf, eps, ps)
\begingroup%
  \makeatletter%
  \providecommand\color[2][]{%
    \errmessage{(Inkscape) Color is used for the text in Inkscape, but the package 'color.sty' is not loaded}%
    \renewcommand\color[2][]{}%
  }%
  \providecommand\transparent[1]{%
    \errmessage{(Inkscape) Transparency is used (non-zero) for the text in Inkscape, but the package 'transparent.sty' is not loaded}%
    \renewcommand\transparent[1]{}%
  }%
  \providecommand\rotatebox[2]{#2}%
  \newcommand*\fsize{\dimexpr\f@size pt\relax}%
  \newcommand*\lineheight[1]{\fontsize{\fsize}{#1\fsize}\selectfont}%
  \ifx\svgwidth\undefined%
    \setlength{\unitlength}{377.78225055bp}%
    \ifx\svgscale\undefined%
      \relax%
    \else%
      \setlength{\unitlength}{\unitlength * \real{\svgscale}}%
    \fi%
  \else%
    \setlength{\unitlength}{\svgwidth}%
  \fi%
  \global\let\svgwidth\undefined%
  \global\let\svgscale\undefined%
  \makeatother%
  \begin{picture}(1,0.192074)%
    \lineheight{1}%
    \setlength\tabcolsep{0pt}%
    \put(0,0){\includegraphics[width=\unitlength,page=1]{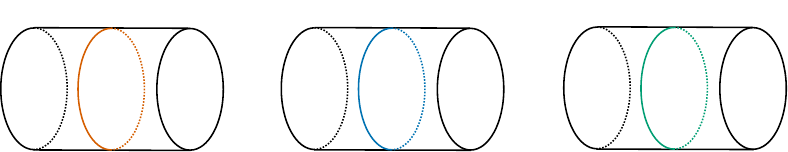}}%
    \put(0.12510406,0.17545703){\makebox(0,0)[lt]{\lineheight{1.25}\smash{\begin{tabular}[t]{l}$\Sigma_1$\end{tabular}}}}%
    \put(0.10759631,0.07613839){\makebox(0,0)[lt]{\lineheight{1.25}\smash{\begin{tabular}[t]{l}$\delta_1$\end{tabular}}}}%
    \put(0.48121576,0.17463136){\makebox(0,0)[lt]{\lineheight{1.25}\smash{\begin{tabular}[t]{l}$\Sigma_2$\end{tabular}}}}%
    \put(0.46370802,0.07531272){\makebox(0,0)[lt]{\lineheight{1.25}\smash{\begin{tabular}[t]{l}$\delta_2$\end{tabular}}}}%
    \put(0.8381701,0.17639859){\makebox(0,0)[lt]{\lineheight{1.25}\smash{\begin{tabular}[t]{l}$\Sigma_3$\end{tabular}}}}%
    \put(0.82066247,0.07707995){\makebox(0,0)[lt]{\lineheight{1.25}\smash{\begin{tabular}[t]{l}$\delta_3$\end{tabular}}}}%
  \end{picture}%
\endgroup%

	\caption{Triple Heegaard diagram of $(S^1\times S^2) \# (S^1 \times S^2)$ with complexity 3.}
	\label{2S1xS2_figure}
\end{figure}
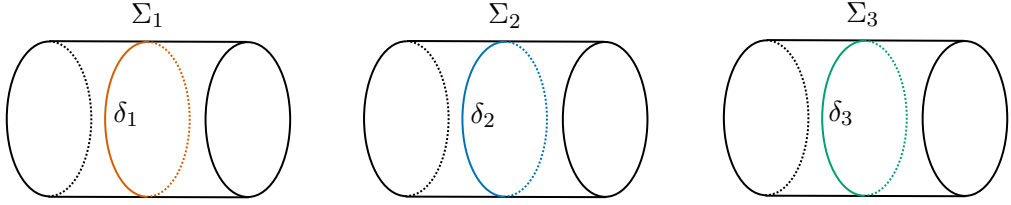

\section{Pseudo-trisections}\label{4man_section} In this section we introduce \emph{pseudo-trisections}, which are a generalisation of relative trisections of compact 4-manifolds with connected boundary. We also describe several notions of stabilisation, and prove that pseudo-trisections are stably equivalent. Finally we introduce \emph{complexity} of pseudo-trisections and use it to compare pseudo-trisections and relative trisections.

\subsection{Pseudo-trisections and their properties}
In this subsection we define pseudo-trisections and inspect some of their properties.

\begin{defn} An \emph{$n$-dimensional 1-handlebody of genus $k$} is a manifold with boundary diffeomorphic to $\natural^k (S^1 \times B^{n-1})$.
\end{defn}
\begin{defn}\label{ptri}A \emph{$(g, \boldsymbol{k}; \boldsymbol{y},b)$-pseudo-trisection} of a compact connected oriented smooth 4-manifold $X$ with one boundary component $Y = \partial X$ is a decomposition $X = X_1 \cup X_2 \cup X_3$ such that
	\begin{enumerate}
		\item each $X_i$ is a 4-dimensional 1-handlebody of genus $k_i$,
		\item each $H_i = X_{i-1} \cap X_i$ is a 3-dimensional 1-handlebody of genus $h_i$ for some $h_i$,
		\item each $Y_i = X_{i} \cap \partial X$ is a 3-dimensional 1-handlebody of genus $y_i$,
		\item $\Sigma_C = X_1 \cap X_2 \cap X_3$ is a surface of genus $g$ with $b$ boundary components,
		\item each $\Sigma_i = Y_{i-1} \cap Y_i$ is a surface of some genus $p_i$ and $b$ boundary components, and
		\item $B = X_1 \cap X_2 \cap X_3 \cap \partial X$ is a $b$-component link.
	\end{enumerate}
See Figure \ref{pseudo_schematic_figure} for a schematic of how the various pieces fit together.
\end{defn}

\begin{figure}	
	\noindent\hspace{0px}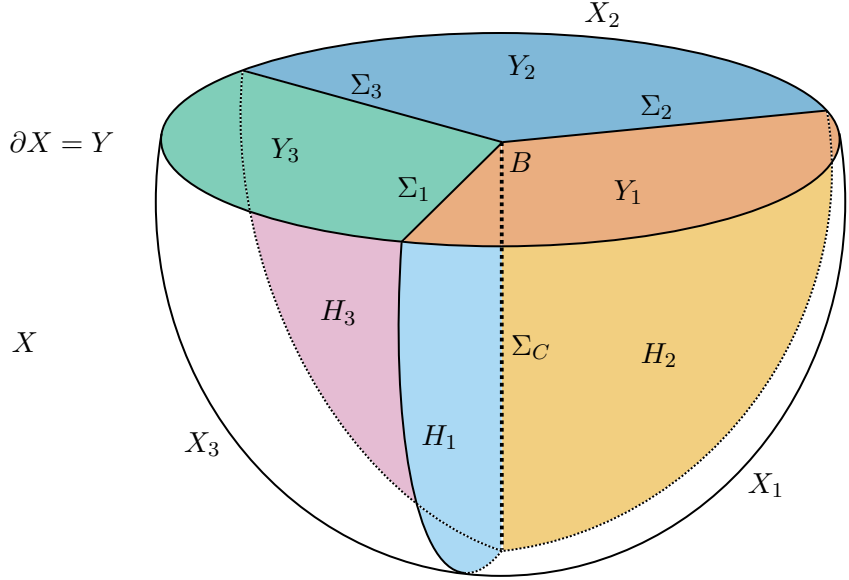
	\caption{The components of a pseudo-trisection of a 4-manifold $X$ with boundary $Y$.}
\label{pseudo_schematic_figure}
\end{figure}

\begin{defn}We fix some terminology: a \emph{sector} of a pseudo-trisection is any of the $X_i$. The \emph{binding} of a pseudo-trisection is the link $B$ (considered as a link in $\partial X$). Finally, the \emph{3-skeleton} of a pseudo-trisection is the union of all 3-dimensional pieces, $H_1 \cup H_2 \cup H_3 \cup \partial X$; and the \emph{2-skeleton} of a pseudo-trisection is the union of all 2-dimensional pieces, $\Sigma_C \cup \Sigma_1 \cup \Sigma_2 \cup \Sigma_3$. We avoid the term \emph{spine}, although it is commonly used to refer to the 3-skeleta of closed 4-manifold trisections.
\end{defn}
\begin{prop}\label{properties}Pseudo-trisections satisfy the following properties:
	\begin{enumerate}
		\item The restriction of a $(g, \boldsymbol{k}, \boldsymbol{y}, b)$-pseudo-trisection of $X$ to $\partial X = Y$ is a $(\boldsymbol{y}, b)$-trisection of $Y$.
		\item The genera $p_i$ of the surfaces $\Sigma_i$ are given by
			$$p_i = \tfrac{1}{2}(y_{i-1} + y_i - y_{i+1} - b + 1).$$
		\item The genera $h_i$ of the handlebodies $H_i$ are given by $$h_i = g+p_i + b - 1 = g + \tfrac{1}{2}(y_{i-1} + y_i - y_{i+1} + b - 1).$$
		\item $\chi(X) = g - |\boldsymbol{k}| + \frac{1}{2}(|\boldsymbol{y}| + b + 1)$.
	\end{enumerate}
\end{prop}
\begin{proof}The first fact is immediate from the definition - specifically, items (3), (5), and (6) are exactly necessary properties for $(Y_1, Y_2, Y_3)$ to be a trisection of $Y$. The second fact is from Proposition \ref{3mantriprop}, restated for completeness.
	The third fact follows from the fact that $H_i$ has boundary $\Sigma_i \cup \Sigma_C$.
Finally, the Euler characteristic formula is derived by repeated applications of the difference formula $\chi(A\cup B) = \chi(A) + \chi(B) - \chi(A\cap B)$. Note that $|\boldsymbol y| + b$ is odd because $|\boldsymbol y| + b = 2|\boldsymbol p| + 4b - 3$, so in each of the last three properties, no half-integers appear.
\end{proof}
\subsection{Relative trisections}
Pseudo-trisections are morally generalisations of \emph{relative trisections}, described in detail in \cite{CAS1}. We formalise the relationship between relative trisections and pseudo-trisections in Proposition \ref{relativeispseudo}. Here we review the definition of relative trisections and some fundamental results concerning relative trisections.

In order to define relative trisections, we must first fix notation for some standard model pieces---we use notation as in \cite{GayCaiCas}. A \emph{relative trisection} will then be defined in terms of these standard pieces.

Let $g,k,p,b$ be non-negative integers with $g \geq p$ and $g+p+b-1 \geq k \geq 2p+b-1$. Let $Z_k = \natural^k S^1 \times B^3$, and $Y_k = \partial Z_k = \#^k S^1 \times S^2$. Next we describe a decomposition of $Y_k$ into three pieces, two corresponding to intersections of sectors in the relative trisection, and the other corresponding to a third of the boundary of the manifold.

Let $D = \{(r, \theta): r \in [0,1], \theta \in [-\pi/3, \pi/3]\}$. Write
$$\partial D = \partial^- D \cup \partial^0 D \cup \partial^+ D,$$
where $\partial^- D$ and $\partial^+ D$ are the edges with $\theta = -\pi/3$ and $\theta = \pi/3$ respectively, and $\partial^0 D$ is the arc. Next, let $P$ be a surface of genus $p$ with $b$ boundary components, and define $U = D \times P$, with
$$\partial U = \partial^- U \cup \partial^0 U \cup \partial^+ U,\quad \partial^\pm U = \partial^\pm D \times P, \partial^0 U = (\partial^0 D \times P)\cup (D \times \partial P).$$
Let $V = \natural^{k-2p-b+1} S^1 \times B^3$. Notice that $\partial V$ has a standard Heegaard splitting. Let $\partial V = \partial^-V \cup \partial^+V$ be the splitting obtained by stabilising the standard Heegaard splitting exactly $g-k+p+b-1$ times.

Since $(k-2p-b+1) + (2p+b-1) = k$, there is an identification $Z_k = U \natural V$. In particular, the boundary connect sum can be taken so that the decompositions of the boundaries align, giving
$$Y^{\pm}_{g,k;p,b} = \partial^{\pm}U \natural \partial^{\pm}V,\quad Y^0_{g,k;p,b} = \partial^0U.$$

\begin{defn}\cite{CAS1} A \emph{$(g,k;p,b)$-trisection} or \emph{$(g,k;p,b)$-relative trisection} of a compact connected oriented 4-manifold $X$ with connected boundary is a decomposition $X_1 \cup X_2 \cup X_3$ such that:
	\begin{enumerate}
		\item there is a diffeomorphism $\varphi_i : X_i \to Z_k$ for each $i$, and
		\item $\varphi_i(X_i \cap X_{i+1}) = Y^-_{g,k;p,b}$, $\varphi_i(X_i \cap X_{i-1}) = Y^+_{g,k;,p,b}$, and $\varphi_i(X_i \cap \partial X) = Y^0_{g,k;p,b}$.
	\end{enumerate}
\end{defn}

\begin{remark}The indices used in the theory of relative trisections are $(g,k;p,b)$, as opposed to $(g,\boldsymbol{k}; \boldsymbol{y}, b)$ for pseudo-trisections. First note that relative trisections as introduced here are \emph{balanced}, so each of the $k_i$ and $p_i$ are independent of $i$. A more significant difference is that in pseudo-trisections we've chosen to work with $y_i$ rather than $p_i$, that is the genera of $Y_i$ rather than the genera of $\Sigma_i$. This is so that we can easily read off that the boundary is $(\boldsymbol{y}, b)$-trisected. We also know how to convert between indices by Proposition \ref{properties}.\end{remark}

\begin{prop}[Gay-Kirby \cite{GayKir}] A relative trisection of a 4-manifold restricts to an open book on the boundary.
\end{prop}
This proposition hints at a connection between relative trisections and pseudo-trisections, as we've shown that trisections of 3-manifolds are generalisations of open books, and pseudo-trisections restrict to 3-manifold trisections on their boundaries.

\begin{theorem}[Castro-Islambouli-Miller-Tomova \cite{CasIslMilTom}] Any two relative trisections of a given 4-manifold $X$ are equivalent up to interior stabilisations, relative stabilisations, and relative double twists.
\end{theorem}
\begin{proof}The reader is directed to \cite{CAS1} for a description of relative stabilisation, and a proof that any two relative trisections of 4-manifolds are equivalent up to internal and relative stabilisations, provided the induced open books on the boundary are equivalent up to Hopf stabilisation. Theorem 3.5 of \cite{PieZud} shows that all open books of a given 3-manifold are equivalent up to Hopf stabilisation and the $\partial U$ move (described in \cite{PieZud}). The $\partial U$ move is extended to the \emph{relative double twist} move of relative trisections in \cite{CasIslMilTom}, completing the proof.
\end{proof}
In Subsection \ref{stabsection} we establish the analogous result for pseudo-trisections.

\begin{prop}\cite{GayCaiCas} Every relative trisection is uniquely encoded by a \emph{relative trisection diagram}. (The diagrams are unique up to handleslides and diffeomorphism.)
\end{prop}
We do not elaborate on the definition of a relative trisection diagram, but the key takeaway is that relative trisections are encoded by diagrams of curves on a surface. We show analogously in Subsection \ref{diagramssection} that pseudo-trisections are uniquely encoded by \emph{pseudo-trisection diagrams}. 

\subsection{Examples of pseudo-trisections}
In this subsection we describe several examples of pseudo-trisections.
\begin{eg}\label{trivialB4}\emph{The trivial trisection of $B^4$.} Let $X_1, X_2, X_3$ each be 4-balls. The boundary of each $X_i$ further admits the trivial trisection (in the context of 3-manifolds). We denote these sectors of $\partial X_i$ by $Y_{i,j}$. Gluing $X_i$ to $X_{i+1}$ by identifying $Y_{i,2}$ and $Y_{i+1,3}$ produces $B^4$. Explicitly, the pseudo-trisection consists of:
	\begin{itemize}
		\item $X_1, X_2, X_3$ as described,
		\item $H_i = Y_{i-1,2} = Y_{i,3}$,
		\item $Y_i = Y_{i, 1}$,
	\end{itemize}
	and the lower dimensional pieces are obtained by appropriate intersections. Note that this is also a relative trisection.
\end{eg}

\begin{eg}\label{B3xS1}\emph{A pseudo-trisection of $B^3\times S^1$ which is not a relative trisection.} Let $X_1, X_2 = B^3 \times S^1$, and $X_3 = B^4$. $X_1$ and $X_2$ have boundary $S^2 \times S^1$. Consider the trisection of $S^2 \times S^1$ induced by its genus 1 Heegaard splitting, as well as the standard trisection of $S^3$. We denote these sectors by $Y_{i,j}$, where $Y_{1,1}, Y_{1,2}, Y_{2,1}, Y_{2,3}$ are solid tori, and $Y_{1,3}$ and $Y_{2,2}$ are solid balls. Let $Y_{3,j}$ be the sectors of the trivial trisection of the boundary of $X_3$.

	We claim that gluing $X_i$ to $X_{i+1}$ by identifying $Y_{i,2}$ and $Y_{i+1, 3}$ produces $B^3 \times S^1$. To see this, observe that $X_1$ and $X_2$ are glued along a copy of $D^2 \times S^1$ in their boundaries. This is a boundary connected sum in the first entry, so the resulting manifold is still $B^3 \times S^1$. Next, $X_3$ is glued to $X_1 \cup X_3$ along two 3-balls, but these two 3-balls also meet along a disk. The result is another boundary connected sum. It follows that the resulting manifold is $B^3 \times S^1$.

The sectors of the pseudo-trisection are all labelled consistently with Example \ref{trivialB4}. See Figure \ref{B3xS1_figure} for a schematic of the pseudo-trisection. Note that this pseudo-trisection is not a relative trisection because the boundary doesn't inherit an open book structure.
\end{eg}
\begin{figure}	
	\noindent\hspace{0px}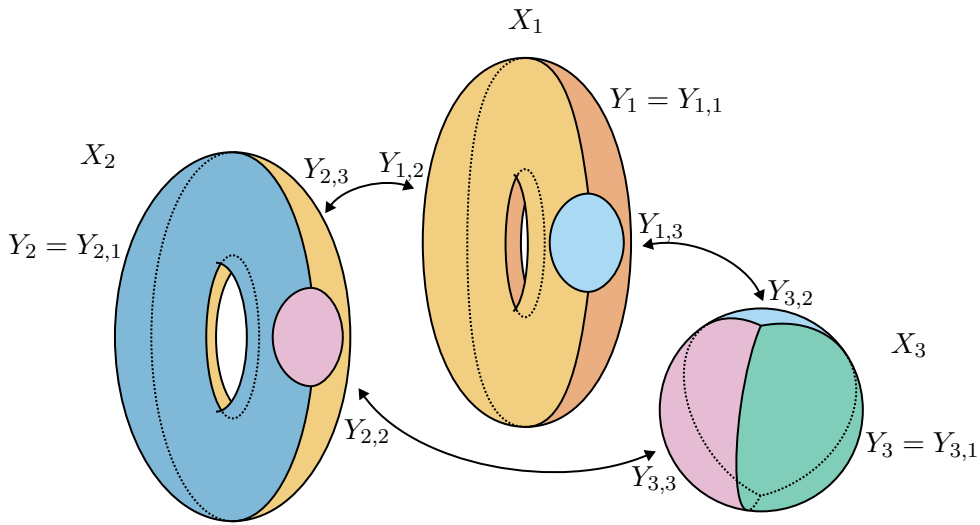
	\caption{A schematic of how the sectors are glued together in Example \ref{B3xS1}.}
\label{B3xS1_figure}
\end{figure}

\begin{eg}\label{B2xS2}\emph{A pseudo-trisection of $B^2 \times S^2$ which is not a relative trisection.} Let each of $X_1,X_2,X_3$ be a 4-ball. Writing $X_1$ and $X_2$ as $B^2 \times B^2$, the genus 1 Heegaard splittings of their boundaries can be written as $(S^1 \times B^2)\cup (B^2 \times S^1)$. These induce trisections
	$$\partial X_1 = (Y_{1,1}, Y_{1,2}, Y_{1,3}) = ((S^1\times B^2) - B, B^2 \times S^1, B),$$
	$$\partial X_2 = (Y_{2,1}, Y_{2,2}, Y_{2,3}) = ((S^1\times B^2) - B, B, B^2 \times S^1).$$
	In the above descriptions, $B$ is a ball removed from a sector of the Heegaard splitting to produce a trisection, as described in Example \ref{3admits}. Let the boundary of $X_3$ be a trivially trisected $S^3$ as in Example \ref{B3xS1}.

	Gluing the three sectors as in Examples \ref{trivialB4} and \ref{B3xS1}, we obtain a pseudo-trisection of $B^2 \times S^2$. To see this, notice that $X_1$ and $X_2$ are two copies of $B^2 \times B^2$ being glued along $B^2 \times S^1$, which produces $B^2 \times S^2$. As in the previous example, gluing $X_3$ does not alter the topology. Again, the sectors are labelled consistently with previous examples.
\end{eg}
\begin{eg}\label{CP2-B4}\emph{A pseudo-trisection of $\CP^2 - B^4$.} There is a standard genus 1 trisection $\mathcal T$ of $\CP^2$ as seen for example in \cite{GayKir}. A neighbourhood of a point in the central surface of $\mathcal T$ is a trivially trisected $B^4$. Removing this ball produces a relative trisection of $\CP^2 - B^4$ (which is also a pseudo-trisection). 
\end{eg}
\begin{eg}\label{bdy_connected_sum} Given pseudo-trisections $(X_1,\mathcal{T}_1)$ and $(X_2,\mathcal{T}_2)$, there are pseudo-trisections of the form $\mathcal{T}_1\natural\mathcal{T}_2$ on $X_1\natural X_2$. These are called \emph{boundary connected sums} of the pseudo-trisections, and restrict to a connected sum of the boundary trisections as described in Example \ref{connected_sum}. These are not unique---there are generally $3b_1 b_2$ distinct boundary connected sums, where $b_1$ and $b_2$ are the number of components of the bindings of $\mathcal{T}_1$ and $\mathcal{T}_2$ respectively.

Explicitly, we build the boundary connected sum as follows:
\begin{enumerate}
	\item Choose points $q_1$ and $q_2$ in the bindings $B_1$ and $B_2$ of $\mathcal{T}_1$ and $\mathcal{T}_2$ respectively. These each have standard neighbourhoods $N(q_1), N(q_2)$ in their respective pseudo-trisections, diffeomorphic to the trivial trisection. Let $M(q_1)$ and $M(q_2)$ denote the restriction of these neighbourhoods to $\partial X_1$ and $\partial X_2$. Note that we have made one of $b_1 b_2$ choices in specifying $q_1$ and $q_2$.
	\item The neighbourhoods $M(q_1)$ and $M(q_2)$ are trivially trisected in the sense that each is a 3-ball with an induced decomposition into three 3-balls, pairwise meeting along disks, with total-intersection an arc. These inherit labels $(Y_1^1, Y_2^1, Y_3^1)$ and $(Y_{1}^2, Y_2^2, Y_3^2)$ respectively. We take the boundary connected sum of $X_1$ and $X_2$ by identifying $M(q_1)$ with $M(q_2)$ while further respecting the decompositions of the balls into sectors.
	\item There are three choices we can make, since the identification can send $Y_1^1$ to any one of $Y_{i}^2$, and this determines the rest of the identification. In total, this means we have built one of $3b_1 b_2$ possible boundary connected sums.
\end{enumerate}
\end{eg}
\begin{remark}The \emph{boundary connected sum} of pseudo-trisections is distinct from the plumbing of relative trisections by Murasugi sum described in Theorem 3.20 of \cite{CasIslMilTom}. Specifically, a boundary connected sum of two relative trisections will generally not be a relative trisection.
\end{remark}
\begin{prop}\label{relativeispseudo} A relative trisection of a 4-manifold with one boundary component is a pseudo-trisection, provided the binding of the induced open book on the boundary is non-empty.
\end{prop}
\begin{proof} It must be shown that relative trisections as above satisfy all six of the defining properties of pseudo-trisections from Definition \ref{ptri}. Property 1 is immediate, since the definition of relative trisections requires that each sector $X_i$ is diffeomorphic to a standard piece $Z_k$. Properties 4, 5, and 6 are exactly \emph{observations 1, 2, and 3} at the start of section 4 of \cite{GayCaiCas2}. This leaves properties 2 and 3.

In a relative trisection, $H_i = X_i \cap X_{i-1}$ is necessarily diffeomorphic to a standard piece $Y^+_{g,k;p,b}$. This is a compression body from $\Sigma_{g,b}$ to $\Sigma_{p,b}$ obtained by compressing along $g-p$ simple closed curves. In the proposition statement we require that the binding of the induced open book in the boundary is non-empty, which is equivalent to requiring that $b$ is at least 1. A compression body of this form is diffeomorphic to a handlebody of genus $g+p+b-1$, verifying that property 2 is satisfied.

Finally we must verify that property 3 is satisfied. In the notation of \cite{GayCaiCas2}, we have that $X_i \cap \partial X$ is diffeomorphic to a standard piece $Y^0_{g,k;p,b}$, which is itself equal to $(D \times \partial \Sigma_{p,b}) \cup (I \times \Sigma_{p,b})$, where $D$ is a wedge of a disk and $I$ is the intersection of $D$ with the boundary of the underlying disk. Topologically the union corresponds to gluing $b$ solid tori to $I \times \Sigma_{p,b}$ along their longitudes, and thus $Y^0_{g,k;p,b}$ is diffeomorphic to $I \times \Sigma_{p,b}$. Again since $b \geq 1$, the latter is a handlebody (of genus $2p+b-1$).
\end{proof}
\begin{remark}The restriction to $(g,k,p,b)$-relative-trisections with $b\geq 1$ is not a strict condition. Firstly, any relative trisection can be \emph{Hopf-stabilised} \cite{CAS1} to increase $b$. Secondly, the usual recipe to construct a relative trisection is to first fix an open book on the boundary 3-manifold, and this open book can be chosen to have non-empty binding.
\end{remark}
\begin{cor}\label{pseudo_existence_1} Every compact connected oriented smooth 4-manifold with connected boundary admits a pseudo-trisection.
\end{cor}
\begin{proof} Every such 4-manifold admits a relative trisection \cite{CAS1}. In particular, the binding can be chosen to be non-empty. This relative trisection is a pseudo-trisection.
\end{proof}

\subsection{Stabilisations and uniqueness of pseudo-trisections}\label{stabsection}
We recall internal stabilisation, and introduce two more notions of stabilisations of pseudo-trisections supported near the boundary. We then show that any pseudo-trisections of a given compact 4-manifold with connected boundary are equivalent under these three notions of stabilisation.

In \cite{CAS1}, it is shown that \emph{relative trisections} of a given 4-manifold with boundary $X$ are stably equivalent, provided their boundaries are stably equivalent. More precisely, two notions of stabilisation are used: \emph{internal stabilisation}, which doesn't affect the open book on the boundary, and \emph{Hopf stabilisation} which stabilises the open book in the boundary by adding a Hopf band. However, not all open books of a given 3-manifold are equivalent under Hopf stabilisation.

In the context of pseudo-trisections, the boundary 3-manifold is equipped with a trisection rather than an open book, and all trisections of a given 3-manifold are related by appropriate notions of stabilisation. This allows us to define a set of moves so that for any given 4-manifold, all pseudo-trisections of the manifold are stably equivalent, without conditioning over the boundary. 

\begin{defn}Given a pseudo-trisection $(X_1, X_2, X_3)$ of a 4-manifold $X$, an \emph{internal stabilisation} is a new pseudo-trisection $(X_1', X_2', X_3')$ of $X$ constructed as follows:
	\begin{enumerate}
		\item Choose a neatly embedded boundary parallel arc $\alpha$ in $H_i = X_i \cap X_{i-1}$ for some $i$, with endpoints on $\Sigma_C = X_1 \cap X_2 \cap X_3$.
		\item Let $N(\alpha)$ be a tubular neighbourhood of $\alpha$ in $X$ supported away from $\partial X$, and define
			\begin{itemize}
				\item $X_{i+1}' = X_{i+1} \cup \overline{N(\alpha)}$,
				\item $X_{i}' = X_i - N(\alpha)$,
				\item $X_{i-1}' = X_{i-1} - N(\alpha)$.
			\end{itemize}
	\end{enumerate}
\end{defn}
See Figure \ref{pseudo_stabilisations_figure} (left) for a schematic of internal stabilisation.

Notice that stabilisation increases the handlebody genus $k_{i+1}$ of $X_{i+1}$ by 1, while leaving $k_i$ and $k_{i-1}$ unchanged. The genus $g$ of $\Sigma_C$ also increases by 1. This move is essentially `one third' of the usual notion of stabilisation for balanced trisections of closed 4-manifolds.

\begin{defn}Given a pseudo-trisection $(X_1, X_2, X_3)$ of a 4-manifold $X$ with boundary $Y$, a \emph{boundary stabilisation} is a new pseudo-trisection $(X_1', X_2', X_3')$ of $X$ constructed as follows:
	\begin{enumerate}
		\item Choose a neatly embedded non-separating arc $\alpha$ in $\Sigma_i$. (Such an arc exists provided $\Sigma_i$ is not a disk.)
		\item Let $N(\alpha)$ be a tubular neighbourhood of $\alpha$ in $X$, and define
			\begin{itemize}
				\item $X_{i+1}' = X_{i+1} \cup \overline{N(\alpha)}$,
				\item $X_{i}' = X_i - N(\alpha)$,
				\item $X_{i-1}' = X_{i-1} - N(\alpha)$.
			\end{itemize}
	\end{enumerate}
\end{defn}
See Figure \ref{pseudo_stabilisations_figure} (middle) for a schematic of boundary stabilisation. 

The definition looks very similar to that of internal stabilisation---one way to differentiate them is that \emph{internal stabilisation} is defined with an arc $\alpha$ in $H_i$ that is far from $\Sigma_i = H_i \cap \partial X$, while \emph{boundary stabilisation} is defined with an arc $\alpha$ that lives entirely in $\partial X$.

We make the following observations:
\begin{enumerate}
	\item The restriction of $N(\alpha)$ to $Y$ is a tubular neighbhourhood of $\alpha$ in $Y$. Therefore the restriction of $(X_1', X_2', X_3')$ to $Y$ is a stabilisation (in the sense of Definition \ref{3manstab}) of the restriction of $(X_1, X_2, X_3)$ to $Y$. (More succinctly, the restriction of a boundary stabilisation is a 3-manifold stabilisation of the restriction.) In particular, if $(X_1, X_2, X_3)$ was a $(g, \boldsymbol k; \boldsymbol y, b)$-pseudo-trisection of $X$, then in the stabilisation $\boldsymbol y$ and $b$ change as described in Subsection \ref{3manbasics}.
	\item The genera $\boldsymbol k$ of the 4-dimensional sectors change in the same way as $\boldsymbol y$. That is, $k_{i+1}$ increases by 1, while $k_i$ and $k_{i-1}$ are unchanged.
	\item The genus $g$ of $\Sigma_C$ increases by 1 if the endpoints of $\alpha$ lie in different components of $B$, and is unchanged if the endpoints of $\alpha$ lie in the same component. This can be seen directly, as boundary stabilisation descends to a band attachment on $\Sigma_C$, but it can also be deduced using the Euler characteristic formula for pseudo-trisections in Proposition \ref{properties}.  
\end{enumerate}

\begin{defn} Given a pseudo-trisection $(X_1, X_2, X_3)$ of a 4-manifold $X$ with boundary $Y$, a \emph{Heegaard stabilisation} is a new pseudo-trisection $(X_1', X_2', X_3')$ of $X$ constructed as follows:
	\begin{enumerate}
		\item Choose a neatly embedded boundary parallel arc $\alpha$ in $Y_i = X_i \cap \partial X$ for some $i$, with both endpoints on $\Sigma_i$.
		\item Let $N(\alpha)$ be a tubular neighbourhood of $\alpha$ in $X_i$, and define
			\begin{itemize}
				\item $X_{i+1}' = X_{i+1}$,
				\item $X_i' = X_i - N(\alpha)$,
				\item $X_{i-1}' = X_{i-1} \cup \overline{N(\alpha)}$.
			\end{itemize}
	\end{enumerate}
\end{defn}
See Figure \ref{pseudo_stabilisations_figure} (right) for a schematic of Heegaard stabilisation.

We make the following observations:
\begin{enumerate}
	\item Analogously to boundary stabilisations restricting to stabilisations of the boundary, a \emph{Heegaard stabilisation} of a pseudo-trisection restricts to a \emph{Heegaard stabilisation} of the 3-manifold trisection of the boundary $Y$. In particular, a Heegaard stabilisation of a pseudo-trisection changes $y_i$ and $y_i-1$ by 1, while leaving $y_{i+1}$ and $b$ unchanged.
	\item The sector $X_{i-1}$ gains a 1-handle in the stabilisation, so $k_{i-1}$ increases by 1. On the other hand, $N(\alpha)$ being carved out of $X_{i}$ does not change the topology of $X_{i}$, since $N(\alpha)$ runs along the boundary of $X_{i}$. Therefore $k_{i}$ and $k_{i+1}$ are unchanged.
	\item The central surface $\Sigma_C$ is unchanged, so $b$ is unchanged. One can verify that all of these index changes are compatible with the Euler characteristic formula.
\end{enumerate}

\begin{figure}	
	\noindent\hspace{0px}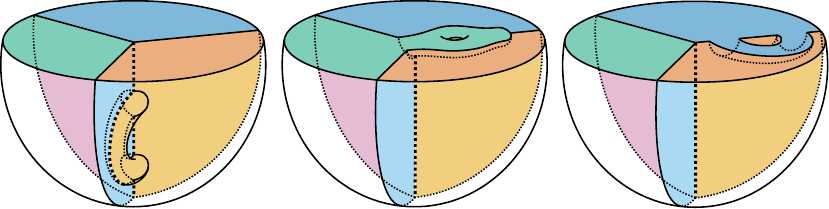
	\caption{From left to right: schematics of internal stabilisation, boundary stabilisation, and Heegaard stabilisation.}
\label{pseudo_stabilisations_figure}
\end{figure}

Internal stabilisation and Heegaard stabilisation are unified into the single notion of \emph{torus stabilisation} at the diagrammatic level, introduced later in Definition \ref{torus_stab_defn}.

\begin{thm1}Let $X$ be a compact connected oriented smooth 4-manifold with one boundary component. Any two pseudo-trisections of $X$ are equivalent up to internal, boundary, and Heegaard stabilisation.
\end{thm1}
\begin{proof}There are two steps: first we stabilise the pseudo trisections to agree on the boundary, and then stabilise the interior.

	Let $\mathcal T, \mathcal T'$ be pseudo-trisections of $X$. Their restrictions $\tau, \tau'$ to $Y = \partial X$ are trisections of $Y$. By Proposition \ref{3stabunique}, $\tau$ and $\tau'$ are stably equivalent. This means there is a trisection $\sigma$ of $Y$ and arcs $\{\alpha\}, \{\alpha'\}$ such that stabilisations and Heegaard stabilisations of $\tau$ along $\{\alpha\}$ produces $\sigma$, and similarly stabilisations and Heegaard stabilisations of $\tau'$ along $\{\alpha'\}$ produces $\sigma$.

	We now return to 4-dimensions: Boundary stabilising and Heegaard stabilising $\mathcal T$ along $\{\alpha\}$ produces a pseudo-trisection $\mathcal S$ which restricts to the 3-manifold trisection $\sigma$ on $Y$. On the other hand, boundary and Heegaard stabilising $\mathcal T'$ along $\{\alpha'\}$ produces a pseudo-trisection $\mathcal S'$ which also restricts to $\sigma$ on $Y$. We have successfully stabilised our pseudo-trisections to agree on the boundary.

	Next we turn to the interior. We proceed as in the proof of the analogous result for closed trisections in \cite{GayKir}. To that end, we require a cobordism structure on $X$, which we obtain by first turning our boundary trisection into a Heegaard splitting.

	Consider the sectors $X_1$ and $X_1'$ in $\mathcal S$ and $\mathcal S'$ respectively. These restrict to $Y_1 = Y_1'$ on the boundary. Choose a maximal collection of non-separating arcs in $\Sigma_3$. Boundary stabilising the two pseudo-trisections along these arcs produces two new pseudo-trisections (which we still call $\mathcal S$ and $\mathcal S'$). However, the restricted trisection on the boundary now has that $\Sigma_3$ is a disk. This means $Y_{2}\cup Y_{3}$ now glues along a disk, so it is itself a handlebody. In summary the decomposition $Y_{1} \cup (Y_{2} \cup Y_{3})$ is a Heegaard splitting.

	Next we view $X$ as a cobordism from $Y_{2} \cup Y_{3}$ to $Y_{1}$. Formally, we do this by considering a tubular neighbourhood $N$ of the Heegaard splitting surface $\Sigma_{1}\cup \Sigma_{2}$, so that $Y$ is of the form $(Y_1-N) \cup N \cup (Y_2 \cup Y_3 - N)$. The theorem now essentially follows from Theorem 21 of \cite{GayKir}: there are relative handle decompositions (equivalently, Morse functions), induced by the pseudo-trisections and compatible with the cobordism structure. Cerf theory tells us that the handle decompositions are related by a series of moves which correspond to internal stabilisation. In summary, after a sequence of internal stabilisations, our pseudo-trisections are isotopic.
\end{proof}

\begin{remark}The proof of Theorem \ref{dim4uniqueness} doesn't involve Heegaard stabilisation unless one of the original trisections of $X$ restricts to the trivial open book on $Y$. For example, if the boundary of $X$ is not a sphere, the above theorem holds using only boundary and internal stabilisations.
\end{remark}

\subsection{Complexity of pseudo-trisections}
This subsection introduces a notion of complexity for pseudo-trisections, which also applies to relative trisections. It can be considered a generalisation of the \emph{genus} of a trisection (of a closed manifold). We show that each notion of stabilization increases complexity by 1. The minimal complexity of a pseudo-trisection of a given 4-manifold with boundary is bounded above by the minimal complexity of a relative trisection (since relative trisections are pseudo-trisections), but we see with some examples that these minimal complexities do not generally agree.

\begin{defn}The \emph{complexity} of a $(g,\boldsymbol{k},\boldsymbol{y},b)$-pseudo-trisection $\mathcal{T}$ of a 4-manifold $X$ with boundary $\partial X$ is the integer
	$$c(\mathcal{T}) = \chi(X) + |\boldsymbol{k}|-1.$$
	Further, we write $c(\partial \mathcal T)$ to denote the complexity of the induced trisection on $\partial X$, and $c(\mathcal T, \partial \mathcal T)$ to denote $c(\mathcal T) + c(\partial \mathcal T)$. 
	We write $c(X)$ to denote the minimum complexity among all pseudo-trisections of $X$, and $c(X,\partial X)$ to denote the minimum value of $c(\mathcal T, \partial \mathcal T)$ among all pseudo-trisections of $X$.
\end{defn}
\begin{remark}It is immediate that $c(X,\partial X)$ is at least $c(X) + c(\partial X)$, but it is not known whether or not they are always equal.
\end{remark}
\begin{prop}\label{cplx_properties}
	Complexity satisfies the following properties:
	\begin{enumerate}
		\item Recalling that $p_i$ denotes the genus of the surface $\Sigma_i$ of the pseudo-trisection,
			$$c(\mathcal{T}) = g + \frac{1}{2}(|\boldsymbol{y}| + b - 1) = g + |\boldsymbol{p}| + 2b - 2.$$
		\item The three types of stabilisation for pseudo-trisections increase complexity by 1.
		\item Write $\beta_i$ for the Betti numbers of $X$.
			\begin{enumerate}
				\item For $\mathcal{T}$ a relative trisection, $c(\mathcal{T}) \geq \beta_1 + \beta_2$.
				\item For $\mathcal{T}$ a pseudo-trisection, $c(\mathcal{T},\partial\mathcal{T})\geq \beta_1 + \beta_2$.
			\end{enumerate}
		\item Given pseudo-trisected manifolds $(X_1, \mathcal{T}_1)$, $(X_2,\mathcal{T}_2)$, complexity is additive in the sense that $c(\mathcal{T}_1 \natural \mathcal{T}_2) = c(\mathcal{T}_1) + c(\mathcal{T}_2)$.
	\end{enumerate}
\end{prop}
\begin{proof}
	In item (1), the first equality is an application of the last item in Proposition \ref{properties}. The second equality is an application of the second item in the aforementioned proposition.

	For item (2), we use the observation that each of the three types of stabilisation for pseudo-trisections increases one of the $k_i$ by 1, while leaving the other two unchanged.

	For item (3), we require several steps. In step 1 we establish a general homological statement about 4-manifolds with boundary. In step 2 we establish bounds for $\beta_1$ in terms of $|\boldsymbol{k}|$ (and $|\boldsymbol{y}|$) for relative trisections (and pseudo-trisections).

	\emph{Step 1.} In the long exact sequence $H_4(X,\partial X) \xrightarrow{i} H_3(\partial X) \xrightarrow{j} H_3(X) \xrightarrow{k} H_3(X,\partial X)$, the map $k$ is necessarily injective because $i$ is an isomorphism. Combining this fact with the isomorphism $H^1(X) \cong H_3(X,\partial X)$, we have
$$\beta_1 = \rank H^1(X) = \rank H_3(X, \partial X) \geq \rank H_3(X) = \beta_3.$$
On the other hand, $\chi(X) = 1 - \beta_1 + \beta_2 - \beta_3$, so
$$\chi(X) - 1 \geq \beta_2 - 2\beta_1.$$
\emph{Step 2.} Suppose $\mathcal{T}$ is either a relative trisection or pseudo-trisection. Let $\gamma$ be a neatly embedded loop in $X$. Then $\gamma$ meets each $X_i$ along a set of arcs whose endpoints lie in $\partial X_i$. Since each $X_i$ is a handlebody, the arcs can be homotoped rel boundary to lie in $\partial X_i$. We do this in every sector, to obtain a loop $\gamma'$ contained entirely within the 3-skeleton of $\mathcal{T}$. Next we repeat this step a dimension lower: $\gamma'$ intersects each $H_i$ or $Y_i$ along an arc with endpoints in $\partial H_i$ or $\partial Y_i$ respectively. These arcs can again be homotoped rel boundary to lie within $\partial H_i$ and $\partial Y_i$ respectively. Doing this in every 3-dimensional sector, we obtain a loop $\gamma''$ contained entirely within the 2-skeleton of $\mathcal{T}$.

The proof now diverges for pseudo-trisections and relative trisections. In the former case, notice that $\gamma''$ can be viewed as loop in $X_i \cup \Sigma_{i-1}$ for any $i$. This means $\pi_1(X_i \cup \Sigma_{i-1})$ surjects onto $\pi_1(X)$. Taking into account how $\Sigma_{i-1}$ is glued to $X_i$, it follows that $2p_{i-1} + b - 1 + k_i \geq \beta_1$. This holds over each $i$, giving
$$|\boldsymbol k| + |\boldsymbol{y}| = 2|\boldsymbol{p}| + 3b - 3 + |\boldsymbol{k}| \geq 3\beta_1.$$
Combining this with the result from step 1 gives
$$c(\mathcal{T}) + c(\partial\mathcal{T}) = \chi(X) + |\boldsymbol{k}| - 1 + |\boldsymbol{y}| \geq 3\beta_1 + \beta_2 - 2\beta_1 = \beta_1 + \beta_2.$$

In the case of relative trisections, the boundary is an open book. Therefore the arcs $\gamma'' \cap \Sigma_i$ can be homotoped rel boundary to any $\Sigma_j$ via the fibration defining the open book. It follows that $\gamma''$ can be homotoped to lie entirely within any one of the $X_i$. It follows that $|\boldsymbol{k}| \geq 3\beta_1$, and therefore that $c(\mathcal{T}) \geq \beta_1 + \beta_2$.

Finally, for item (4), we have that
\begin{align*}
	c(\mathcal{T}_1\natural \mathcal{T}_2) &= \chi(X_1\natural X_2) + |\boldsymbol{k}_{\mathcal{T}_1 + \mathcal{T}_2}| - 1\\
					       &= \chi(X_1) + \chi(X_2) - 1 + |\boldsymbol{k}_{\mathcal{T}_1}|+|\boldsymbol{k}_{\mathcal{T}_2}| - 1 = c(\mathcal{T}_1) + c(\mathcal{T}_2)
					      \end{align*}
					     as required.
	\end{proof}
	\begin{remark} Proposition \ref{cplx_properties} implies that $c(X_1\natural X_2) \leq c(X_1) + c(X_2)$. It is not known if this inequality is ever strict. This is analogous to the question of whether or not trisection genus is additive under connected sum.
	\end{remark}
\begin{eg}\label{Proof_cplx_examples}
Figure \ref{cplx_table} lists the minimum complexities of some relative trisections and pseudo-trisections. It demonstrates that even for simple 4-manifolds complexities of relative trisections tend to be much greater than the lower bound given in Proposition \ref{cplx_properties}. On the other hand, complexities of pseudo-trisections appear to be equal to $\beta_1 + \beta_2$. This suggests that relative trisections tend to have significantly higher complexity than pseudo-trisections.

The five (minimal) pseudo-trisections in Figure \ref{cplx_table} are presented in Examples \ref{trivialB4}, \ref{B3xS1}, \ref{B2xS2}, and \ref{CP2-B4} and below in Figure \ref{2S2xD2_figure}. Since any pseudo-trisection with complexity $0$ is the trivial pseudo-trisection, it is automatic that the given pseudo-trisections for $S^1\times B^3$ and $S^2\times D^2$ must be minimal. Next, one can see that any pseudo-trisection with complexity $1$ must have boundary a lens space or $S^1\times S^2$. It follows that the given pseudo-trisection for $(S^2\times D^2)\natural (S^2\times D^2)$ is also minimal. Another minimal pseudo-trisection of $(S^2\times D^2)\natural (S^2\times D^2)$ is given by the boundary connected sum of a complexity 1 pseudo-trisection with itself.

	For relative trisections, it is again clear that the trivial trisection of $B^4$ is minimal, and any other trisection has complexity at least $1$. Therefore the relative trisection given in Example \ref{CP2-B4} is minimal. Any relative trisection with complexity $1$ must have boundary $S^3$. It follows that a relative trisection of $S^1\times B^3$ with complexity $2$ would be minimal. Such a relative trisection exists---it is presecribed by the relative trisection diagram consisting of an annulus (and no curves). Moreover, this is the only relative trisection with complexity $2$ and boundary not a sphere, as can be seen by enumerating diagrams.

From here, it follows that any relative trisection of $S^2\times D^2$ with complexity 3 must be minimal. However, the author could not find such an example, or rule out its existence. To the author's best knowledge, the smallest known relative trisection of $S^2\times D^2$ has complexity 4---see Example 2.9 of \cite{CasOzb}.

Finally, one can show again by enumerating diagrams that any relative trisection of $(S^2\times D^2) \natural (S^2\times D^2)$ must have complexity at least $4$. To the author's best knowledge, the smallest known relative trisection of $(S^2\times D^2)\natural (S^2\times D^2)$ has complexity 8, by taking the ``plumbing" of a relative trisection of $S^2\times D^2$ with itself (in the sense of Theorem 3.20 of \cite{CasIslMilTom}). Note that it is not known whether or not relative trisection genus (equivalently, complexity) is additive under boundary connected sum \cite{Additivity}.
\end{eg}

\begin{figure}
\begin{tabular}{|c|c|c|cc|}
\hline
		 & & {Rel.-T.} & \multicolumn{2}{c|}{Pseudo-T.}              \\ \hline
		 & $\beta_1 + \beta_2$ & $c(X)$ & \multicolumn{1}{c|}{$c(X)$} & $c(X, \partial X)$ \\ \hline
	$B^4$            & 0 & 0 & \multicolumn{1}{c|}{0}       & 0         \\ \hline
	$\CP^2 - B^4$    & 1 & 1 & \multicolumn{1}{c|}{1}       & 1         \\ \hline
	$S^1\times B^3$  & 1 & 2 & \multicolumn{1}{c|}{1}       & 3         \\ \hline
	$S^2 \times D^2$ & 1 & [3,4] & \multicolumn{1}{c|}{1}       & 3         \\ \hline
	$(S^2\times D^2) \natural (S^2\times D^2)$ & 2 & [4,8] & \multicolumn{1}{c|}{2} & 5 \\ \hline
\end{tabular}
\caption{Complexities of some simple 4-manifolds. When not known, the interval of possibilities is written.}
\label{cplx_table}
\end{figure}

\begin{eg}\label{cplx_big_gap}For pseudo-trisections, $c(\natural^\ell (S^2\times D^2)) \leq \ell$. This is because we can iteratively take the boundary connected sum of the complexity 1 pseudo-trisection in Example \ref{B2xS2} with itself, and complexity is subadditive. On the other hand, restricting to relative trisections and assuming additivity of relative trisection genus, $c(\natural^\ell (S^2\times D^2)) \geq 3\ell$ by Example \ref{Proof_cplx_examples}.
\end{eg}
\begin{remark}
	In Examples \ref{Proof_cplx_examples} and \ref{cplx_big_gap}, we reference the question of whether or not relative trisection genus is additive under boundary connected sum. Evidence for the additivity of relative trisection genus is presented in a recent preprint of Takahashi \cite{Tak}. Note that additivity of relative trisection genus would imply additivity of (closed) trisection genus, which in turn implies the Smooth Poincar\'e conjecture in dimension 4 \cite{LamMei}.
\end{remark}
	\begin{remark} The author is unaware of manifolds $X$ for which $c(X) < \beta_1 + \beta_2$ or $c(X,\partial X) = \beta_1 + \beta_2, c(\partial X) > 0$. It appears likely that the bound given in Proposition \ref{cplx_properties} for pseudo-trisections can be improved.
	\end{remark}
	\begin{remark} To better measure the extent to which pseudo-trisections have lower complexity than relative trisections, we would need a sufficiently tight \emph{upper bound} on the complexity of pseudo-trisections. This corresponds to further refining existence results in a way which is not done here. 	\end{remark}

\section{Pseudo-trisection diagrams}\label{diagrams}
In this section we introduce \emph{pseudo-trisection diagrams}. These represent pseudo-trisections in the same way that trisection diagrams represent relative trisections. We establish a one-to-one correspondence between pseudo-trisections diagrams (up to diagram moves) and pseudo-trisections (up to stabilisation). We also use pseudo-trisection diagrams to prove Theorem \ref{pseudo_existence_2}, namely that for any 4-manifold with a given trisection of the boundary, there exists a pseudo-trisection of the 4-manifold extending the boundary trisection. Finally we describe some orientation conventions.
\subsection{Pseudo-trisection diagrams}\label{diagramssection}

In this subsection we introduce \emph{pseudo-trisection diagrams}, which are diagrams for pseudo-trisections extended from triple Heegaard diagrams (in the same fashion that trisection diagrams are extended from Heegaard diagrams). A pseudo-trisection diagram consists of several arcs and closed curves on four surfaces, so that restricting to any three surfaces (and correspondingly coloured arcs) produces a triple Heegaard diagram. This is analogous to trisection diagrams consisting of three simultaneous Heegaard diagrams.

\begin{defn}\label{pseudo_tri_diagram_defn} A \emph{$(g, \boldsymbol k, \boldsymbol y, b)$-pseudo-trisection diagram} consists of the data $(\Sigma_C, \Sigma_1, \Sigma_2, \Sigma_3, \boldsymbol\alpha_1, \boldsymbol\alpha_2,$ $\boldsymbol\alpha_3, \boldsymbol\delta_{1}, \boldsymbol\delta_{2}, \boldsymbol\delta_{3})$, with the following constraints.
	\begin{enumerate}
		\item Each $\Sigma_{\circ}$, $\circ \in \{C, 1, 2, 3\}$ is a surface with boundary. Spefically, $\Sigma_C$ has genus $g$ and $b$ boundary components, while $\Sigma_i$ has genus $p_i = \frac{1}{2}(y_{i-1}+y_i - y_{i+1} - b + 1)$ and $b$ boundary components.
		\item There is an identification of the boundaries of all four of the $\Sigma_{\circ}$. In particular, $\Sigma_i \cup \Sigma_{i+1}$ is a closed surface of genus $y_i$, and $\Sigma_C \cup \Sigma_i$ is a closed surface of genus $h_i = g + \frac{1}{2}(y_{i-1} + y_i - y_{i+1} + b - 1)$.
		\item Each $\boldsymbol\delta_i$ is a collection of disjoint neatly embedded arcs and simple closed curves on $\Sigma_i$ and $\Sigma_{i+1}$, which glues to a cut system of $y_i$ curves on $\Sigma_i \cup \Sigma_{i+1}$. Similarly, each $\boldsymbol \alpha_i$ is a collection of disjoint neatly embedded arcs and simple closed curves on $\Sigma_C$ and $\Sigma_i$, which glues to a cut system of $h_i$ curves on $\Sigma_C \cup \Sigma_i$.
		\item The data
	\begin{align*}
		&(\Sigma_C, \Sigma_1, \Sigma_2, \boldsymbol\alpha_1, \boldsymbol\delta_{1}, \boldsymbol\alpha_2)\\
		&(\Sigma_C, \Sigma_2, \Sigma_3, \boldsymbol\alpha_2, \boldsymbol\delta_{2}, \boldsymbol\alpha_3)\\
		&(\Sigma_C, \Sigma_3, \Sigma_1, \boldsymbol\alpha_3, \boldsymbol\delta_{3}, \boldsymbol\alpha_1)
	\end{align*}
	are all \emph{triple Heegaard diagrams}. Each of these three triple Heegaard diagrams encode $\#^{k_i}S^1 \times S^2$ respectively.
	\end{enumerate}
\end{defn}
For notational brevity, we frequently write $(\Sigma_C,\Sigma_i, \boldsymbol\alpha_i,\boldsymbol\delta_i)$ instead of $(\Sigma_C,\Sigma_1, \Sigma_2, \Sigma_3, \boldsymbol\alpha_1,\boldsymbol\alpha_2,\boldsymbol\alpha_3,$ $\boldsymbol\delta_1,\boldsymbol\delta_2,\boldsymbol\delta_3)$. 

\begin{defn}\label{realisation}There is a \emph{realisation map}
	$$\mathcal R: \{(g,\boldsymbol{k};\boldsymbol{y},b)\text{-pseudo-trisection diagrams}\} \to \frac{\{(g,\boldsymbol{k};\boldsymbol{y},b)\text{-pseudo-trisections}\}}{\text{diffeomorphism}}$$ defined as follows:
	\begin{enumerate}
		\item First, glue the underlying surfaces $\Sigma_C, \Sigma_1, \Sigma_2, \Sigma_3$ along their common boundaries to obtain the \emph{2-skeleton} $W$.
		\item Next, each pair of surfaces glues to a closed surface, and there is a corresponding cut system. (For example, $(\Sigma_C, \Sigma_2, \boldsymbol{\alpha}_2)$.) Each cut system describes how to glue in a 3-dimensional 1-handlebody, as in the proof of Proposition \ref{3mantriuptodiffeo}. Gluing each of the six 3-dimensional 1-handlebodies produces the \emph{3-skeleton} $Z$.
		\item By assumption, the diagrams $\mathcal{D}_1 = (\Sigma_C, \Sigma_1, \Sigma_2, \boldsymbol\alpha_1, \boldsymbol\delta_{1}, \boldsymbol\alpha_2)$, $\mathcal{D}_2 = (\Sigma_C, \Sigma_2, \Sigma_3, \boldsymbol\alpha_2, \boldsymbol\delta_{2}, \boldsymbol\alpha_3)$, and $\mathcal{D}_3 = (\Sigma_C, \Sigma_3, \Sigma_1, \boldsymbol\alpha_3, \boldsymbol\delta_{3}, \boldsymbol\alpha_1)$ are each triple Heegaard diagrams of $\#^{k_i}(S^1\times S^2)$. These bound 4-dimensional 1-handlebodies, which are labelled $X_1, X_2$, and $X_3$ respectively, and are precisely the sectors of a pseudo-trisection.
	\end{enumerate}
\end{defn}
\begin{prop}\label{realisation_welldefined}The \emph{realisation map} $\mathcal R$ is well defined in the sense that any two pseudo-trisections constructed by the definition of $\mathcal R$ are indeed diffeomorphic.
\end{prop}
\begin{proof}Suppose $X$ and $X'$ are pseudo-trisected 4-manifolds constructed from a given pseudo-trisection diagram by $\mathcal R$. Let $X_i$ and $X_i'$ be the sectors, $W$ and $W'$ the 2-skeleta, and $Z$ and $Z'$ be the 3-skeleta of $X$ and $X'$ respectively. Notice that the 2-skeleta are exactly the union of the surfaces in the pseudo-trisection diagram, so we start with the identity map $W \to W'$. Next, we extend this to a map $Z \to Z'$ by (the proof of) Proposition \ref{3mantriuptodiffeo}. This map restricts to diffeomorphisms $\partial X_i \to \partial X_i'$ as trisected 3-manifolds. Finally, we apply a theorem of Laudenbach and Po\'enaru \cite{LauPoe} which states that diffeomorphisms of $\#k(S^1\times S^2)$ extend to diffeomorphisms of $\natural^k(S^1\times B^3)$. Namely, the diffeomorphisms $\partial X_i \to \partial X_i'$ above each extend to maps $X_i \to X'_i$. These three extensions glue along $Z$ to give a diffeomorphism $X \to X'$.
\end{proof}
\begin{figure}	
	\noindent\hspace{0px}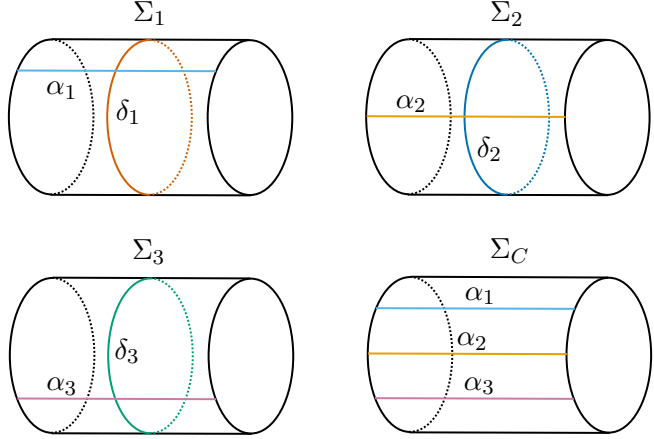
	\caption{A pseudo-trisection diagram of $(S^2\times D^2) \natural (S^2\times D^2)$.}
\label{2S2xD2_figure}
\end{figure}

\begin{eg}Figure \ref{2S2xD2_figure} is an example of a pseudo-trisection diagram, namely of $(S^2\times D^2) \natural (S^2\times D^2)$. Restricting the diagram to $(\Sigma_1, \Sigma_2, \Sigma_3, \delta_1, \delta_2, \delta_3)$ produces exactly the triple Heegaard diagram of $(S^1\times S^2)\# (S^1\times S^2)$ shown in Figure \ref{2S1xS2_figure}. On the other hand, the other triple Heegaard diagrams such as $(\Sigma_C, \Sigma_1, \Sigma_2, \alpha_1, \delta_1, \alpha_2)$ each describe $S^1\times S^2$, so this is a valid pseudo-trisection diagram. One can show that the diagram really does encode $(S^2\times D^2)\natural (S^2\times D^2)$ by applying the realisation map.
\end{eg}
\begin{figure}	
	\noindent\hspace{0px}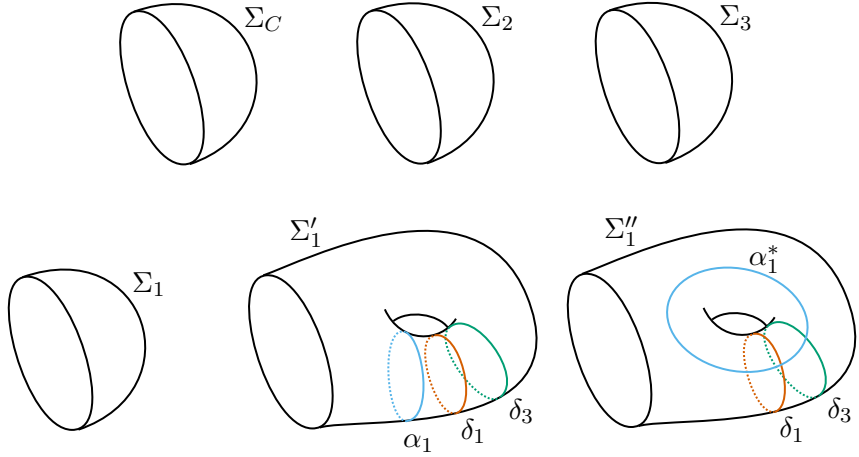
	\caption{Three pseudo-trisection diagrams (arising from three different choices of $\Sigma_1$).}
\label{simple_diagrams_figure}
\end{figure}

\begin{eg}Figure \ref{simple_diagrams_figure} depicts several of the simplest pseudo-trisection diagrams, by using the same choice of $\Sigma_C, \Sigma_2$, and $\Sigma_3$, but switching out $\Sigma_1$ and the associated curves.
	\begin{align*}
		(\Sigma_C, \Sigma_1, \Sigma_2, \Sigma_3, \varnothing, \varnothing, \varnothing, \varnothing, \varnothing, \varnothing) &\longrightarrow B^4,\\
		(\Sigma_C, \Sigma_1', \Sigma_2, \Sigma_3, \alpha_1, \varnothing, \varnothing, \delta_1, \varnothing, \delta_3) &\longrightarrow S^1\times B^3,\\
		(\Sigma_C, \Sigma_1'', \Sigma_2, \Sigma_3, \alpha_1^*, \varnothing, \varnothing, \delta_1, \varnothing, \delta_3) &\longrightarrow S^2\times B^2.
	\end{align*} In fact, these three diagrams are precisely the diagrams giving rise to the pseudo-trisections in Examples \ref{trivialB4}, \ref{B3xS1}, and \ref{B2xS2}.
\end{eg}
\begin{lem}\label{handleslides_isotopies_lem}Fix a pseudo-trisection $\mathcal T$ of $X$. Any two pseudo-trisection diagrams $\mathcal D$ and $\mathcal D'$ with realisation $\mathcal T$ are diffeomorphic after a sequence of handleslides and isotopies. (That is, there are diffeomorphisms $\varphi_\circ : \Sigma_\circ \to \Sigma_\circ'$ rel boundary preserving the arcs on the surfaces.)
\end{lem} 
\begin{proof} This essentially follows from the proof of Lemma \ref{diagramuniqueness_lem}. For each sector $Y_i$ or $H_i$ of the trisection, the corresponding \emph{meridian system} (considered up to isotopy) \cite{Joh} is unique up to handleslides. A pseudo-trisection diagram is a collection of six overlapping meridian systems.
\end{proof}
Note that \emph{handleslides} and \emph{isotopies} of pseudo-trisections are identical to those of triple Heegaard diagrams, described in Example \ref{handleslide_eg}.

Next we give diagrammatic descriptions of boundary stabilisation, Heegaard stabilisation, and internal stabilisation. At the diagrammatic level, we find that Heegaard stabilisation and internal stabilisation fit into the same framework, which we call \emph{torus stabilisation}.
\begin{defn}\label{torus_stab_defn}Given a pseudo-trisection diagram, a \emph{torus stabilisation} is a new pseudo-trisection diagram obtained by increasing the genus of one of the $\Sigma_i$ or $\Sigma_C$ by 1, and correspondingly adding three curves as shown in Figure \ref{torus_stab_figure}. There are two types, as follows:
	\begin{itemize}
		\item \emph{Type I.} Let $T^2$ be a torus with three curves, namely two meridians $\mu_1, \mu_2$ and a longitude $\gamma$. For a fixed $j \in \{1,2,3\}$, a \emph{type I torus stabilisation} of $(\Sigma_C, \Sigma_i, \boldsymbol\alpha_i, \boldsymbol\delta_{i})$ is the diagram obtained by replacing $\Sigma_j$ with $\Sigma_j \# T^2$, $\boldsymbol\alpha_j$ with $\boldsymbol\alpha_j \cup \{\mu_1\}$, $\boldsymbol\delta_{j-1}$ with $\boldsymbol\delta_{j-1}\cup \{\mu_2\}$, and $\boldsymbol\delta_j$ with $\boldsymbol\delta_j \cup \{\gamma\}$.
		\item \emph{Type II.} Let $(T^2, \mu_1, \mu_2, \gamma)$ be as above. For a fixed $j \in \{1, 2, 3\}$, a \emph{type II torus stabilisation} of $(\Sigma_C, \Sigma_i, \boldsymbol\alpha_i, \boldsymbol\delta_{i})$ is the diagram obtained by replacing $\Sigma_C$ with $\Sigma_C \# T^2$, $\boldsymbol\alpha_j$ with $\boldsymbol\alpha_j \cup \{\gamma\}$, $\boldsymbol\alpha_{j+1}$ with $\boldsymbol\alpha_{j+1}\cup \{\mu_1\}$, and $\boldsymbol\alpha_{j-1}$ with $\boldsymbol\alpha_{j-1} \cup \{\mu_2\}$.
	\end{itemize}
\end{defn}
One can check that torus stabilisation is well defined by observing that the result of torus stabilisation is a valid pseudo-trisection diagram, and moreover that it is a diagram of the same underlying 4-manifold.

\begin{figure}	
	\noindent\hspace{80px}%% Creator: Inkscape 1.2.2 (1:1.2.2+202305151915+b0a8486541), www.inkscape.org
%% Accompanies image file 'torus_stabilisation.pdf' (pdf, eps, ps)
\begingroup%
  \makeatletter%
  \providecommand\color[2][]{%
    \errmessage{(Inkscape) Color is used for the text in Inkscape, but the package 'color.sty' is not loaded}%
    \renewcommand\color[2][]{}%
  }%
  \providecommand\transparent[1]{%
    \errmessage{(Inkscape) Transparency is used (non-zero) for the text in Inkscape, but the package 'transparent.sty' is not loaded}%
    \renewcommand\transparent[1]{}%
  }%
  \providecommand\rotatebox[2]{#2}%
  \newcommand*\fsize{\dimexpr\f@size pt\relax}%
  \newcommand*\lineheight[1]{\fontsize{\fsize}{#1\fsize}\selectfont}%
  \ifx\svgwidth\undefined%
    \setlength{\unitlength}{440.8849205bp}%
    \ifx\svgscale\undefined%
      \relax%
    \else%
      \setlength{\unitlength}{\unitlength * \real{\svgscale}}%
    \fi%
  \else%
    \setlength{\unitlength}{\svgwidth}%
  \fi%
  \global\let\svgwidth\undefined%
  \global\let\svgscale\undefined%
  \makeatother%
  \begin{picture}(1,0.20070621)%
    \lineheight{1}%
    \setlength\tabcolsep{0pt}%
    \put(0,0){\includegraphics[width=\unitlength,page=1]{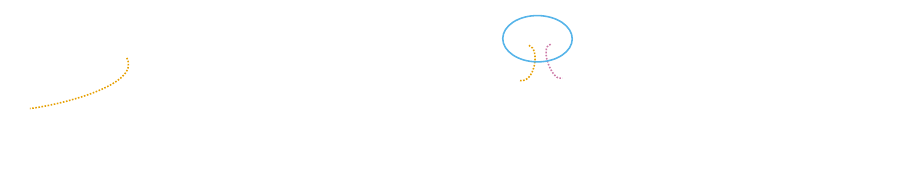}}%
    \put(0.04646436,0.00376552){\makebox(0,0)[lt]{\lineheight{1.25}\smash{\begin{tabular}[t]{l}$(\Sigma_C, \boldsymbol{\alpha}_i \cap \Sigma_C)$\end{tabular}}}}%
    \put(0,0){\includegraphics[width=\unitlength,page=2]{torus_stabilisation.pdf}}%
    \put(0.31806883,0.00498379){\makebox(0,0)[lt]{\lineheight{1.25}\smash{\begin{tabular}[t]{l}$(\Sigma_C, \boldsymbol{\alpha}_i \cap \Sigma_C) \# (T^2, \gamma, \mu_1, \mu_2)$\end{tabular}}}}%
  \end{picture}%
\endgroup%

	\caption{An example of a torus stabilisation (which is type II with the given labels). See Definition \ref{torus_stab_defn}.}
\label{torus_stab_figure}
\end{figure}

\begin{prop}\label{torus_stab_equiv}Torus stabilisations of pseudo-trisection diagrams correspond to Heegaard stabilisations and internal stabilisations of pseudo-trisections. More precisely, the following hold:
	\begin{enumerate}
		\item Let $\mathcal D$ be a pseudo-trisection diagram with realisation $\mathcal T$. Let $\mathcal D'$ be a type I torus stabilisation of $\mathcal D$. Then the realisation of $\mathcal D'$ is a Heegaard stabilisation of $\mathcal T$.
		\item Let $\mathcal T'$ be a Heegaard stabilisation of $\mathcal T$. Then there is a pseudo-trisection diagram $\mathcal D$ such that its realisation is $\mathcal T$, and moreover a type I torus stabilisation of $\mathcal D$ has realisation $\mathcal T'$.
		\item Let $\mathcal D$ be a pseudo-trisection diagram with realisation $\mathcal T$. Let $\mathcal D'$ be a type II torus stabilisation of $\mathcal D$. Then the realisation of $\mathcal D'$ is an internal stabilisation of $\mathcal T$.
		\item Let $\mathcal T'$ be an internal stabilisation of $\mathcal T$. Then there is a pseudo-trisection diagram $\mathcal D$ such that its realisation is $\mathcal T$, and moreover a type II torus stabilisation of $\mathcal D$ has realisation $\mathcal T'$.
	\end{enumerate}
\end{prop}
\begin{proof}Item (1) is a direct consequence of the definition of realisation. Specifically, the connected sum of some $\Sigma_j$ with a torus results in the 3-dimensional 1-handlebodies $H_j, Y_{j-1}$, and $Y_j$ each gaining a handle $h_{H_j}, h_{Y_{j-1}}$, and $h_{Y_j}$ respectively. The curves added in the type I torus stabilisation ensure that $h_{H_j}$ and $h_{Y_{j-1}}$ are parallel, while the other two pairs cancel. Gluing $H_j \cup h_{H_j}, Y_{j-1} \cup h_{Y_{j-1}}$, and $H_{j-1}$, we obtain a 3-manifold differing from $\partial X_{j-1}$ by a connected sum with $S^1\times S^2$. This corresponds to the boundary of the handle attached in Heegaard stabilisation.

	To prove item (2) one can choose any pseudo-trisection diagram $\mathcal D$ of $\mathcal T$. The Heegaard stabilisation involves an arc $\alpha$ in some $Y_j$ with endpoints $p, q$ in $\Sigma_j$, with neighbourhood $N(\alpha) \subset Y_j$. Let $\gamma_1$ be the arc obtained by pushing $\alpha$ to the boundary of $N(\alpha)$, and $\gamma_2$ the arc obtained by isotoping $\alpha$ onto $\Sigma_j$ and pushing the endpoints to agree with $\gamma_1$. Let $\gamma$ be the closed curve $\gamma_1 \cup \gamma_2$. Next, let $\mu_1$ and $\mu_2$ be meridians of $N(\alpha)$. Create the diagram $\mathcal D'$ by a surgery replacing disks around $p$ and $q$ with an annulus corresponding to the boundary of $N(\alpha)$, and add new curves $\gamma, \mu_1, \mu_2$ to the diagram as in the definition of type I torus stabilisations. After handleslides and isotopies of $\gamma$, the resulting diagram is a type I torus stabilisation of $\mathcal D$ which has realisation $\mathcal T'$.

	Items (3) and (4) are analogous to (1) and (2) and left to the reader.
\end{proof}

\begin{defn}\label{band_stab_defn}Given a pseudo-trisection diagram, a \emph{band stabilisation} is a new pseudo-trisection diagram obtained by the following procedure:
	\begin{enumerate}
		\item Identify a non-separating neatly embedded arc $\alpha$ in some $\Sigma_i$. Thicken the arc to a band in $\Sigma_i$.
		\item New surfaces $\Sigma_i', \Sigma_{i+1}', \Sigma_{i+2}', \Sigma_C'$ are obtained by subtracting the band in $\Sigma_i'$, and attaching the band to the other three surfaces.
		\item The preexisting curves on the surfaces are cut and pasted if they intersect the band, as shown by the change from $\delta_1$ to $\delta_1'$ in Figure \ref{band_stabilisation_figure}. Three new curves are added, one to each of $\boldsymbol{\delta}_{i+1}, \boldsymbol{\alpha}_{i+1}$, and $\boldsymbol{\alpha}_{i+2}$. Each curve consists of two arcs in the diagram, one in each corresponding surface. Each arc is a cocore of the attached band. See Figure \ref{band_stabilisation_figure} for a depiction of the case when $i = 1$.
	\end{enumerate}
\end{defn}

\begin{figure}	
	\noindent\hspace{20px}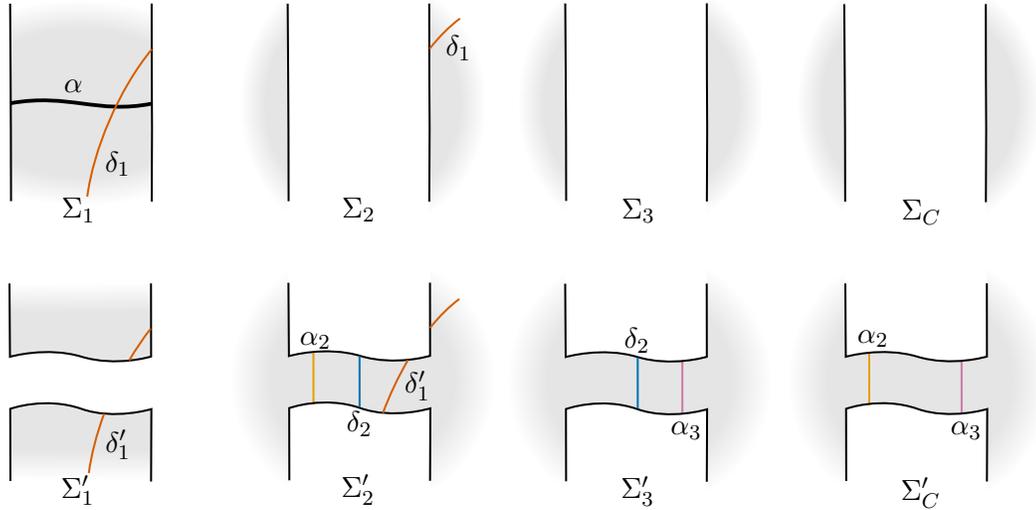
	\caption{An example of a band stabilisation. See Definition \ref{band_stab_defn}.}
\label{band_stabilisation_figure}
\end{figure}

One can check that band stabilisation is well defined by observing that the result is a pseudo-trisection diagram, and moreover that it is a diagram of the same underlying 4-manifold.
\begin{prop}\label{band_stab_equiv}Band stabilisations of pseudo-trisection diagrams correspond to boundary stabilisations of pseudo-trisections. More precisely, the following hold:
	\begin{enumerate}
		\item Let $\mathcal D$ be a pseudo-trisection diagram with realisation $\mathcal T$. Let $\mathcal D'$ be a band stabilisation of $\mathcal D$. Then the realisation of $\mathcal D'$ is a boundary stabilisation of $\mathcal T$.
		\item Let $\mathcal T'$ be a boundary stabilisation of $\mathcal T$. Then there is a pseudo-trisection diagram $\mathcal D$ such that its realisation is $\mathcal T$, and moreover a band stabilisation of $\mathcal D$ has realisation $\mathcal T'$.
	\end{enumerate}
\end{prop}
\begin{proof}The proof is similar to the proof of Proposition \ref{torus_stab_equiv}. The equivalence is established by viewing the surfaces of the diagram as the 2-skeleton of the pseudo-trisection, and equating the arc $\alpha$ in the definition of band stabilisation with the arc used in the definition of boundary stabilisation.
\end{proof}
\begin{thm2}The \emph{realisation map} $$\mathcal R : \{\text{pseudo-trisection diagrams}\} \to \frac{\{\text{pseudo-trisections}\}}{\text{diffeomorphism}}$$ induces a bijection
	$$\frac{\Big\{\text{pseudo-trisection diagrams}\Big\}}{\begin{matrix}\text{band and torus stabilisation,}\\ \text{handleslide, isotopy}\end{matrix}} \longrightarrow \frac{\Big\{\begin{matrix}\text{compact oriented 4-manifolds}\\ \text{with one boundary component}\end{matrix}\Big\}}{\text{diffeomorphism}}.$$	

\end{thm2}
\begin{proof}This proof is analogous to that of Proposition \ref{diagramuniqueness}. By Lemma \ref{handleslides_isotopies_lem}, any two pseudo-trisection diagrams of a fixed pseudo-trisection are diffeomorphic after a sequence of handleslides and isotopies. By Theorem \ref{dim4uniqueness}, any two pseudo-trisections of a given 4-manifold are equivalent up to boundary, Heegaard, and internal stabilisation. By Propositions \ref{band_stab_equiv} and \ref{torus_stab_equiv}, band and torus stabilisations correspond to boundary, Heegaard, and internal stabilisations.
\end{proof}

\subsection{Boundary stabilisation shift and relative existence}
\emph{Boundary stabilisation} of pseudo-trisections is an extension of \emph{stabilisation} of 3-manifold trisections to the entire 4-manifold. In this subsection we show that there is a reverse operation---destabilisation of the boundary extends to the 4-manifold as well. We use this operation to establish that any trisection of the boundary 3-manifold of $X$ extends to a pseudo-trisection of $X$. This mirrors the existence result for relative trisections \cite{GayKir}.

\begin{defn}\label{bdy_stab_shift_defn}\emph{Boundary stabilisation shift} is an operation on pseudo-trisections with stabilised boundary defined diagrammatically in Figure \ref{bdy_stabilisation_shift_figure} which can be thought of as exchanging a stabilisation of the boundary for an internal stabilisation. The details are as follows:
	\begin{enumerate} 
		\item Let $X$ be equipped with a pseudo-trisection $\mathcal T$ such that $\mathcal T|_{\partial X}$ is a stabilisation of a trisection $\tau'$ of $\partial X$.  
		\item Let $(\Sigma_C, \Sigma_i, \boldsymbol{\alpha}_i, \boldsymbol{\delta}_i)$ be a pseudo-trisection diagram of $\mathcal T$ realising the stabilisation of the boundary. (This is the top row of Figure \ref{bdy_stabilisation_shift_figure}.) Locally, each of $\Sigma_1, \Sigma_2$, and $\Sigma_3$ are as depicted by shading, but nothing can be said about how $\Sigma_C$ meets its boundary.
		\item The \emph{boundary stabilisation shift} of $\mathcal T$ is the pseudo-trisection given by the diagram in the bottom row of Figure \ref{bdy_stabilisation_shift_figure}. This is obtained by band surgery on each of $\Sigma_i$ and $\Sigma_C$: a band is glued to obtain each of $\Sigma_1'$ and $\Sigma_C'$, with a corresponding meridian in $\boldsymbol{\alpha}_1$, and bands are deleted to obtain $\Sigma_2'$ and $\Sigma_3'$.
	\end{enumerate}
\end{defn}

\begin{figure}	
	\noindent\hspace{47px}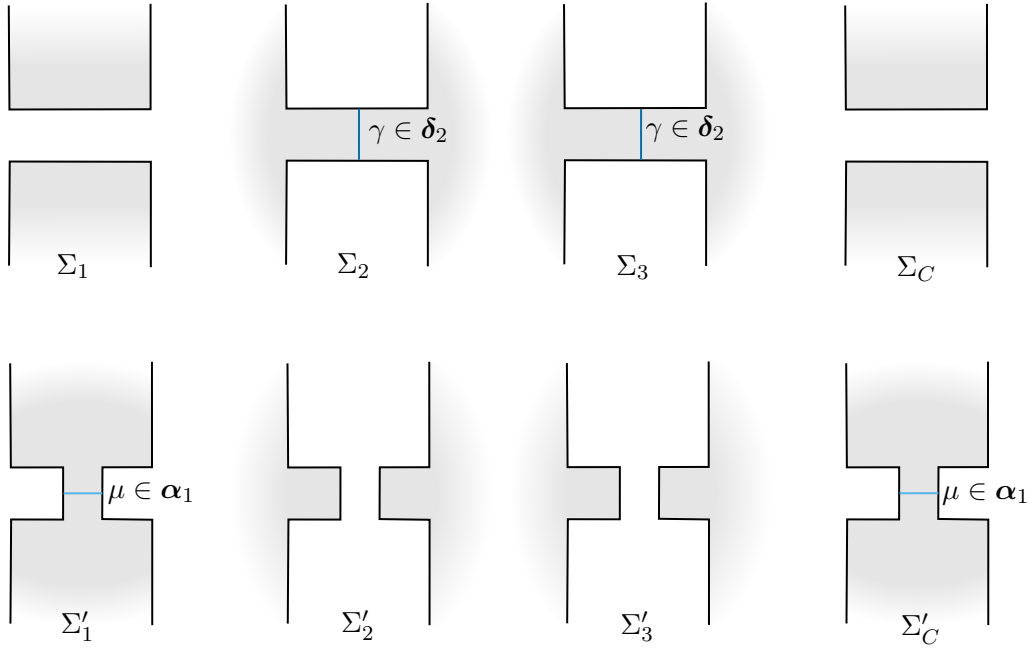
	\caption{A diagrammatic definition of boundary stabilisation shift, see Definition \ref{bdy_stab_shift_defn}.}
\label{bdy_stabilisation_shift_figure}
\end{figure}

\begin{lem}\label{bdy_destab_well_defined}\emph{Boundary stabilisation shift} is well defined.
\end{lem}
\begin{proof} Referring back to Definition \ref{pseudo_tri_diagram_defn}, being well defined requires that each $\boldsymbol{\delta}_i$ and $\boldsymbol{\alpha}_i$ must be a cut system of its corresponding surface, and that each of the triple Heegaard diagrams induced on the three triples $(\Sigma_C', \Sigma_i', \Sigma_{i+1}')$ are specifically diagrams of connected sums of $S^1\times S^2$.

	For the first condition, observe that the only surface-pairs whose unions have changed in the operation are $\Sigma_2' \cup \Sigma_3'$ and $\Sigma_1' \cup \Sigma_C'$. In the former, the genus is decreased by a surgery along a neighbourhood of $\gamma$, so $\boldsymbol{\delta}_2 - \{\gamma\}$ is a cut system of $\Sigma_2' \cup \Sigma_3'$. In the latter, the genus is increased and $\mu$ is the corresponding meridian, so $\boldsymbol{\alpha}_1 \cup \{\mu\}$ is a cut system.

	For the second condition, we observe that $(\Sigma_C', \Sigma_1', \Sigma_2', \boldsymbol{\alpha}_1 \cup \{\mu\}, \boldsymbol{\delta}_1, \boldsymbol{\alpha}_2)$ is a stabilisation of $(\Sigma_C, \Sigma_1, \Sigma_2, \boldsymbol{\alpha}_1, \boldsymbol{\delta}_1, \boldsymbol{\alpha}_2)$. The latter is necessarily a diagram of $\#^{k_1}(S^1\times S^2)$ because our operation is on a valid pseudo-trisection diagram. Consequently, $(\Sigma_C', \Sigma_1', \Sigma_2', \boldsymbol{\alpha}_1\cup \{\mu\}, \boldsymbol{\delta}_1, \boldsymbol{\alpha}_2)$ is a diagram of $\#^{k_1}(S^1\times S^2)$. Similarly, $(\Sigma_C', \Sigma_3', \Sigma_1', \boldsymbol{\alpha}_3, \boldsymbol{\delta}_3, \boldsymbol{\alpha}_1 \cup \{\mu\})$ is a diagram of $\#^{k_3}(S^1\times S^2)$. The remaining check is that $(\Sigma_C', \Sigma_2', \Sigma_3', \boldsymbol{\alpha}_2, \boldsymbol{\delta}_2 \cup \{\gamma\}, \boldsymbol{\alpha}_3)$ is locally a destabilisation of $(\Sigma_C, \Sigma_2, \Sigma_3, \boldsymbol{\alpha}_2, \boldsymbol{\delta}_2, \boldsymbol{\alpha}_3)$, so it is a diagram of $\#^{k_2}(S^1\times S^2)$.
\end{proof}

\begin{prop}\label{stab_shift_prop}Boundary stabilisation shift of $\mathcal T$ to $\mathcal T'$ satisfies the following properties:
	\begin{enumerate}
		\item $\mathcal T|_{\partial X}$ is a stabilisation of $\mathcal T'|_{\partial X}$ (in the sense of 3-manifold trisections).
		\item Complexity of the pseudo-trisection is invariant under boundary stabilisation shift, while complexity of the boundary decreases by 1. 

		\item Suppose $\mathcal T$ is a boundary stabilisation of a pseudo-trisection $\mathcal T^*$. Then $\mathcal T'$ is an internal stabilisation of $\mathcal T^*$.
		\end{enumerate}
\end{prop}
\begin{proof}Item (1). This is an immediate consequence of the definition, by restricting the diagram to $(\Sigma_i', \boldsymbol{\delta}_i')$ and observing that it stabilises to $(\Sigma_i, \boldsymbol{\delta}_i).$

	Item (2). The complexity of the pseudo-trisection is unchanged since each $k_i$ is unchanged (as was observed in the proof of Lemma \ref{bdy_destab_well_defined}.) The complexity of the boundary decreases by 1, by Proposition \ref{3d_cplx_prop}.

	Item (3). This is proved diagrammatically, in Figure \ref{stabilisation_shift_proof_figure}. The top row shows local regions of a pseudo-trisection diagram of $\mathcal T^*$, and $\alpha$ denotes the curve along which the boundary stabilisation occurs. Applying a boundary stabilisation along $\alpha$ followed by a boundary stabilisation shift, the resulting pseudo-trisection diagram is shown in the second row. Applying diffeomorphisms to the $\Sigma_i'$ and $\Sigma_C'$, and isotoping the curves, the bottom row is obtained. Applying handleslides and isotopies to $\mu$, and a diffeomorphism rel boundary to $\Sigma_C'$, the diagram of the bottom row can be expressed as a type II torus stabilisation of the initial diagram. The result now follows from Proposition \ref{torus_stab_equiv}.
\end{proof}

\begin{figure}	
	\noindent\hspace{32px}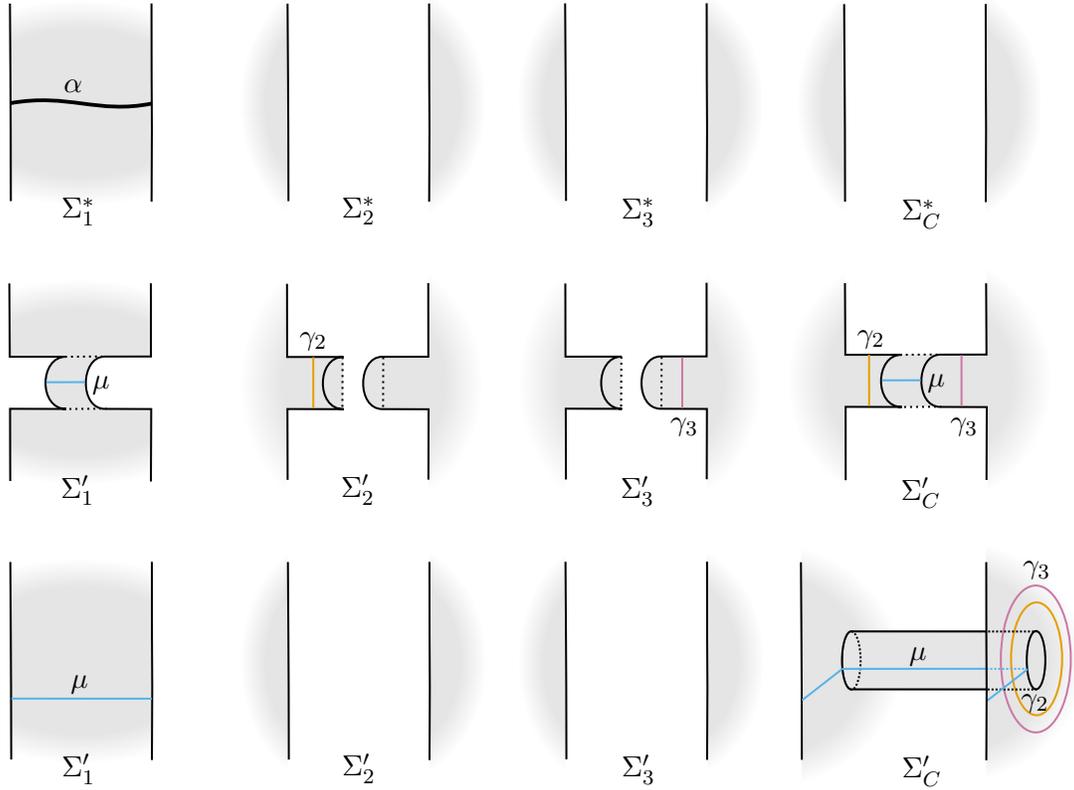
	\caption{Proof of Proposition \ref{stab_shift_prop}, item (3).}
\label{stabilisation_shift_proof_figure}
\end{figure}

\begin{thm0}Let $X$ be a compact oriented 4-manifold with non-empty connected boundary $Y$. Let $\tau$ be a trisection of $Y$. Then there is a pseudo-trisection $\mathcal T_1$ of $X$ such that $\mathcal T_1 |_{Y} = \tau$. 
\end{thm0}
\begin{proof}Case 1: $\tau$ is the trivial trisection of $Y = S^3$. Fill the boundary $Y$ with a 4-ball
	$B^4$ to obtain a closed manifold $X\cup B^4$. There is a trisection of $X\cup B^4$ by the usual existence result in the closed setting \cite{GayKir}. Choose any point in the central surface of the trisection, and there is a trivially trisected neighbourhood $N$ of the point. Now $X$ is diffeomorphic to $(X \cup B^4) - N$, and the latter is equipped with a pseudo-trisection with boundary $\tau$.

	Case 2: $Y$ is equipped with any trisection $\tau$ other than the trivial trisection of $S^3$. ($Y$ may or may not be $S^3$.) Equip $X$ with a pseudo-trisection $\mathcal T$ by Corollary \ref{pseudo_existence_1}. By Proposition \ref{3stabunique}, there is a trisection $\tau'$ of $Y$ such that there is a sequence of stabilisations and Heegaard stabilisations from $\mathcal T|_Y$ to $\tau'$, and a sequence of stabilisations from $\tau$ to $\tau'$. At the 4-dimensional level, there is an induced sequence of boundary stabilisations and Heegaard stabilisations from $\mathcal T$ to some pseudo-trisection $\mathcal T'$ with $\mathcal T'|_Y = \tau'$. Further, there is a sequence of boundary stabilisation shifts from $\mathcal T'$ to a pseudo-trisection $\mathcal T_1$ with $\mathcal T_1|_Y = \tau$.
\end{proof}

\subsection{Orientations of pseudo-trisections}
\label{diagram_orientation}

In this subsection we outline a convention for how a pseudo-trisection diagram determines an orientation of the corresponding 4-manifold. In general this requires one additional label to be added to the pseudo-trisection diagram.

\begin{enumerate}
	\item If the binding of the pseudo-trisection consists of one component and is depicted as a circle in the given pseudo-trisection diagram, it is oriented counterclockwise. Otherwise, one component of the binding must be labeled with an orientation.
	\item The orientation of the boundary component extends to an orientation of the surfaces $\Sigma_\circ, \circ \in \{C, 1, 2, 3\}$ by the outward-normal-first convention. (This further induces orientations on any remaining components of the binding.)
	\item An oriented normal vector to each $\Sigma_\circ$ is determined by the right hand rule. Pieces of the 3-skeleton are oriented as follows:
		\begin{enumerate}
			\item The boundary $\partial Y_i$ of $Y_i$ is oriented as $(-\Sigma_i) \cup \Sigma_{i+1}$. Now $Y_i$ is oriented by the outward-normal-first convention.
			\item The boundary $\partial H_i$ of $H_i$ is oriented as $(-\Sigma_C) \cup \Sigma_i$. Now $H_i$ is oriented by the outward-normal-first convention.
		\end{enumerate}
	\item Next, the boundaries $\partial X_i$ of the sectors are oriented as $Y_i \cup H_i \cup (-H_{i+1})$. The sectors $X_i$ are oriented by the outward-normal-first convention.
	\item The 4-manifold $X$ is oriented as $X_1 \cup X_2 \cup X_3$. The boundary $\partial X = Y$ is oriented as $Y_1 \cup Y_2 \cup Y_3$, and this orientation agrees with the induced boundary orientation from $X$.
\end{enumerate}

\section{Pseudo-bridge trisections}\label{bridge_section}

In this section we extend from considering 4-manifolds to considering pairs $(X, \mathcal K)$ where $\mathcal K$ is an embedded surface. In the case where $X$ and $\mathcal K$ are closed, these pairs are understood via \emph{bridge trisections} as in \cite{MeiZup}. In the relative case (i.e. when $X$ is compact with boundary $\partial X$, and $\mathcal K \subset X$ is a neatly embedded surface with boundary $L \subset \partial X$), there is a notion of relative bridge trisections introduced in general in \cite{Meier}. We summarise relative bridge trisections before introducing an analogue of bridge position for pseudo-trisections. We also introduce \emph{pseudo-bridge trisection diagrams} and describe how some properties of the encoded surface can be read off the diagrams. Recall the conventions for pairs $(X, \mathcal K)$ described in section \ref{conventions}---the introduction.

\subsection{Relative bridge trisections}
The general theory of relative bridge trisections, due to Meier \cite{Meier}, is reviewed here.

In the same way that relative trisections are defined to be determined by the spine (equivalently, a diagram on the central surface), relative bridge trisections are defined so that the embedded surface is encoded by a diagram on the central surface (without, for example, having to describe the boundary link of the surface). This is achieved through additional Morse-theoretic structure.

Suppose $X$ is equipped with a relative trisection. Let $N(B)$ be a tubular neighbourhood of the binding of the relative trisection. Each $H_i' = (X_i \cap X_{i+1}) - N(B)$ has the structure of a compression body between two surfaces with boundary, namely $\Sigma_C' = \Sigma_C - N(B)$ and $\Sigma_{i+1}' = \Sigma_{i+1} - N(B)$. By virtue of being a compression body, there is a standard self-indexing Morse function $\Phi_i$ which vanishes exactly on $\Sigma_C'$ and is maximised exactly on $\Sigma_{i+1}'$. Let $\tau_i$ be a tangle in $H_i$. It can be assumed that $\tau_i$ lies in $H_i'$ because $\tau_i$ avoids $B$ by the transversality theorem. After an isotopy of $\tau_i$, $\Phi_i$ restricted to $\tau_i$ is Morse. An arc in $\tau_i$ is \emph{vertical} if it has no local extrema, and is \emph{flat} if it has a unique local extremum, which is a maximum.

\begin{defn}Let $\tau_i$ be a tangle in $H_i$. This is a \emph{$(b,v)$-tangle} if it is isotopic rel boundary to a tangle consisting exactly of $b$ flat arcs and $v$ vertical arcs. The $v$ vertical arcs may end in separate components of $\partial X$, in which case the components are ordered and we have a $(b,\boldsymbol v)$-tangle where $\boldsymbol v$ is an ordered partition of $v$.
\end{defn}
In relative trisections the compressionbodies $H_i$ are also handlebodies; the notion of triviality above ensures that all arcs can be simultaneously isotoped to lie in the boundary.

\begin{defn}Let $\mathcal{D}_i$ be a disk-tangle in $X_i$. This is a \emph{trivial tangle} if it can be isotoped rel boundary to lie in $\partial X_i$. In this case the boundary is an unlink in $H_i \cup H_{i-1}$, and we say $\partial \mathcal{D}_i$ is in \emph{$(b,\boldsymbol v)$-bridge position} if the tangles in $H_i$ and $H_{i-1}$ are both $(b,\boldsymbol v)$-tangles. Given that the boundary link of $\partial X_i$ is in $(b,\boldsymbol v)$-bridge position, we say that $\mathcal{D}_i$ is a \emph{$(c,\boldsymbol v)$-disk-tangle}, where $c+v$ is the number of components of $\mathcal{D}_i$.
\end{defn}

\begin{defn} Let $X$ be a compact 4-manifold with boundary, and $\mathcal K$ a neatly embedded surface in $X$. Let $b,c, v$ be non-negative integers, $\boldsymbol v$ an ordered partition of $v$ into non-negative integers, and $\boldsymbol c$ an ordered partition of $c$ into three non-negative integers. The surface $\mathcal K$ is in \emph{$(b,\boldsymbol c; \boldsymbol v)$-bridge trisected position} in a relative trisection of $X$ if
	\begin{enumerate}
		\item each $\mathcal D_i = X_i \cap \mathcal K$ is a trivial $(c_i; \boldsymbol v)$-disk-tangle in $X_i$, and
		\item each $\tau_i = H_i \cap \mathcal K$ is a $(b;\boldsymbol v)$-tangle in $H_i$.
	\end{enumerate}
\end{defn}

\begin{remark} The above definition requires a lot of structure, keeping track of $(b,\boldsymbol c;\boldsymbol v)$. This is essentially so that the $\tau_i$ are enough to encode the surface. In the case of pseudo-bridge trisections we'll introduce a more general notion of bridge position for which the analogues of $\tau_i$ \emph{aren't} enough to encode the surface, but projections of tangles onto each surface in the pseudo-trisection diagram are enough. The overarching idea is that we're generalising the notion of bridge position in such a way as to obtain conceptually simpler diagrams and lower diagram complexity for a fixed surface (e.g. lower bridge index), but in exchange for diagrams with more pieces.
\end{remark}

\subsection{Pseudo-bridge trisections}
Next we introduce \emph{pseudo-bridge trisections}, the analogue of relative bridge trisections in the setting of pseudo-trisections. We also introduce \emph{pseudo-shadow diagrams}, which are analogous to shadow diagrams of relative bridge trisections. We show how a pseudo-shadow diagram uniquely encodes a pseudo-bridge trisection up to isotopy. We will immediately see that any pair $(X,\mathcal K)$ admits a pseudo-bridge trisection by application of the existence of relative bridge trisections.

\begin{defn}Let $X$ be a pseudo-trisected 4-manifold, and let $\mathcal K \subset X$ be a neatly embedded surface. The surface $\mathcal K$ is in \emph{pseudo-bridge position} if
	\begin{enumerate}
		\item each $\mathcal D_i = X_i \cap \mathcal K$ is a trivial disk-tangle in $X_i$, and
		\item each $\tau_i = H_i \cap \mathcal K$ and $L_i = Y_i \cap \mathcal K$ is a trivial tangle in $H_i$ and $Y_i$ respectively.
	\end{enumerate}
	Note that here a tangle is \emph{trivial} if it can be isotoped rel boundary to lie in the boundary of the ambient space. A \emph{bridge point} of the pseudo-bridge trisection is any point lying on the intersection between $\mathcal K$ and the 2-skeleton of the pseudo-trisection. The data of a pseudo-trisection together with a surface in pseudo-bridge position is termed a \emph{pseudo-bridge trisection}.
\end{defn}

\begin{eg}\label{D2_in_B4}The standard disk $D^2$ in $B^4$ admits a pseudo-bridge trisection which is also a relative bridge trisection. See Figure \ref{D2_in_B4_figure} (A).
	\begin{enumerate}
		\item Let $B^4$ be the standard unit ball in $\mathbb{R}^4$ expressed in $\{r, \theta,z,w\}$-coordinates.
		\item The trivial pseudo-trisection of $B^4$ \ref{trivialB4} (which is also a relative trisection) is given by decomposing into sectors with $0 \leq \theta \leq 2\pi/3$, $2\pi/3 \leq \theta \leq 4\pi/3$, and $4\pi/3 \leq \theta \leq 2\pi$.
		\item A standard disk embedding is given by $(z,w) = (0,0)$. This disk intersects the aforementioned pseudo-trisection in bridge position. 
	\end{enumerate}
	One can show that this is a \emph{relative bridge trisection} of the standard disk.
\end{eg}
\begin{eg}\label{D2_in_B4_2}The standard disk in $B^4$ also admits a simpler pseudo-bridge trisection with less symmetry, which in particular is not a relative bridge trisection. See Figure \ref{D2_in_B4_figure} (B).
	\begin{enumerate}
		\item Let $B^4$ be parametrised and trisected as in Example \ref{D2_in_B4}.
		\item Consider the cartesian parametrisation of $\mathbb{R}^4$, that is, $(x,y, z, w) = (r\cos \theta, r\sin \theta, z, w)$. The plane $(x,w) = (1/2, 0)$ intersects $B^4$ along a standard disk as a pseudo-bridge trisection. 
	\end{enumerate}
	This pseudo-trisection fails to be a relative trisection for Morse theoretic reasons: although $\tau_1$ is a trivial tangle, it has a local minimum and cannot be isopted to be \emph{flat} or \emph{vertical}.
\end{eg}

\begin{figure}
  \begin{subfigure}{0.48\textwidth}
	  \begin{center}
    %% Creator: Inkscape 1.2.2 (1:1.2.2+202305151915+b0a8486541), www.inkscape.org
%% Accompanies image file 'D2_in_B4.pdf' (pdf, eps, ps)
\begingroup%
  \makeatletter%
  \providecommand\color[2][]{%
    \errmessage{(Inkscape) Color is used for the text in Inkscape, but the package 'color.sty' is not loaded}%
    \renewcommand\color[2][]{}%
  }%
  \providecommand\transparent[1]{%
    \errmessage{(Inkscape) Transparency is used (non-zero) for the text in Inkscape, but the package 'transparent.sty' is not loaded}%
    \renewcommand\transparent[1]{}%
  }%
  \providecommand\rotatebox[2]{#2}%
  \newcommand*\fsize{\dimexpr\f@size pt\relax}%
  \newcommand*\lineheight[1]{\fontsize{\fsize}{#1\fsize}\selectfont}%
  \ifx\svgwidth\undefined%
    \setlength{\unitlength}{184.4301995bp}%
    \ifx\svgscale\undefined%
      \relax%
    \else%
      \setlength{\unitlength}{\unitlength * \real{\svgscale}}%
    \fi%
  \else%
    \setlength{\unitlength}{\svgwidth}%
  \fi%
  \global\let\svgwidth\undefined%
  \global\let\svgscale\undefined%
  \makeatother%
  \begin{picture}(1,0.92625784)%
    \lineheight{1}%
    \setlength\tabcolsep{0pt}%
    \put(0,0){\includegraphics[width=\unitlength,page=1]{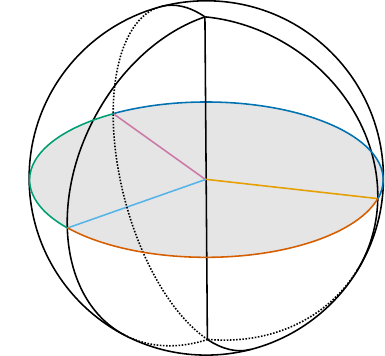}}%
    \put(0.65212261,0.21232571){\makebox(0,0)[lt]{\lineheight{1.25}\smash{\begin{tabular}[t]{l}$L_1$\end{tabular}}}}%
    \put(0.6971371,0.66930249){\makebox(0,0)[lt]{\lineheight{1.25}\smash{\begin{tabular}[t]{l}$L_2$\end{tabular}}}}%
    \put(-0.00046596,0.49484497){\makebox(0,0)[lt]{\lineheight{1.25}\smash{\begin{tabular}[t]{l}$L_3$\end{tabular}}}}%
    \put(0.36039067,0.35709136){\makebox(0,0)[lt]{\lineheight{1.25}\smash{\begin{tabular}[t]{l}$\tau_1$\end{tabular}}}}%
    \put(0.72492265,0.45551393){\makebox(0,0)[lt]{\lineheight{1.25}\smash{\begin{tabular}[t]{l}$\tau_2$\end{tabular}}}}%
    \put(0.34593252,0.52385603){\makebox(0,0)[lt]{\lineheight{1.25}\smash{\begin{tabular}[t]{l}$\tau_3$\end{tabular}}}}%
  \end{picture}%
\endgroup%

    \caption{A standard disk in relative bridge position.}
    \end{center}
    \label{D2_in_B4_fig_a}
  \end{subfigure}
  \hspace*{\fill}
  \begin{subfigure}{0.48\textwidth}
	  \begin{center}
    %% Creator: Inkscape 1.2.2 (1:1.2.2+202305151915+b0a8486541), www.inkscape.org
%% Accompanies image file 'D2_in_B4_b.pdf' (pdf, eps, ps)
\begingroup%
  \makeatletter%
  \providecommand\color[2][]{%
    \errmessage{(Inkscape) Color is used for the text in Inkscape, but the package 'color.sty' is not loaded}%
    \renewcommand\color[2][]{}%
  }%
  \providecommand\transparent[1]{%
    \errmessage{(Inkscape) Transparency is used (non-zero) for the text in Inkscape, but the package 'transparent.sty' is not loaded}%
    \renewcommand\transparent[1]{}%
  }%
  \providecommand\rotatebox[2]{#2}%
  \newcommand*\fsize{\dimexpr\f@size pt\relax}%
  \newcommand*\lineheight[1]{\fontsize{\fsize}{#1\fsize}\selectfont}%
  \ifx\svgwidth\undefined%
    \setlength{\unitlength}{170.82991843bp}%
    \ifx\svgscale\undefined%
      \relax%
    \else%
      \setlength{\unitlength}{\unitlength * \real{\svgscale}}%
    \fi%
  \else%
    \setlength{\unitlength}{\svgwidth}%
  \fi%
  \global\let\svgwidth\undefined%
  \global\let\svgscale\undefined%
  \makeatother%
  \begin{picture}(1,1)%
    \lineheight{1}%
    \setlength\tabcolsep{0pt}%
    \put(0,0){\includegraphics[width=\unitlength,page=1]{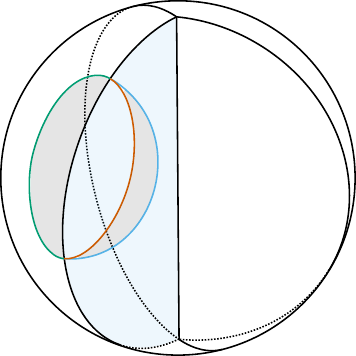}}%
    \put(0.02372575,0.60983391){\makebox(0,0)[lt]{\lineheight{1.25}\smash{\begin{tabular}[t]{l}$L_3$\end{tabular}}}}%
    \put(0.29142743,0.52330456){\makebox(0,0)[lt]{\lineheight{1.25}\smash{\begin{tabular}[t]{l}$L_1$\end{tabular}}}}%
    \put(0.37786461,0.30500412){\makebox(0,0)[lt]{\lineheight{1.25}\smash{\begin{tabular}[t]{l}$\tau_1$\end{tabular}}}}%
  \end{picture}%
\endgroup%

    \caption{A standard disk in pseudo-bridge position.}
    \end{center}
    \label{D2_in_B4_fig_b}
  \end{subfigure}
\caption{Schematics of two distinct pseudo-bridge trisections of a standard disk in $B^4$}
\label{D2_in_B4_figure}
\end{figure}
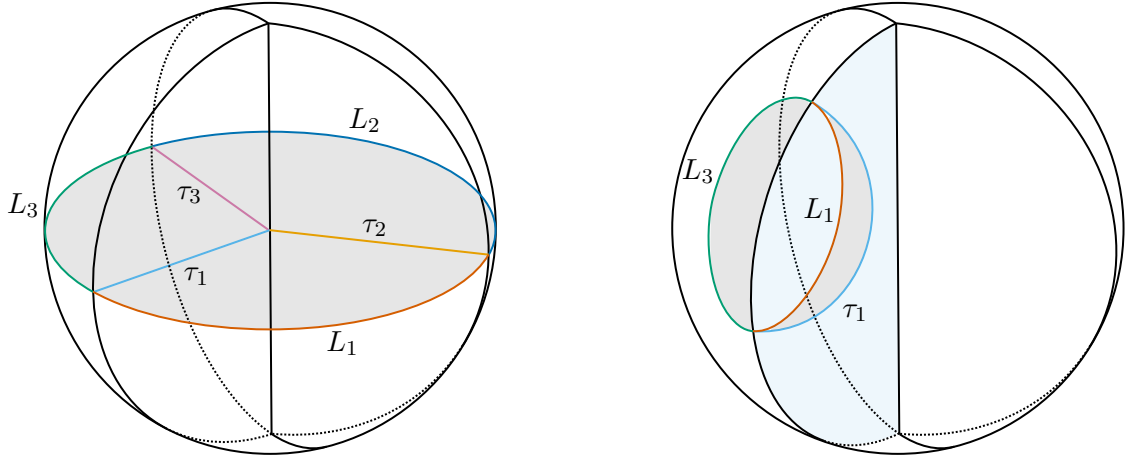

\begin{remark}Any $(X,\mathcal K)$ admits a pseudo-bridge trisection. First, there is a relative trisection $\mathcal T$ of $X$ with non-empty binding by the existence result of Gay and Kirby \cite{GayKir}. Such a relative trisection is necessarily a pseudo-trisection by Proposition \ref{relativeispseudo}. On the other hand, by the existence result of Meier \cite{Meier}, $(X,\mathcal K)$ admits a relative bridge trisection whose underlying relative trisection is $\mathcal {T}$. This relative bridge trisection is necessarily a pseudo-bridge trisection.
\end{remark}

\subsection{Pseudo-shadow diagrams}
\emph{Shadow diagrams} \cite{Meier} are a diagrammatic theory of relative bridge trisections compatible with relative trisection diagrams. Here we introduce \emph{pseudo-shadow diagrams}, a diagrammatic theory of pseudo-bridge trisections compatible with pseudo-trisection diagrams. In particular, any pair $(X,\mathcal K)$ admits a description by a pseudo-shadow diagram.

\begin{defn}Let $\mathcal K, \mathcal K' \subset X$. The surfaces $\mathcal K$ and $\mathcal K'$ are \emph{neatly isotopic} if they are homotopic through neat embeddings.
\end{defn}

\begin{prop}\label{bridgeskeleton} A pseudo-bridge trisected surface $\mathcal K \subset X$ is determined up to neat isotopy by the \emph{3-skeleton} of the pseudo-bridge trisection, that is, by the data of all the tangles $(H_i, \tau_i)$, $(Y_i, L_i)$, and how they intersect.
\end{prop}
\begin{proof} The data of the 3-skeleton describes how the triples $((H_i, \tau_i),(H_{i+1}, \tau_{i+1}),(Y_i, L_i))$ glue together to produce links in the boundaries of the sectors $X_i$. These links are necessarily unlinks, which bound unique trivial disk-tangles in the sectors up to neat isotopy. Compare with Lemma 2.3 and Corollary 2.4 in \cite{MeiZup}.
\end{proof}

Consequently, we can define an analogue of \emph{shadow diagrams} as follows:

\begin{defn}
	A \emph{pseudo-shadow diagram} consists of the data $(\mathcal D, \tau_1, \tau_2, \tau_3, L_1,L_2,L_3)$ where
	\begin{enumerate}
		\item $\mathcal{D} = (\Sigma_C, \Sigma_i, \boldsymbol\alpha_i, \boldsymbol\delta_i)$ is a pseudo-trisection diagram.
		\item For $i \in \{1,2,3\}$, $\tau_i$ is a collection of generically embedded arcs in $\Sigma_i \cup \Sigma_{C}$, all with distinct endpoints. If any two arcs in a given $\tau_i$ cross, the crossing order is indicated. The arcs intersect $B = \partial \Sigma_i = \partial \Sigma_C$ transversely.
		\item Similarly, each $L_i$ is a collection of generically embedded arcs in $\Sigma_i \cup \Sigma_{i+1}$ with distinct endpoints. If any two cross, the crossing order is indicated. Any intersections with $B$ are transverse.
		\item For any $i$, $\tau_i, L_i$, and $L_{i-1}$ may intersect $\Sigma_i$. Each of these families of arcs are pairwise transverse in $\Sigma_i$.
		\item There is a collection of points $\mathcal B_i \subset \Sigma_i$ which consists precisely of the endpoints of $\tau_i, L_i$, and $L_{i-1}$ that lie in $\Sigma_i$. (Consequently, each point in $\mathcal B_i$ has degree 3 in a pseudo-shadow diagram.)
		\item The above two points also hold verbatim for $(\Sigma_C, \tau_1, \tau_2, \tau_3)$ instead of $(\Sigma_i, \tau_i, L_i, L_{i-1})$.The set of endpoints of arcs in $\tau_i$ lying in $\Sigma_C$ is denoted $\mathcal B_C$.
		\item The unions $\tau_i \cup \tau_{i+1} \cup L_i$ form unlinks in each of the 3-dimensional manifolds encoded by the corresponding triple Heegaard diagram in the underlying pseudo-trisection diagram.
	\end{enumerate}
\end{defn}
For notational brevity, we frequently write $(\mathcal D, \tau_i, L_i)$ instead of $(\mathcal D, \tau_1, \tau_2, \tau_3, L_1,L_2,L_3)$.

\begin{defn}\label{realisation_bridge}There is a \emph{realisation map} 
$$\mathcal R : \{\text{pseudo-shadow diagrams}\} \to \frac{\{\text{pseudo-bridge trisections}\}}{\text{diffeomorphism}}$$
extending the realisation map of Definition \ref{realisation} defined as follows:
\begin{enumerate}
	\item A pseudo-shadow diagram has an underlying pseudo-trisection diagram. First, this is mapped to a pseudo-trisection by the realisation map of Definition \ref{realisation}.
	\item Each collection of arcs $\tau_i$ and $L_i$ lies in the 2-skeleton of the above pseudo-trisection. Each collection is isotoped rel boundary to form a tangle in $H_i$ and $Y_i$ respectively.
	\item Each triple union $\tau_i \cup \tau_{i+1} \cup L_i$ forms an unlink in $\partial X_i$. This bounds a collection of disks in $\partial X_i$ (uniquely up to isotopy rel boundary).
	\item The above collections of disks are isotoped rel boundary to be neatly embedded in $X_i$. The result is now a pseudo-bridge trisection.
\end{enumerate}
\end{defn}

\begin{prop}The realisation map of Definition \ref{realisation_bridge} is well defined.
\end{prop}
\begin{proof}We must show that the map as defined truly sends a pseudo-shadow diagram to a pseudo-bridge trisection, and that any two pseudo-bridge trisections built from the same pseudo-shadow diagram by the map are neatly isotopic. (More precisely, after a diffeomorphism, the embedded surfaces differ by neat isotopy.)

	Let $(X, \mathcal K)$ be the realisation of a pseudo-shadow diagram $\mathcal D$. Then each $\mathcal K \cap X_i$ is necessarily a trivial disk tangle, because it is obtained by isotoping a collection of disks rel boundary from $\partial X_i$ to $X_i$. Similarly, each $\mathcal K \cap H_i$ and $\mathcal K \cap Y_i$ is necessarily a trivial tangle. It follows that $(X, \mathcal K)$ is a pseudo-bridge trisection.

	Next, suppose $(X, \mathcal K)$ and $(X',\mathcal K')$ are realisations of the same pseudo-shadow diagram $\mathcal D$. By Proposition \ref{realisation_welldefined}, the underlying 4-manifolds $X$ and $X'$ and their pseudo-trisections are diffeomorphic. Hereafter $\mathcal K, \mathcal K'$ are considered to both lie in $X$. The arcs $\tau_i$ and $L_i$ lift to the 2-skeleton of the pseudo-trisection of $X$, and are further isotoped into $H_i$ and $Y_i$ respectively to produce exactly the 3-skeleta of $(X, \mathcal K)$ and $(X, \mathcal K')$. By Proposition \ref{bridgeskeleton}, the 3-skeleta determine the surface in pseudo-bridge position up to neat isotopy.
\end{proof}

\begin{eg}\label{pseudo_bridge_eg}Figure \ref{disk_B4_diagram_figure} is a pseudo-shadow diagram. This can be verified by inspecting the definition of a pseudo-shadow diagram. The most interesting condition is item (7). The union $\tau_1 \cup \tau_2 \cup L_1$ consists of an unknot in $B^3$, as shown in Figure \ref{disk_B4_trivialcheck}. 

	The underlying pseudo-trisection diagram of the aforementioned pseudo-shadow diagram is that of the trivial trisection (of $B^4$). The surface is built by applying the realisation map. Explicitly, the arcs $\tau_1, \tau_2$, and $L_1$ lift to an unknot in $\partial X_1 = B^3$. Similarly, $\tau_2, \tau_3, L_2$ lift to an unknot in $\partial X_2 = B^3$. The common intersection of these two unknots is $\tau_2$. The two unknots next encode disks in $X_1$ and $X_2$ respectively, and the union of these disks is itself a disk in $X_1 \cup X_2$ with boundary $\tau_1 \cup \tau_3 \cup L_1 \cup L_2$. Finally, $\tau_1 \cup \tau_3 \cup L_3$ lifts to an unknot in $\partial X_3$, which bounds a disk in $X_3$. Gluing in this final disk produces a disk in $X$ exactly as in Figure \ref{D2_in_B4_figure} (A).
\end{eg}

\begin{figure}	
	\noindent\hspace{0px}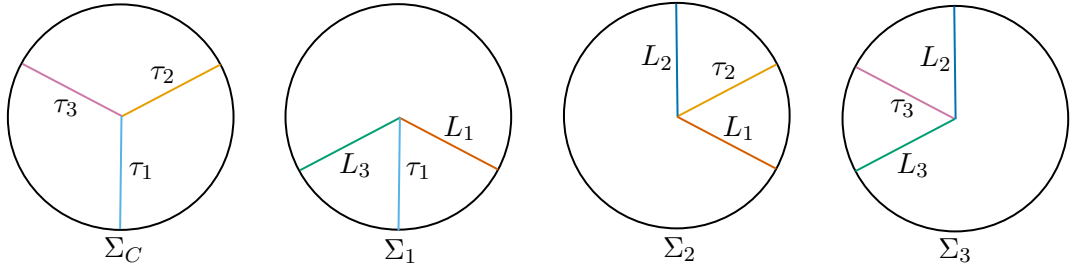
	\caption{A pseudo-shadow diagram of a disk in relative bridge position in $B^4$.}
\label{disk_B4_diagram_figure}
\end{figure}

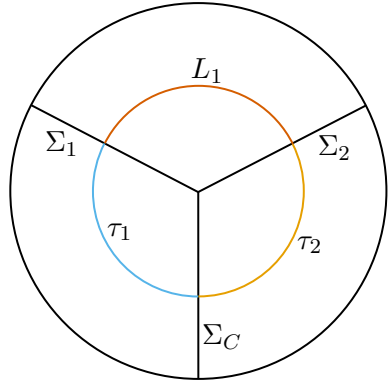
\begin{figure}
	\noindent\hspace{0px}%% Creator: Inkscape 1.2.2 (1:1.2.2+202305151915+b0a8486541), www.inkscape.org
%% Accompanies image file 'disk_B4_trivcheck.pdf' (pdf, eps, ps)
\begingroup%
  \makeatletter%
  \providecommand\color[2][]{%
    \errmessage{(Inkscape) Color is used for the text in Inkscape, but the package 'color.sty' is not loaded}%
    \renewcommand\color[2][]{}%
  }%
  \providecommand\transparent[1]{%
    \errmessage{(Inkscape) Transparency is used (non-zero) for the text in Inkscape, but the package 'transparent.sty' is not loaded}%
    \renewcommand\transparent[1]{}%
  }%
  \providecommand\rotatebox[2]{#2}%
  \newcommand*\fsize{\dimexpr\f@size pt\relax}%
  \newcommand*\lineheight[1]{\fontsize{\fsize}{#1\fsize}\selectfont}%
  \ifx\svgwidth\undefined%
    \setlength{\unitlength}{158.78825293bp}%
    \ifx\svgscale\undefined%
      \relax%
    \else%
      \setlength{\unitlength}{\unitlength * \real{\svgscale}}%
    \fi%
  \else%
    \setlength{\unitlength}{\svgwidth}%
  \fi%
  \global\let\svgwidth\undefined%
  \global\let\svgscale\undefined%
  \makeatother%
  \begin{picture}(1,0.89731741)%
    \lineheight{1}%
    \setlength\tabcolsep{0pt}%
    \put(0,0){\includegraphics[width=\unitlength,page=1]{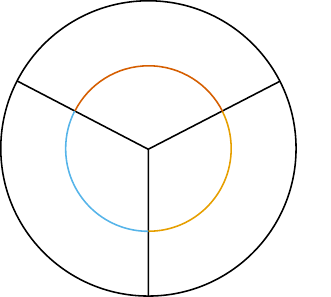}}%
    \put(0.23218859,0.34218612){\makebox(0,0)[lt]{\lineheight{1.25}\smash{\begin{tabular}[t]{l}$\tau_1$\end{tabular}}}}%
    \put(0.68281304,0.31107293){\makebox(0,0)[lt]{\lineheight{1.25}\smash{\begin{tabular}[t]{l}$\tau_2$\end{tabular}}}}%
    \put(0.4311649,0.71663963){\makebox(0,0)[lt]{\lineheight{1.25}\smash{\begin{tabular}[t]{l}$L_1$\end{tabular}}}}%
    \put(0.45758553,0.08522272){\makebox(0,0)[lt]{\lineheight{1.25}\smash{\begin{tabular}[t]{l}$\Sigma_C$\end{tabular}}}}%
    \put(0.731929,0.53354542){\makebox(0,0)[lt]{\lineheight{1.25}\smash{\begin{tabular}[t]{l}$\Sigma_2$\end{tabular}}}}%
    \put(0.08483592,0.54329549){\makebox(0,0)[lt]{\lineheight{1.25}\smash{\begin{tabular}[t]{l}$\Sigma_1$\end{tabular}}}}%
  \end{picture}%
\endgroup%

	\caption{The knot $\tau_1\cup \tau_2 \cup L_1 \subset S^3$, see Example \ref{pseudo_bridge_eg}.}
\label{disk_B4_trivialcheck}
\end{figure}

\begin{eg}Figure \ref{disk_B4_diagram_figure_2} is a pseudo-shadow diagram of a disk in $B^4$, namely that depicted in Figure \ref{D2_in_B4_figure} (B).
\end{eg}

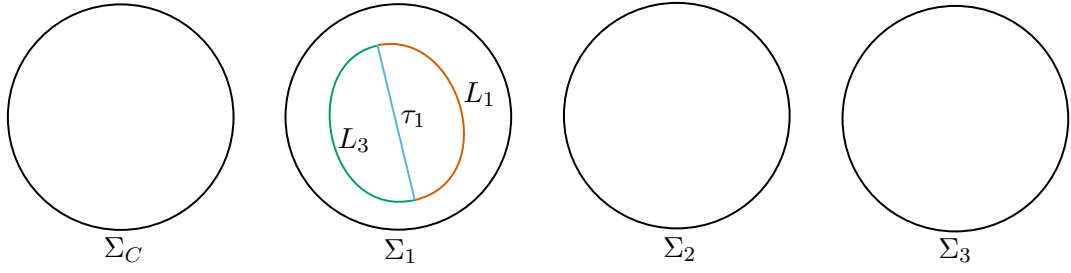
\begin{figure}	
	\noindent\hspace{0px}%% Creator: Inkscape 1.2.2 (1:1.2.2+202305151915+b0a8486541), www.inkscape.org
%% Accompanies image file 'disk_B4_2.pdf' (pdf, eps, ps)
\begingroup%
  \makeatletter%
  \providecommand\color[2][]{%
    \errmessage{(Inkscape) Color is used for the text in Inkscape, but the package 'color.sty' is not loaded}%
    \renewcommand\color[2][]{}%
  }%
  \providecommand\transparent[1]{%
    \errmessage{(Inkscape) Transparency is used (non-zero) for the text in Inkscape, but the package 'transparent.sty' is not loaded}%
    \renewcommand\transparent[1]{}%
  }%
  \providecommand\rotatebox[2]{#2}%
  \newcommand*\fsize{\dimexpr\f@size pt\relax}%
  \newcommand*\lineheight[1]{\fontsize{\fsize}{#1\fsize}\selectfont}%
  \ifx\svgwidth\undefined%
    \setlength{\unitlength}{400.3114428bp}%
    \ifx\svgscale\undefined%
      \relax%
    \else%
      \setlength{\unitlength}{\unitlength * \real{\svgscale}}%
    \fi%
  \else%
    \setlength{\unitlength}{\svgwidth}%
  \fi%
  \global\let\svgwidth\undefined%
  \global\let\svgscale\undefined%
  \makeatother%
  \begin{picture}(1,0.24593409)%
    \lineheight{1}%
    \setlength\tabcolsep{0pt}%
    \put(0,0){\includegraphics[width=\unitlength,page=1]{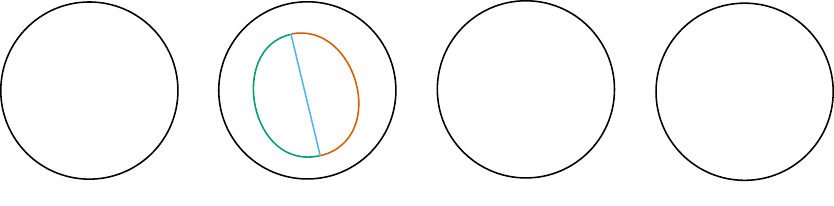}}%
    \put(0.09130874,0.00551286){\makebox(0,0)[lt]{\lineheight{1.25}\smash{\begin{tabular}[t]{l}$\Sigma_C$\end{tabular}}}}%
    \put(0.35497537,0.00414717){\makebox(0,0)[lt]{\lineheight{1.25}\smash{\begin{tabular}[t]{l}$\Sigma_1$\end{tabular}}}}%
    \put(0.61773093,0.00554722){\makebox(0,0)[lt]{\lineheight{1.25}\smash{\begin{tabular}[t]{l}$\Sigma_2$\end{tabular}}}}%
    \put(0.87705064,0.00484398){\makebox(0,0)[lt]{\lineheight{1.25}\smash{\begin{tabular}[t]{l}$\Sigma_3$\end{tabular}}}}%
    \put(0.3709173,0.13194615){\makebox(0,0)[lt]{\lineheight{1.25}\smash{\begin{tabular}[t]{l}$\tau_1$\end{tabular}}}}%
    \put(0.42985378,0.15175413){\makebox(0,0)[lt]{\lineheight{1.25}\smash{\begin{tabular}[t]{l}$L_1$\end{tabular}}}}%
    \put(0.31141033,0.10857385){\makebox(0,0)[lt]{\lineheight{1.25}\smash{\begin{tabular}[t]{l}$L_3$\end{tabular}}}}%
  \end{picture}%
\endgroup%

	\caption{A pseudo-shadow diagram of a disk in pseudo-bridge position in $B^4$.}
\label{disk_B4_diagram_figure_2}
\end{figure}

\begin{comment}
Next we describe orientation conventions. Consider the data $(\Sigma_C, \tau_1 \cap \Sigma_C, \tau_2 \cap \Sigma_C)$. Here $\tau_1 \cap \Sigma_C$ is a projection of (part of) a tangle from $H_1$ to $\Sigma_C$, and similarly $\tau_2 \cap \Sigma_C$ is a projection of (part of) a tangle from $H_2$ to $\Sigma_C$. Locally the arcs $\tau_1 \cap \Sigma_C, \tau_2 \cap \Sigma_C$ lift to a a product $[-1,1]\times \Sigma_C$, with the common endpoints of the arcs lying on $\{0\}\times \Sigma_C$, and one of the $\tau_1, \tau_2$ entirely in $[-1,0]\times \Sigma_C$ and the other entirely in $[0,1]\times \Sigma_C$. This choice corresponds to a choice of how crossings between arcs in the pseudo-shadow diagram lift to the 3-skeleton of the underlying pseudo-trisection.

The orientation conventions are now described, but require some preliminary work.
\end{comment}
\subsection{Resolving crossings in pseudo-shadow diagrams}

When applying the realisation map, arcs in the pseudo-shadow diagram are lifted from the 2-skeleton into the 3-skeleton. Crossings between arcs are resolved in different ways depending on the types of arcs, as well as whether or not the orientations of the underlying surface and 3-manifold are compatible. For example, in Subsection \ref{diagram_orientation}, we saw that $\partial Y_i = (-\Sigma_i) \cup \Sigma_{i+1}$ as an oriented union. Lifting a crossing from $\Sigma_i$ to $Y_i$ may reverse it, while lifting from $\Sigma_{i+1}$ to $Y_i$ may preserve it. The following list outlines how to resolve any crossing when realising a pseudo-bridge diagram.
\begin{remark}The actual \emph{choice} or \emph{convention} happens when we declare how an orientation of the binding (of a pseudo-trisection diagram) induces an orientation of the pseudo-trisected 4-manifold. This determines exactly how arcs in a pseudo-shadow diagram must lift to the 4-manifold. The rules being outlined below should be thought of as ``what do crossings between arcs in a pseudo-shadow diagram actually represent, given the orientation conventions from Subsection \ref{diagram_orientation}?".
\end{remark}
\begin{enumerate}
	\item For each $i$, consider the oriented pairs $(\partial X_i, \Sigma_i), (\partial X_i, \Sigma_{i+1}), (\partial X_i, \Sigma_C)$, and $(Y, \Sigma_i)$. Each of the twelve pairs is further equipped with a normal vector field on $\Sigma_{\circ}, \circ \in \{i, i+1, C\}$ which extends the orientation of $\Sigma_\circ$ to agree with the orientation of the ambient 3-manifold.
	\item In a pseudo-shadow diagram, $L_{i-1}$ \emph{lies above} $L_{i}$ in the sense that any arcs of $L_{i-1}$ intersecting arcs of $L_{i}$ in $\Sigma_{i}$ lift to a position above $L_{i}$ with respect to the aforementioned normal vector field on $(Y, \Sigma_{i})$. Similarly, $\tau_i$ lies above $L_i$; $L_i$ lies above $\tau_{i+1}$; and $\tau_{i+1}$ lies above $\tau_i$. (Note that specifying two tangles in the pseudo-shadow diagram determines a unique ambient pair $(\partial X_i, \Sigma_i), \ldots, (Y, \Sigma_i)$.)
	\item Any self-crossings, that is, crossings of $L_i$ or $\tau_i$ with themselves in $\Sigma_\circ$, necessarily have crossing order indicated in the diagram. We say crossing orders are \emph{preserved} if the over-strand \emph{lies above} the under-strand and \emph{reversed} otherwise. We have:
		\begin{enumerate}
			\item Self-crossings of $L_{i}$ in $(Y, \Sigma_{i+1})$ and $(\partial X_{i}, \Sigma_{i+1})$ are preserved; self-crossings of $L_i$ in $(Y, \Sigma_i)$ and $(\partial X_i, \Sigma_i)$ are reversed.
			\item Self-crossings of $\tau_i$ in $(\partial X_i, \Sigma_i)$ and $(\partial X_{i-1}, \Sigma_C)$ are preserved; self-crossings of $\tau_i$ in $(\partial X_i, \Sigma_C)$ and $\partial X_{i-1}, \Sigma_i)$ are reversed.
		\end{enumerate}
\end{enumerate}

\begin{eg}\label{lefthandedness} Figure \ref{trefoil_B4_diagram_figure} is a valid pseudo-shadow diagram of a neatly embedded surface $\mathcal K$ in $B^4$. The boundary link $L_1 \cup L_2 \cup L_3 = \partial \mathcal K \subset S^3$ is easily seen to arise from a knot diagram with three crossings. However, determining whether it is an unknot, a left-handed trefoil, or right-handed trefoil requires some care.

	We observe that the self-crossings are: $L_3$ in $\Sigma_1$, $L_1$ in $\Sigma_2$, and $L_2$ in $\Sigma_3$. It follows that each of the three self-crossings are \emph{preserved}. Since every crossing is a negative crossing, it follows that the boundary knot is a left-handed trefoil.
\end{eg}

\begin{figure}
	\noindent\hspace{0px}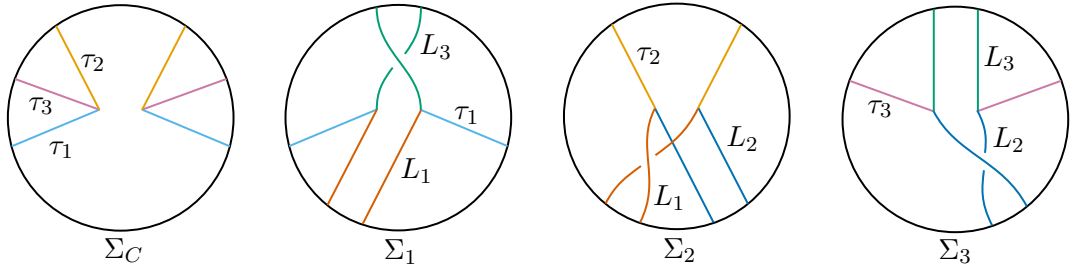
	\caption{A pseudo-shadow diagram of a genus 1 surface in $B^4$ with boundary a trefoil knot. The boundary knot is given by $L_1, L_2,$ and $L_3$.}
\label{trefoil_B4_diagram_figure}

\end{figure}

\subsection{Pseudo-shadow diagram examples and calculations}

In this subsection we provide several examples of pseudo-shadow diagrams, and show how to compute some invariants such as orientability (of the embedded surface), the intrinsic topology of the surface, and the homology class represented by the surface in the ambient 4-manifold.

\begin{prop}A surface realising a pseudo-shadow diagram is orientable if and only if the graph $\bigcup_i (L_i \cup \tau_i) \subset (\bigcup_i \Sigma_i) \cup \Sigma_C$ can be consistently labelled such that every vertex is either a source or a sink. (In this context the \emph{vertices} are the bridge points and the intersections of the $L_i$ and $\tau_i$ with the binding.)
\end{prop}
\begin{proof}This follows directly from the orientation convention of Subsection \ref{diagram_orientation}.
\end{proof}
\begin{eg}Figure \ref{Moebius_CP2_diagram_figure} is a pseudo-shadow diagram of a M\"obius strip in $\CP^2 - B^4$. To see that the surface represented by the diagram is non-orientable, observe that the vertices $b_1, b_2, b_3$ lie in a 3-cycle $\gamma$ in the graph $\bigcup_i (L_i \cup \tau_i)$, and therefore cannot be consistently labelled as sources or sinks. (For such a labelling to exist, any cycle must have even length.)
\end{eg}

\begin{figure}	
	\noindent\hspace{0px}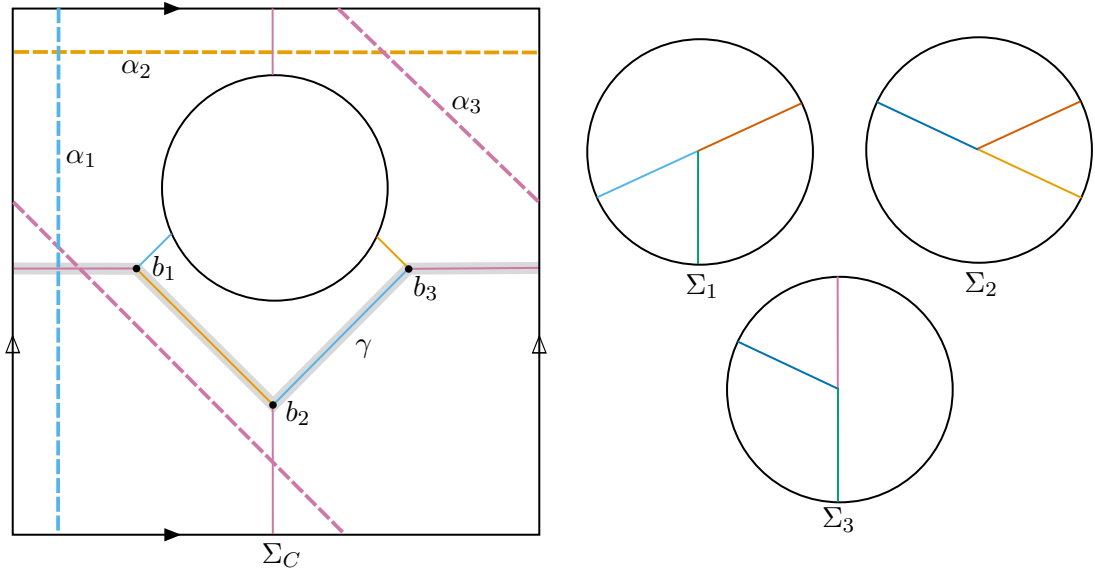
	\caption{A pseudo-shadow diagram of a M\"obius strip in pseudo-bridge position in $\CP^2 - B^4$.}
\label{Moebius_CP2_diagram_figure}
\end{figure}

\begin{prop}The Euler characteristic of a surface realising a pseudo-shadow diagram $(\mathcal D, \tau_i,$ $L_i)$ is given by $F - |\mathcal{B}|/2$, where $F$ is the total number of components in each of the unlinks $L_i \cup \tau_i \cup \tau_{i+1}$ in $\partial X_i$, and $\mathcal B$ consists of all bridge points.
\end{prop}
\begin{proof}This follows from the usual formula of the Euler characteristic of a cellular complex. The number of 0-cells is $|\mathcal {B}|$, and the number of 2-cells is $F$. Since bridge-points are 3-valent in a pseudo-shadow diagram, the number of 1-cells is $3|\mathcal{B}|/2$. The alternating sum of these quantities gives the desired result.
\end{proof}

\begin{eg}\label{LHT_example_topology}Figure \ref{LHT_CP2_diagram_figure} is a valid pseudo-shadow diagram, and represents a surface in $\CP^2-B^4$ with boundary a left-handed trefoil in $S^3$. To see the trefoil, one can proceed as in Example \ref{lefthandedness}. The links $L_i \cup \tau_i \cup \tau_{i+1}$ each have two components, so $F=6$ in the Euler characteristic formula. On the other hand, there are 10 bridge points, $\chi(\mathcal K) = 6 - 5 = 1$. Since the surface has one boundary component, we know that it has genus $0$. We conclude that the pseudo-shadow diagram represents a disk in $\CP^2-B^4$ with boundary a trefoil. (Figuring out exactly \emph{which} disk up to homology takes some more work.)
\end{eg}

\begin{figure}
	\noindent\hspace{0px}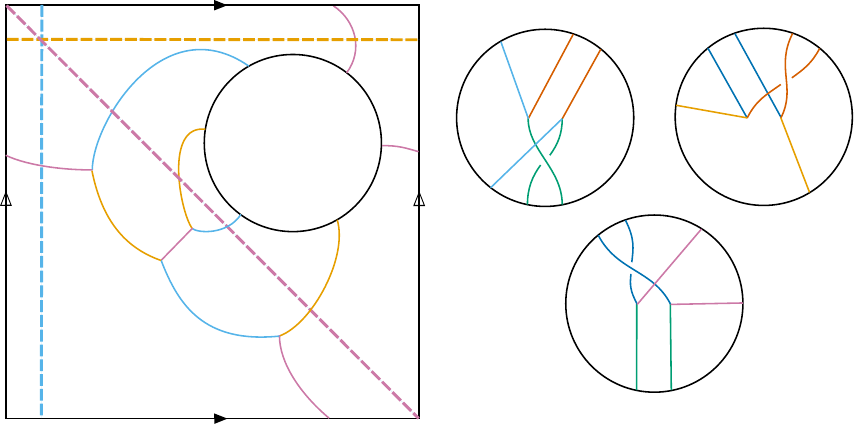
	\caption{A slice disk of a left-handed trefoil in $\CP^2$, representing $2[-\CP^1] \in H_2(\CP^2)$.}
\label{LHT_CP2_diagram_figure}
\end{figure}

Finally we describe how the homology class $[\mathcal K]$ in $H_2(X)$ or $H_2(X, \partial X)$ may be deduced from a pseudo-shadow diagram. The idea is to use Lefschetz duality to convert the problem into one of computing intersections of $[\mathcal K]$ with other homology classes in $H_2(X)$ or $H_2(X, \partial X)$. These intersections can then be computed diagrammatically.

The details are as follows. We first define the ``intersection form'' 
$$Q_X: H_2(X)\otimes H_2(X, \partial X) \to \Z$$
by $(\alpha, \beta) \mapsto \alpha^*(\beta))$ where $\alpha^* \in H^2(X, \partial X)$ is the lift of $\alpha$ to $H^2(X, \partial X)$ via the cap product with the fundamental class, i.e.
$$\alpha = \alpha^* \frown [X].$$

Next we show that this intersection form is unimodular provided $H_1(X,\partial X)$ is torsion-free. By the universal coefficient theorem, $\Hom(H_2(X, \partial X), \Z)$ is isomorphic to $H^2(X, \partial X)$. Specifically, this is because $\text{Ext}^1(H_1(X, \partial X))$ vanishes when $H_1(X, \partial X)$ is torsion-free. Applying this isomorphism, the intersection form can be expressed as
$$Q_X' : H_2(X) \to H^2(X, \partial X),$$
where the map is exactly that arising in Lefschetz duality. Therefore $Q_X'$ is invertible, and equivalently $Q_X$ is unimodular.

When $Q_X$ is unimodular, we can use intersection data to deduce homology classes. Suppose $e_1,\ldots, e_n$ are generators of $H_2(X, \partial X)$. Then knowing each of the
$$Q_X([\mathcal K], e_1),\ldots, Q_X([\mathcal K], e_n) \in \Z$$
is enough to determine $[\mathcal K]$. The final step is to express these intersections diagrammatically. 

Suppose $\mathcal K, E_1, \ldots, E_n$ are surfaces in pseudo-bridge position in $X$. Then the intersections $Q_X([\mathcal K], [E_j])$ are the sums of the intersections in each $X_i$. In each $X_i$, the intersections of disk tangles are equivalently the linking numbers of their boundaries. Therefore computing the intersections reduces to computing linking numbers of links in pseudo-shadow diagrams. This process is described in detail in the following example.

\begin{eg}\label{computing_intersection}Figure \ref{LHT_CP2_diagram_figure} is a pseudo-shadow diagram of a disk $\mathcal K$ in $\CP^2 - B^4$ with boundary a left-handed trefoil, introduced in Example \ref{LHT_example_topology}. Next we determine the homology class $[\mathcal K] \in H_2(\CP^2 - B^4)$ via the procedure described above. We know that $H_2(\CP^2 - B^4) \cong H_2(\CP^2-B^4, S^3)$ is generated by $\CP^1$, so our problem reduces to computing linking numbers of links in a pseudo-shadow diagram of $\CP^1$ with those in Figure \ref{LHT_CP2_diagram_figure}. We now describe, in detail, how to carry out this calculation in the sector $X_2$. (The procedure is similar for sectors $X_1$ and $X_3$.)

\begin{figure}
	\noindent\hspace{0px}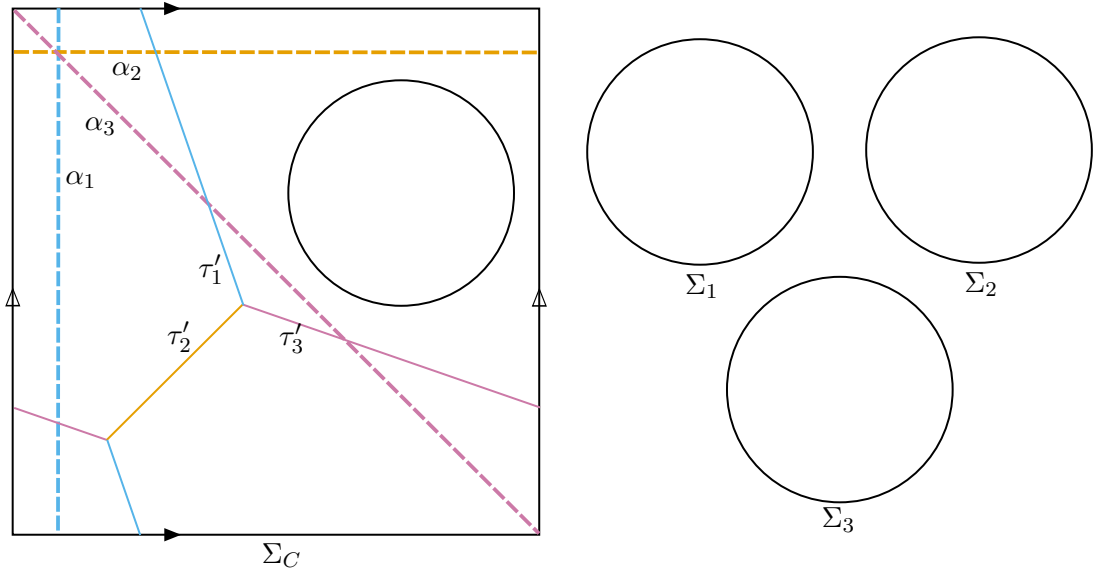
	\caption{A pseudo-shadow diagram of $\CP^1$ in $\CP^2 - B^4$. Note that the surface has empty boundary. Referenced in Example \ref{computing_intersection}.}
\label{CP1_CP2_diagram_figure}
\end{figure}

\begin{figure}
	\noindent\hspace{0px}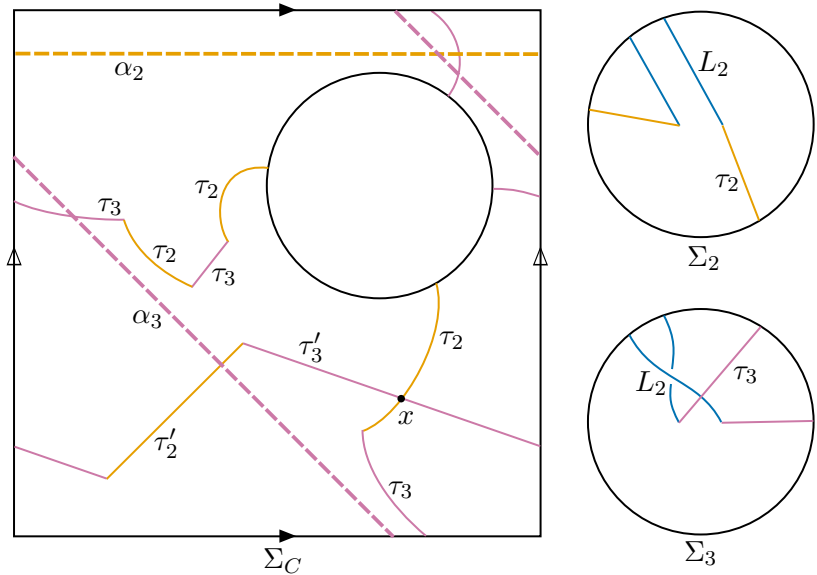
	\caption{Restrictions of pseudo-shadow diagrams for $\CP^1$ and a disk $\mathcal K$ to $X_2$, referenced in Example \ref{computing_intersection}. This is the first step in computing the intersection of $\CP^1$ and $\mathcal K$ (restricted to $X_2$).}
\label{CP1_LHT_intersection}
\end{figure}

\begin{figure}
	\noindent\hspace{0px}%% Creator: Inkscape 1.2.2 (1:1.2.2+202305151915+b0a8486541), www.inkscape.org
%% Accompanies image file 'CP1_LHT_in_CP2_step1.pdf' (pdf, eps, ps)
\begingroup%
  \makeatletter%
  \providecommand\color[2][]{%
    \errmessage{(Inkscape) Color is used for the text in Inkscape, but the package 'color.sty' is not loaded}%
    \renewcommand\color[2][]{}%
  }%
  \providecommand\transparent[1]{%
    \errmessage{(Inkscape) Transparency is used (non-zero) for the text in Inkscape, but the package 'transparent.sty' is not loaded}%
    \renewcommand\transparent[1]{}%
  }%
  \providecommand\rotatebox[2]{#2}%
  \newcommand*\fsize{\dimexpr\f@size pt\relax}%
  \newcommand*\lineheight[1]{\fontsize{\fsize}{#1\fsize}\selectfont}%
  \ifx\svgwidth\undefined%
    \setlength{\unitlength}{396.92427902bp}%
    \ifx\svgscale\undefined%
      \relax%
    \else%
      \setlength{\unitlength}{\unitlength * \real{\svgscale}}%
    \fi%
  \else%
    \setlength{\unitlength}{\svgwidth}%
  \fi%
  \global\let\svgwidth\undefined%
  \global\let\svgscale\undefined%
  \makeatother%
  \begin{picture}(1,0.53353664)%
    \lineheight{1}%
    \setlength\tabcolsep{0pt}%
    \put(0,0){\includegraphics[width=\unitlength,page=1]{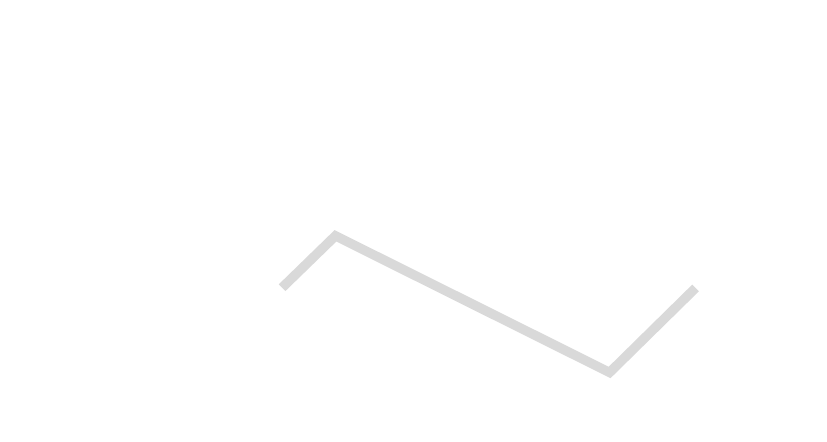}}%
    \put(0.38598615,0.07288021){\makebox(0,0)[lt]{\lineheight{1.25}\smash{\begin{tabular}[t]{l}$\Sigma_C$\end{tabular}}}}%
    \put(0.87532948,0.284542){\makebox(0,0)[lt]{\lineheight{1.25}\smash{\begin{tabular}[t]{l}$\Sigma_2$\end{tabular}}}}%
    \put(0.09762117,0.00418256){\makebox(0,0)[lt]{\lineheight{1.25}\smash{\begin{tabular}[t]{l}$\Sigma_3$\end{tabular}}}}%
    \put(0,0){\includegraphics[width=\unitlength,page=2]{CP1_LHT_in_CP2_step1.pdf}}%
    \put(0.21138828,0.48469251){\makebox(0,0)[lt]{\lineheight{1.25}\smash{\begin{tabular}[t]{l}$\alpha_2$\end{tabular}}}}%
    \put(0.19417289,0.37044632){\makebox(0,0)[lt]{\lineheight{1.25}\smash{\begin{tabular}[t]{l}$\alpha_3$\end{tabular}}}}%
    \put(0.58975507,0.12060654){\makebox(0,0)[lt]{\lineheight{1.25}\smash{\begin{tabular}[t]{l}$\gamma'$\end{tabular}}}}%
    \put(0,0){\includegraphics[width=\unitlength,page=3]{CP1_LHT_in_CP2_step1.pdf}}%
  \end{picture}%
\endgroup%

	\caption{Cutting and pasting $\Sigma_C$ (from Figure \ref{CP1_LHT_intersection}) so that the curves $\alpha_2$ and $\alpha_3$ become standard in the torus. The curve $\gamma'$ has been highlighted to differentiate it from $\gamma$. Referenced in Example \ref{computing_intersection}.}
\label{CP1_LHT_intersection_step_1}
\end{figure}

\begin{figure}
	\noindent\hspace{0px}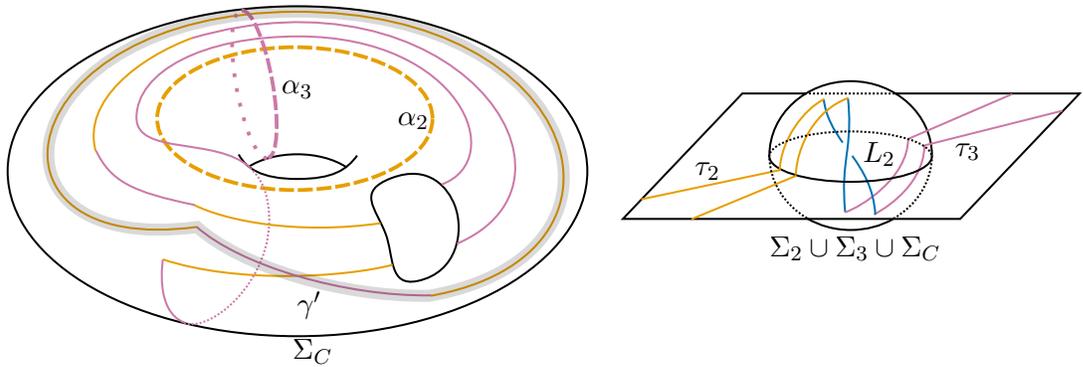
	\caption{Left: identifying the ``edges" of $\Sigma_C$. Right: the local structure of $\Sigma_2 \cup \Sigma_3 \cup \Sigma_C$. Referenced in Example \ref{computing_intersection}.}
\label{CP1_LHT_intersection_step_2}
\end{figure}

\begin{figure}
	\noindent\hspace{0px}%% Creator: Inkscape 1.2.2 (1:1.2.2+202305151915+b0a8486541), www.inkscape.org
%% Accompanies image file 'CP1_LHT_in_CP2_step3.pdf' (pdf, eps, ps)
\begingroup%
  \makeatletter%
  \providecommand\color[2][]{%
    \errmessage{(Inkscape) Color is used for the text in Inkscape, but the package 'color.sty' is not loaded}%
    \renewcommand\color[2][]{}%
  }%
  \providecommand\transparent[1]{%
    \errmessage{(Inkscape) Transparency is used (non-zero) for the text in Inkscape, but the package 'transparent.sty' is not loaded}%
    \renewcommand\transparent[1]{}%
  }%
  \providecommand\rotatebox[2]{#2}%
  \newcommand*\fsize{\dimexpr\f@size pt\relax}%
  \newcommand*\lineheight[1]{\fontsize{\fsize}{#1\fsize}\selectfont}%
  \ifx\svgwidth\undefined%
    \setlength{\unitlength}{238.02471979bp}%
    \ifx\svgscale\undefined%
      \relax%
    \else%
      \setlength{\unitlength}{\unitlength * \real{\svgscale}}%
    \fi%
  \else%
    \setlength{\unitlength}{\svgwidth}%
  \fi%
  \global\let\svgwidth\undefined%
  \global\let\svgscale\undefined%
  \makeatother%
  \begin{picture}(1,0.51321005)%
    \lineheight{1}%
    \setlength\tabcolsep{0pt}%
    \put(0,0){\includegraphics[width=\unitlength,page=1]{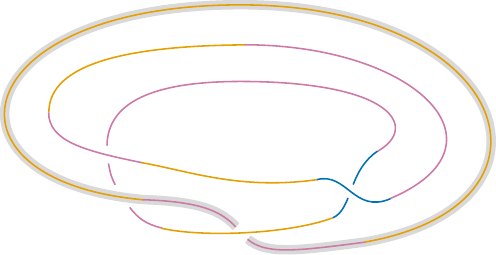}}%
    \put(0.2797233,0.27618801){\makebox(0,0)[lt]{\lineheight{1.25}\smash{\begin{tabular}[t]{l}$\gamma_1$\end{tabular}}}}%
    \put(0.1338383,0.31936263){\makebox(0,0)[lt]{\lineheight{1.25}\smash{\begin{tabular}[t]{l}$\gamma_2$\end{tabular}}}}%
    \put(0.04436314,0.42385384){\makebox(0,0)[lt]{\lineheight{1.25}\smash{\begin{tabular}[t]{l}$\gamma'$\end{tabular}}}}%
  \end{picture}%
\endgroup%

	\caption{Lifting $\gamma$ and $\gamma'$ to $\R^3 = S^3 - \{\text{pt}\}$. Referenced in Example \ref{computing_intersection}.}
\label{CP1_LHT_intersection_step_3}
\end{figure}
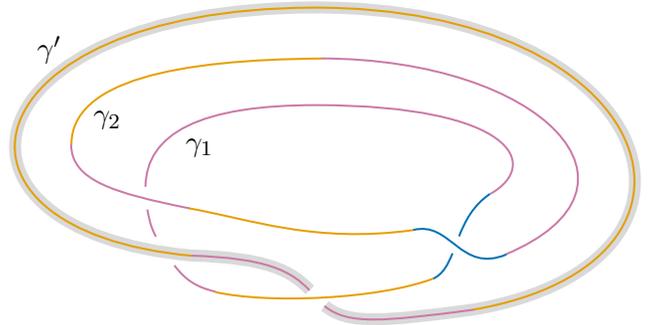
	\begin{enumerate}
		\item Figure \ref{CP1_CP2_diagram_figure} shows a pseudo-shadow diagram of $\CP^1$. Restricting this diagram to $(\Sigma_C, \Sigma_2, \Sigma_3,$ $\alpha_2, \alpha_3, \tau_2', \tau_3')$ produces a triple Heegaard diagram of $S^3$, equipped with the data of $\tau_2'$ and $\tau_3'$. Similarly, the pseudo-shadow diagram of Figure \ref{LHT_CP2_diagram_figure} can be restricted to a triple Heegaard diagram of $S^3$, equipped with the data of $L_2, \tau_2$, and $\tau_3$. The realisation map lifts both $\gamma = L_2\cup \tau_2 \cup \tau_3$ and $\gamma' = \tau_2' \cup \tau_3'$ to unlinks in $S^3$. Our goal is to compute the linking number $lk(\gamma, \gamma')$. Figure \ref{CP1_LHT_intersection} shows each of the aforementioned triple Heegaard diagrams overlayed as a single diagram. There is an apparent crossing, denoted $x$, between $\gamma$ and $\gamma'$. Subsequent diagrammatic manipulations will determine how this crossing impacts $lk(\gamma, \gamma')$.
		\item To lift the curves $\gamma$ and $\gamma'$ to $S^3$, we must first build the underlying $S^3$ from the triple Heegaard diagram in Figure \ref{CP1_LHT_intersection}. To this end, we ``cut and paste" the representation of $\Sigma_C$ so that the curves $\alpha_2$ and $\alpha_3$ are standard axes on the torus. The result is depicted in Figure \ref{CP1_LHT_intersection_step_1}.
		\item Next, we glue the various pieces of the triple Heegaard diagram together. Figure \ref{CP1_LHT_intersection_step_2} shows the result of identifying the ``edges" of $\Sigma_C$ on the left, and the result of gluing each of $\Sigma_2, \Sigma_3$, and $\Sigma_C$ along their common boundary on the right. Note also that $L_2$ has been lifted from $\Sigma_2 \cup \Sigma_3$ to the sector $Y_2$. This step is subtle: by the orientation convention, $\tau_2$ lies above $L_2$, and $L_2$ lies above $\tau_3$. Indeed, they are depicted as such in the right hand diagram, and this is consistent with the orientations induced by the conventions of Subsection \ref{diagram_orientation}. Moreover, the convention requires that the self-crossing of $L_2$ is \emph{reversed} when lifting to $Y_2$.
		\item The arcs $\tau_2'$ lift to the sector of the trisection of $S^3$ lying ``outside" $\Sigma_C$, and the arcs $\tau_3'$ lift to the sector ``inside" $\Sigma_C$. The final result of lifting all arcs to the sectors $H_2, H_3, Y_2$ is shown in Figure \ref{CP1_LHT_intersection_step_3}. From here, it can be seen that $lk(\gamma, \gamma') = \pm 1$, where $\gamma$ is an unlink with two components, $\gamma_1$ and $\gamma_2$. The sign of the linking number depends on how the surface $\mathcal K$ is oriented.
	\end{enumerate}

	In summary, we have established that $lk(\gamma, \gamma') = lk(L_2\cup\tau_2\cup\tau_3, \tau_2'\cup\tau_3') = \pm 1$. Further, this linking number is equal to the signed intersection of $\mathcal K \cap X_2$ with $\CP^1 \cap X_2$. Repeating this process in $X_1$ and $X_3$, those linking numbers are found to be 0 and $\pm 1$ respectively. Therefore $Q_{\CP^2}([\mathcal K], [\CP^1]) = 0 \pm 1 \pm 1 = \pm2$. Since $Q_{\CP^2} = (1)$, we deduce that $[\mathcal K] = 2H$ where $H = \pm [\CP^1]$ is a generator of $H_2(\CP^2)$. This is consistent with the slice disk of a left handed trefoil in $\CP^2$ described in Example 2.4 of \cite{ManMarPic}. Fixing an orientation of the left handed trefoil would fix the sign of $[\mathcal K]$. 
\end{eg}

\bibliographystyle{plain}
\bibliography{trisections.bib}
\end{document}